\documentclass{article}

\usepackage{authblk}

\usepackage[utf8]{inputenc}
\usepackage[T1]{fontenc}
\usepackage{textcomp}
\usepackage{amsmath,amssymb,amsthm}
\usepackage{lmodern}
\usepackage[a4paper]{geometry}
\usepackage{xcolor, pict2e}
\usepackage{microtype}
\usepackage{caption}
\usepackage{listings}
\usepackage{multicol}
\usepackage{moreverb}
\usepackage{hyperref}
\hypersetup{pdfstartview=XYZ}
\usepackage{wrapfig}
\usepackage[sans]{dsfont}

\usepackage{hyperref}
\hypersetup{
    colorlinks=true,
    linkcolor=blue,
    citecolor=magenta,
    urlcolor=blue,
    pdfborder={0 0 0}
}

\usepackage[font=sf, labelfont={sf,bf}, margin=1cm]{caption}

\usepackage{float}
\usepackage[caption = false]{subfig}
\usepackage{graphicx}

\newtheorem*{theorem*}{Theorem}
\newtheorem{theorem}{Theorem}[section]
\newtheorem{proposition}{Proposition}[section]
\newtheorem{lemma}{Lemma}[section]
\newtheorem{corollary}{Corollary}[section]
\newtheorem{notation}{Notation}[section]

\newtheorem{theoremm}{Theorem}

\newtheorem{lemmaa}{Lemma}

\newtheorem{conjecture}{Conjecture}

\theoremstyle{remark}
\newtheorem{remark}{Remark}[section]

\newcommand{\R}{\mathbb{R}}
\newcommand{\N}{\mathbb{N}}
\newcommand{\Z}{\mathbb{Z}}
\newcommand{\Q}{\mathbb{Q}}
\newcommand{\C}{\mathbb{C}}

\newcommand{\Bc}{\mathcal{B}}
\newcommand{\Cc}{\mathcal{C}}

\newcommand{\Fc}{\mathcal{F}}

\newcommand{\Mc}{\mathcal{M}}
\newcommand{\Nc}{\mathcal{N}}

\newcommand{\Expect}[1]{\mathbb{E} \left[ #1 \right] }
\newcommand{\EXPECT}[2]{\mathbb{E}_{#1} \left[ #2 \right] }
\newcommand{\prob}{\mathbb{P}}
\newcommand{\Prob}[1]{\mathbb{P} \left( #1 \right) }
\newcommand{\PROB}[2]{\mathbb{P}_{#1} \left( #2 \right) }

\renewcommand{\P}{\mathbb{P}}
\newcommand{\E}{\mathbb{E}}
\newcommand{\Pb}{\mathds{P}}
\newcommand{\Eb}{\mathds{E}}

\newcommand{\abs}[1]{\left\vert #1 \right\vert}

\newcommand{\floor}[1]{\left\lfloor #1 \right\rfloor}

\newcommand{\indic}[1]{ \mathbf{1}_{ \left\{ #1 \right\} } }
\newcommand{\eps}{\varepsilon}

\DeclareMathOperator{\CR}{CR}

\renewcommand{\d}{\mathrm{d}}

\makeatletter
\newcommand*\bigcdot{\mathpalette\bigcdot@{.5}}
\newcommand*\bigcdot@[2]{\mathbin{\vcenter{\hbox{\scalebox{#2}{$\m@th#1\bullet$}}}}}
\makeatother

\usepackage{enumitem}

\usepackage{accents}
\newlength{\dhatheight}
\newcommand{\dhat}[1]{%
    \settoheight{\dhatheight}{\ensuremath{\hat{#1}}}%
    \addtolength{\dhatheight}{-0.35ex}%
    \hat{\vphantom{\rule{1pt}{\dhatheight}}%
    \smash{\hat{#1}}}}

\title{Critical Brownian multiplicative chaos}
\author{Antoine Jego\thanks{On leave from the University of Cambridge, partly supported by the EPSRC grant
EP/L016516/1 for the University of Cambridge Centre for Doctoral Training, the Cambridge Centre for Analysis. E-mail address: \href{mailto:apfj2@cam.ac.uk}{apfj2@cam.ac.uk}}
}
\affil{University of Vienna}
\date {}

\numberwithin{equation}{section}

\begin{document}

\renewcommand{\theparagraph}{\thesubsection.\arabic{paragraph}} 

\maketitle

\begin{abstract}
Brownian multiplicative chaos measures, introduced in \cite{jegoGMC,AidekonHuShi2018,bass1994}, are random Borel measures that can be formally defined by exponentiating $\gamma$ times the square root of the local times of planar Brownian motion. So far, only the subcritical measures where the parameter $\gamma$ is less than 2 were studied. This article considers the critical case where $\gamma =2$, using three different approximation procedures which all lead to the same universal measure. On the one hand, we exponentiate the square root of the local times of small circles and show convergence in the Seneta--Heyde normalisation as well as in the derivative martingale normalisation. On the other hand, we construct the critical measure as a limit of subcritical measures. This is the first example of a non-Gaussian critical multiplicative chaos.

We are inspired by methods coming from critical Gaussian multiplicative chaos, but there are essential differences, the main one being the lack of Gaussianity which prevents the use of Kahane's inequality and hence a priori controls. Instead, a continuity lemma is proved which makes it possible to use tools from stochastic calculus as an effective substitute.
\end{abstract}

\tableofcontents

\section{Introduction}

Thick points of planar Brownian motion/random walk are points that have been visited unusually often by the trajectory. The study of these points has a long history going back to the famous conjecture of Erd\H{o}s and Taylor \cite{erdos_taylor1960} on the leading order of the number of times a planar simple random walk visits the most visited site during the first $n$ steps. Since then, the understanding of these thick points has considerably improved. On the random walk side, \cite{dembo2001} settled Erd\H{o}s--Taylor conjecture and computed the number of thick points at the level of exponent, for random walk having symmetric increments with finite moments of all order. \cite{rosen2006,BassRosen2007}, and more recently \cite{jego2020}, streamlined the proof and extended these results to a wide class of planar random walk. On the Brownian motion side, \cite{bass1994} constructed random measures supported on the set of thick points. Their results concern only a partial range $\{a \in (0,1/2)\}$ of the thickness parameter $a$\footnote{$a$ is related to the parameter $\gamma$ in Gaussian multiplicative chaos theory by $a=\gamma^2/2$, so $a < 1/2$ corresponds to $\gamma < 1$.}. \cite{AidekonHuShi2018} and \cite{jegoGMC} extended simultaneously the results of \cite{bass1994} by building these random measures for the whole subcritical range $\{a \in (0,2)\}$. \cite{jegoRW} gave an axiomatic characterisation of these measures and showed that they describe the scaling limit of thick points of planar simple random walk for any fixed $a < 2$. All these aforementioned works are subcritical results. The aim of this paper is to extend the theory to the critical point $a=2$ by constructing a random measure supported by the thickest points of a planar Brownian trajectory. This enables us to formulate a precise conjecture on the convergence in distribution of the supremum of local times of planar random walk.

Our construction is inspired by Gaussian multiplicative chaos theory (GMC), i.e. the study of random measures formally defined as the exponential of $\gamma$ times a log-correlated Gaussian field, such as the two-dimensional Gaussian free field (GFF), where $\gamma \geq 0$ is a parameter. Since such a field is not defined pointwise but is rather a random generalised function, making sense of such a measure requires some nontrivial work. The theory was introduced by Kahane \cite{kahane} and has expanded significantly in recent years. By now it is relatively well understood, at least in the subcritical case where $\gamma <\sqrt{2d}$ \cite{RobertVargas2010, DuplantierSheffieldGMC, RhodesVargasGMC, ShamovGMC, berestycki2017} and even in the critical case $\gamma = \sqrt{2d}$ \cite{DRSV2014b, DRSV2014a, JunnilaSaksman2017, JSW19, powell2018, aru19, aru17}.
In this article, the log-correlated field we have in mind is the (square root of) the local time process of a planar Brownian motion, appropriately stopped. The main interest of our construction from GMC point of view is that this field is non-Gaussian, so that our results give the first example of a critical chaos for a truly non-Gaussian field.\footnote{We point out the work of \cite{SaksmanWebb2016} on the Riemann zeta function where the limiting field is Gaussian, but not the approximation. See also \cite{FyodorovKeating2014, Webb2015, NikulaSaksmanWebb2018, Lambert2018, BerestyckiWebWong2018, Junnila18} for subcritical results.}

\subsection{Main results}

Let $\prob_x$ be the law under which $(B_t)_{t \geq 0}$ is a planar Brownian motion starting from $x \in \R^2$. Let $D \subset \R^2$ be an open bounded simply connected domain, $x_0 \in D$ be a starting point and $\tau$ be the first exit time of $D$:
\[
\tau := \inf \{ t \geq 0: B_t \notin D \}.
\]
For all $x \in \R^2, t>0, \eps >0,$ define the local time $L_{x,\eps}(t)$ of $\left( \abs{B_s - x}, s \geq 0 \right)$ at $\eps$ up to time $t$ (here $\abs{\cdot}$ stands for the Euclidean norm):
\begin{equation}
\label{eq:def_local_time_BM}
L_{x,\eps}(t) := \lim_{r \rightarrow 0^+} \frac{1}{2 r} \int_0^t \indic{ \eps - r \leq \abs{B_s - x} \leq \eps + r} ds.
\end{equation}
\cite[Proposition 1.1]{jegoGMC} shows that we can make sense of the local times $L_{x,\eps}(\tau)$ simultaneously for all $x$ and $\eps$ with the convention that $L_{x,\eps}(\tau) = 0$ if the circle $\partial D(x,\eps)$ is not entirely included in $D$. We can thus define for any thickness parameter $\gamma \in (0,2]$ and any Borel set $A$,
\begin{equation}
\label{eq:def_measure_subcritical_normalisation}
m_\eps^\gamma(A) := \sqrt{|\log \eps|} \eps^{\gamma^2/2} \int_A e^{\gamma \sqrt{\frac{1}{\eps} L_{x,\eps}(\tau)} } dx.
\end{equation}
We recall:

\begin{theoremm}[Theorem 1.1 of \cite{jegoGMC}]\label{thmm:subcritical_measures}
Let $\gamma \in (0,2)$. The sequence of random measures $m_\eps^\gamma$ converges as $\eps \to 0$ in probability for the topology of weak convergence on $D$ towards a Borel measure $m^\gamma$ called Brownian multiplicative chaos.
\end{theoremm}

See \cite{AidekonHuShi2018} for a different construction of the subcritical Brownian multiplicative chaos, as well as \cite{bass1994} for partial results. See also \cite{jegoRW} for more properties on these measures.

Our first result towards extending the theory to the critical point $\gamma =2$ is the fact that the subcritical normalisation yields a vanishing measure in the critical case:

\begin{proposition}\label{prop:subcritical_vanishes}
$m_\eps^{\gamma =2} (D)$ converges in $\P_{x_0}$-probability to zero.
\end{proposition}

To obtain a non-trivial object we thus need to renormalise the measure slightly differently. 
Firstly, we consider the Seneta--Heyde normalisation:
for all Borel set $A$, define
\begin{equation}
\label{eq:def_measure_SH}
m_\eps(A) := \sqrt{|\log \eps|} m_\eps^{\gamma=2}(A)
= |\log \eps| \eps^2 \int_A e^{2 \sqrt{ \frac{1}{\eps} L_{x,\eps}(\tau) }} dx.
\end{equation}
Secondly, we consider the derivative martingale normalisation which formally corresponds to (minus) the derivative of $m_\eps^\gamma$ with respect to $\gamma$ evaluated at $\gamma = 2$: for all Borel set $A$, define
\begin{equation}
\label{eq:def_measure_derivative}
\mu_\eps(A) := - \frac{\d m_\eps^\gamma(A)}{\d \gamma} \Bigr|_{\gamma=2} = \sqrt{|\log \eps|} \eps^2 \int_A \left( - \sqrt{\frac{1}{\eps} L_{x,\eps}(\tau)} + 2 \log \frac{1}{\eps} \right) e^{2 \sqrt{\frac{1}{\eps} L_{x,\eps}(\tau)} } dx.
\end{equation}

\begin{theorem}\label{th:convergence_measures}
The sequences of random positive measures $(m_\eps)_{\eps > 0}$ and random signed measures $(\mu_\eps)_{\eps > 0}$ converge in $\P_{x_0}$-probability for the topology of weak convergence towards random Borel measures $m$ and $\mu$. Moreover, the limiting measures satisfy:
\begin{enumerate}
\item
\label{thmthm:SH=derivative}
$m = \sqrt{\frac{2}{\pi}} \mu$ $\P_{x_0}$-a.s. In particular, $\mu$ is a random positive measure.
\item
\label{thmthm:nondegeneracy}
Nondegeneracy:
$\mu(D) \in (0,\infty)$ $\P_{x_0}$-a.s.
\item
\label{thmthm:first_moment}
First moment: $\EXPECT{x_0}{\mu(D)} = \infty$.
\item
\label{thmthm:nonatomicity}
Nonatomicity:
$\P_{x_0}$-a.s. simultaneously for all $x \in D$, $\mu(\{x\}) = 0$.
\end{enumerate}
\end{theorem}

Our next main result is the construction of critical Brownian multiplicative chaos as a limit of subcritical measures. Before stating such a result, we need to ensure that we can make sense of the subcritical measures simultaneously for all $\gamma \in (0,2)$.

\begin{proposition}
\label{prop:process_subcritical}
Let $\Mc$ be the set of finite Borel measures on $\R^2$.
The process $\gamma \in (0,2) \mapsto m^\gamma \in \Mc$ of subcritical Brownian multiplicative chaos measures possesses a modification such that for all continuous nonnegative function $f$, $\gamma \in (0,2) \mapsto \int f dm^\gamma \in \R$ is lower semi-continuous.
\end{proposition}

\begin{theorem}
\label{th:critical_from_subcritical}
Let $ \gamma \in (0,2) \mapsto m^\gamma$ be the process of subcritical Brownian multiplicative chaos measures from Proposition \ref{prop:process_subcritical}. Then, $(2-\gamma)^{-1} m^\gamma$ converges towards $2 \mu$ as $\gamma \to 2^{-}$ in probability for the topology of weak convergence of measures.
\end{theorem}

\begin{remark}
In Proposition \ref{prop:process_subcritical}, we do not obtain continuity of the process in $\gamma$. The main difficulty here is that, in order to use Kolmogorov's continuity theorem, one has to consider moments of order larger than 1. When $\gamma \geq \sqrt{2}$, the second moment blows up and we have to deal with non-integer moments which are difficult to estimate without the use of Kahane's convexity inequalities but this tool is restricted to the Gaussian setting.
To bypass this difficulty, we apply Kolmogorov's criterion to versions of the measures that are restricted to specific `good' events allowing us to make $L^2$-computations. The drawback is that it does not yield continuity of the process but only lower semi-continuity. See Appendix \ref{sec:app_continuity}.
\end{remark}

We mention that the construction of the critical measure as a limit of subcritical measures is only partially known in the GMC realm. Such a result has first been proved to hold in the specific case of the two-dimensional GFF \cite{aru19} exploiting on the one hand the construction of Liouville measures as multiplicative cascades \cite{aru17} and on the other hand the strategy of Madaule \cite{madaule16} who proves a result analogous to Theorem \ref{th:critical_from_subcritical} in the case of multiplicative cascades/branching random walk. It has then been extended to a wide class of log-correlated Gaussian fields in dimension two by comparing them to the GFF \cite{JSW19}. In other dimensions, a natural reference log-correlated Gaussian field is lacking and the result is so far unknown. We believe that the approach we use in this paper to prove Theorem \ref{th:critical_from_subcritical} can be adapted in order to show that critical GMC measures can be built from their subcritical versions in \emph{any} dimension.\footnote{After the first version of the current paper was finished, the fact that the critical GMC measure can be obtained as a limit of the subcritical measures has been established in any dimension in \cite{powell_review}.}

Theorem \ref{th:critical_from_subcritical} can be seen as exchanging the limit in $\eps$ and the derivative with respect to $\gamma$. Surprisingly, a factor of 2 pops up when one exchanges the two:
\[
\lim_{\gamma \to 2^{-}} \lim_{\eps \to 0} \frac{(m^\gamma_\eps - m^2_\eps)}{2-\gamma}
=
\lim_{\gamma \to 2^{-}} \frac{1}{2-\gamma} m^\gamma 
= 2 \lim_{\eps \to 0} \mu_\eps
= 2 \lim_{\eps \to 0} \lim_{\gamma \to 2^{-}} \frac{(m^\gamma_\eps - m^2_\eps)}{2-\gamma}.
\]
This factor of 2 is present as well in the context of GMC \cite{aru19,JSW19} and cascades \cite{madaule16}.

Theorem \ref{th:critical_from_subcritical} is important because it hints at the universal nature of the measure $\mu$, in the following sense.
First, recall that the article \cite{jegoRW} gives an axiomatic characterisation of the subcritical measures $m^\gamma$ implying their universality in the sense that different approximations yield the same limiting measures. Thus, Theorem \ref{th:critical_from_subcritical} can be seen as showing a form of universality for $\mu$ as well.
Furthermore, the subcritical measures $m^\gamma$ are known to be conformally covariant \cite{jegoGMC,AidekonHuShi2018} and Theorem \ref{th:critical_from_subcritical} allows us to extend this conformal covariance to the critical measures.

\begin{corollary}\label{cor:conformal}
Let $\phi : D \to D'$ be a conformal map between two bounded simply connected domains. Let $x_0 \in D$ and denote by $\mu^D$ and $\mu^{D'}$ the critical Brownian multiplicative chaos measures built in Theorem \ref{th:convergence_measures} for the domains $(D,x_0)$ and $(D', \phi(x_0))$ respectively. Then we have
\[
(\mu^D \circ \phi^{-1})(dx) \overset{\mathrm{law}}{=} \abs{ \phi'(\phi^{-1}(x)) }^4 \mu^{D'}(dx).
\]
\end{corollary}

\begin{proof}
Let $\gamma \in (0,2)$ and denote by $m^{\gamma,D}$ and $m^{\gamma,D'}$ the subcritical measures built in Theorem \ref{thmm:subcritical_measures} for the domains $(D,x_0)$ and $(D',\phi(x_0))$ respectively. By \cite[Corollary 1.4 \textit{(iv)}]{jegoGMC}, it is known that
\begin{equation}
\label{eq:proof_cor_conformal}
(m^{\gamma,D} \circ \phi^{-1})(dx) \overset{\mathrm{law}}{=} \abs{ \phi'(\phi^{-1}(x)) }^{2+\gamma^2/2} m^{\gamma,D'}(dx).
\end{equation}
By Theorem \ref{th:critical_from_subcritical}, we obtain the desired result by dividing both sides of the above equality by $2(2-\gamma)$ and then by letting $\gamma \to 2$.

Let us note that in \cite{jegoGMC} the conformal covariance \eqref{eq:proof_cor_conformal} of the subcritical measures is stated between domains that are assumed to have a boundary composed of a finite number of analytic curves. This extra assumption was made to match the framework of \cite{AidekonHuShi2018} but we emphasise that it is useless in our context. Proposition 6.2 of \cite{jegoGMC} only requires the domain to be bounded and simply connected. This proposition characterises the law of $m^{\gamma,D}$ together with the Brownian motion from which it has been built. The conformal covariance then follows from this proposition as it is written in Section 5 of \cite{AidekonHuShi2018}.
\end{proof}

Note that we could not hope to apply directly the approach used in the subcritical case to prove conformal covariance at criticality. Indeed, in the subcritical regime, this is based on a characterisation of the law of the couple formed by the measure together with the Brownian motion from which it has been built. This characterisation is in turn based on $L^1$ computations that are infinite at criticality (Theorem \ref{th:convergence_measures}, point \ref{thmthm:first_moment}).

\subsection{Conjecture on the supremum of local times of random walk}
\label{subsec:open_questions}

In recent years, much effort has been put in the study of the supremum of log-correlated fields, the ultimate goal being the convergence in distribution of the supremum properly centred.  In many examples, the limiting law is a Gumbel distribution randomly shifted by the log of the total mass of an associated critical chaos. 
This has been established for example in the following instances: branching random walk \cite{aidekon2013}, local times of random walk on regular trees \cite{abe2018}, cover time of binary trees \cite{CortinesLouidorSaglietti2018,DRZ2019}, discrete GFF \cite{BramsonDingZeitouni}, log-correlated Gaussian field \cite{madaule15,DingRoyZeitouni17}. See \cite{arguin17,Shi15} and \cite[Section 2]{BiskupLouidor2016a} for more references. By analogy with these results, it is natural to make the following conjecture that we present in the more natural setting of random walk.

For $x \in \Z^2$ and $N \geq 1$, let $\ell_x^N$ be the total number of times a planar simple random walk starting from the origin has visited the vertex $x$ before exiting the square $[-N,N]^2$. Define a random Borel measure $\mu_N$ on $\R^2 \times \R$ by setting for all Borel sets $A \subset \R^2$ and $T \subset \R$, 
\[
\mu_N(A \times T) := \sum_{x \in \Z^2} \indic{x/N \in A} \indic{\sqrt{\ell_x^N} - 2 \pi^{-1/2} \log N + \pi^{-1/2} \log \log N \in T}.
\]

\begin{conjecture}\label{pb:sup}
There exist constants $c_1, c_2>0$ such that $(\mu_N, N \geq 1)$ converges in distribution for the topology of vague convergence on $\R^2 \times (\R \cup \{+ \infty\})$ towards the Poisson point process
\[
\mathrm{PPP}(c_1 \mu \otimes c_2 e^{-c_2 t} dt )
\]
where $\mu$ is the critical Brownian multiplicative chaos in the domain $[-1,1]^2$ with the origin as a starting point.
In particular, for all $t \in \R$,
\[
\Prob{ \sup_{x \in \Z^2} \sqrt{\ell_x^N} \leq \frac{2}{\sqrt{\pi}} \log N - \frac{1}{\sqrt{\pi}} \log \log N + t } \xrightarrow[N \to \infty]{} \Expect{ \exp \left( -c_1 \mu( [-1,1]^2) e^{-c_2 t}  \right)  }.
\]
\end{conjecture}

The leading order term $2 \pi^{-1/2} \log N$ has been conjectured by Erd\H{o}s and Taylor \cite{erdos_taylor1960} and proven by \cite{dembo2001}. See also \cite{rosen2006,BassRosen2007,jego2020}.
We expect $-\pi^{-1/2} \log \log N$ to be the second order term since, with this choice of constant, the expectation of $\mu_N(\R^2 \times (0,\infty))$ blows up like $\log N$. Indeed, in analogy with the case of the 2D discrete GFF (see \cite{BiskupLouidor2016a}), this should be the correct way of scaling the point measure to get a nondegenerate limit.

Let us compare this conjecture with the case of the 2D discrete GFF $(\phi_N(x))_{x \in \Z^2}$, that is the centred Gaussian vector whose covariance is given by $\E[\phi_N(x) \phi_N(y)] = \EXPECT{x}{\ell_y^N}$. \cite{BramsonDingZeitouni} (see \cite{BiskupLouidor2020} for the link with Liouville measure) showed that for all $t \in \R$,
\[
\Prob{ \sup_{x \in \Z^2} \frac{1}{\sqrt{2}} \phi_N(x) \leq \frac{2}{\sqrt{\pi}} \log N - \frac{3}{4 \sqrt{\pi}} \log \log N + t } \xrightarrow[N \to \infty]{} \Expect{ \exp \left( -c_1 \mu^L( [-1,1]^2) e^{-c_2 t}  \right)  }
\]
where $c_1, c_2>0$ are some constants and $\mu^L$ is the Liouville measure in $[-1,1]^2$.
Despite strong links between local times and half of the GFF squared (see lecture notes \cite{rosen_2014} for an overview of the topic), Conjecture \ref{pb:sup} would show that the supremum of the former is slightly smaller than the supremum of the latter, enhancing subtle differences between the two fields (see \cite[Corollary 1.1]{jegoRW} and \cite[Corollary 1.1]{jegoGMC} for results in this direction).

Let us mention that \cite{jego2020} shows results analogous to Conjecture \ref{pb:sup} in dimensions larger or equal to three and that \cite{jegoRW} establishes the subcritical analogue of Conjecture \ref{pb:sup} in dimension two.
A first step towards solving Conjecture \ref{pb:sup} might be to give a characterisation of the law of critical Brownian multiplicative chaos analogous to the subcritical characterisation of \cite{jegoRW}. Since the first moment blows up, fixing the normalisation of the measure is one of the main challenges in this regard.

\subsection{Proof outline}
\label{subsec:proof_outline}

We now explain the main ideas and difficulties of the proof of Theorems \ref{th:convergence_measures} and \ref{th:critical_from_subcritical}.

We start by recalling that, as noticed in \cite{jegoGMC}, if the domain $D$ is a disc $D = D(x,\eta)$ centred at $x$, then the local times $L_{x,r}(\tau), r >0,$ exhibit the following Markovian structure: for all $\eta' \in (0,\eta)$ and all $z \in D(0,\eta) \backslash D(0,\eta')$, under $\prob_z$ and conditioned on $L_{x,\eta'}(\tau)$,
\begin{equation}\label{eq:local_times_and_Bessel_process}
\left( \sqrt{ \frac{1}{r} L_{x,r}(\tau)  }, r = \eta' e^{-s}, s \geq 0 \right)
\overset{\mathrm{law}}{=}
\left( X_s, s \geq 0 \right)
\end{equation}
with $(X_s,s \geq 0)$ being a zero-dimensional Bessel process starting from $\sqrt{ L_{x,\eta'}(\tau) / {\eta'}}$. This is an easy consequence of rotational invariance of Brownian motion and second Ray-Knight isomorphism for local times of one-dimensional Brownian motion.
In order to exploit this relation, we will very often stop the Brownian trajectory at the first exit time $\tau_{x,R}$ of the disc $D(x,R)$, $R$ being the diameter of the domain $D$.

What makes the critical case so special is that the approximating measures are not normalised by the first moment any more (otherwise we would get a vanishing measure as shown in Proposition \ref{prop:subcritical_vanishes}). We thus need to introduce good events before being able to even make $L^1$-computations.
Defining the right events and showing that they do not change the measures with high probability is one of the crucial steps of this paper that we are about to explain. We first explain the most natural events to consider and we then explain why we will actually consider different events.

\paragraph*{Naive definition of good events}

In analogy with the case of log-correlated Gaussian fields, it is natural to consider the following events to make the measures bounded in $L^1$: let $\beta >0$ be large and for all $x \in D$ and $\eps >0$, define
\[
G_\eps(x) := \left\{ \forall \delta \in [\eps, 1], \sqrt{\frac{1}{\delta} L_{x,\delta}(\tau_{x,R})} \leq 2 \log \frac{1}{\delta} + \beta \right\}.
\]
Here, we stop the Brownian path at time $\tau_{x,R}$ to be able to use \eqref{eq:local_times_and_Bessel_process}.
One would expect $\PROB{x_0}{\bigcap_{x \in D} \bigcap_{\eps>0} G_\eps(x)} \to 1$ as $\beta \to \infty$ since, by analogy with the Gaussian case (see \cite[Corollary 2.4]{powell2018} for instance),  the following should hold true:
\begin{equation}
\label{eq:intro_sup}
\sup_{x \in D} \sup_{\eps > 0} \left( \sqrt{\frac{1}{\eps} L_{x,\eps}(\tau_{x,R})} - 2 \log \frac{1}{\eps} \right) < \infty
\quad \quad \P_{x_0}-\mathrm{a.s.}
\end{equation}
Because of the lack of self-similarity and Gaussianity of our model, showing \eqref{eq:intro_sup} turns out to be far from easy (see the introduction of Section \ref{sec:adding_good_events} for more about this). We thus take a detour to justify that the introduction of the events $G_\eps(x)$ is harmless. We first control the supremum of the more regular local times of small annuli allowing us to introduce good events associated to these local times. Crucially, these good events will be enough to make the measures bounded in $L^1$.
Using \emph{repulsion estimates} associated to zero-dimensional Bessel process $X$, we will finally be able to transfer the restrictions on the local times of annuli (requiring for all $k \geq 0$, $\min_{[k,k+1]} X \leq 2k + 2 \log(k) + \beta/2$) over to restrictions on the local times of circles (requiring for all $s \geq 0$, $X_s \leq 2s + \beta$). This is the content of Section \ref{sec:adding_good_events}.

Other repulsion estimates with a similar flavour will tell us that, once we restrict ourselves to the events $G_\eps(x)$, we will be able to restrict further the measures to the good events
\[
G_\eps'(x) := \left\{ \forall \delta \in [\eps,1], \sqrt{\frac{1}{\delta} L_{x,\delta}(\tau_{x,R})} \leq 2 \log \frac{1}{\delta} + \beta - \frac{\sqrt{|\log \delta|}}{M \log(2 + |\log \delta|)^2} \right\}
\]
for some large $M>0$. This is the content of Lemma \ref{lem:bad_points}. This second layer of good event will make the measures bounded in $L^2$ (Proposition \ref{prop:bdd_L2}). We will conclude the proof by showing that the measures restricted to the second layer of good events converge in $L^2$ (Proposition \ref{prop:Cauchy_L2}).

\paragraph*{Actual definition of good events}

We now explain why we actually define different good events. This paper extensively uses the relation \eqref{eq:local_times_and_Bessel_process} between local times and zero-dimensional Bessel process. When making $L^1$-computations, we will bound from above the local times $L_{x,\eps}(\tau)$ by $L_{x,\eps}(\tau_{x,R})$ and we will use directly \eqref{eq:local_times_and_Bessel_process}. Difficulties arise when we start to make $L^2$-computations since we need to consider local times at two different centres. We will resolve this issue with the following reasoning. Consider a Brownian excursion from $\partial D(x,1)$ to $\partial D(x,2)$ and condition on the initial and final points of the excursion (this will be important to keep track of the number of excursions). Because of this conditioning, rotational symmetry is broken and the law of the local times $(L_{x,\delta}(\tau_{x,2}), \delta \leq 1)$ is no longer given by a zero-dimensional Bessel process. But if we condition further on the fact that the excursion went deep inside $D(x,1)$, then it will have forgotten its starting position and the law of $(L_{x,\delta}(\tau_{x,2}), \delta \leq 1)$ will be very close to the one given in \eqref{eq:local_times_and_Bessel_process}. This is the content of the \textbf{continuity lemma} (Lemma \ref{lem:independence local times and exit point}) which is a much more precise version of \cite[Lemma 5.1]{jegoGMC} giving a quantitative estimate of the error in the aforementioned approximation. Importantly, this approximation cannot be true if we look at the local times $L_{x,\delta}(\tau_{x,2})$ for \emph{all} radii $\delta \leq 1$. Instead, we must restrict ourselves to dyadic radii $\delta \in \{e^{-n}, n \geq 0\}$ so that the Brownian path has enough space to forget its initial position. See Remark \ref{rem:dyadic_radii}. This is one reason why we cannot define the good events $G_\eps(x)$ and $G_\eps'(x)$ using this continuum of radii. Another reason is that it would prevent us from decoupling the two-point estimates needed in the proof of Proposition \ref{prop:Cauchy_L2} (see especially \eqref{eq:proof_prop_Cauchy_g}).

Moreover, we will not define the good events using only local times at dyadic radii neither. Indeed, doing so would then require us to estimate probabilities associated to zero-dimensional Bessel process evaluated at discrete times. These probabilities are much harder to estimate than their continuous time counterpart and our approach cannot afford to lose too much on these estimates (especially in the identifications of the different limiting measures). We will resolve this using the following surprising trick: we will consider a field $(h_{x,\delta}, x \in D, \delta \in (0,1])$ that interpolates the local times $\sqrt{\frac{1}{\delta} L_{x,\delta}(\tau_{x,R})}$ between dyadic radii by zero-dimensional Bessel bridges that have a very small range of dependence (see Lemma \ref{lem:process_h}). In this way, the one-point estimates will be the same as if we considered local times at all radii but we will be able to decouple things to make the two-point computations. We believe this new idea will be useful in subsequent studies.

\medskip

\paragraph*{Paper outline}

The rest of the paper is organised as follows. Section \ref{sec:high_level} proves Theorems \ref{th:convergence_measures} and \ref{th:critical_from_subcritical} subject to the intermediate results Proposition \ref{prop:first_layer_good_event}, Lemma \ref{lem:bad_points} and Propositions \ref{prop:bdd_L2} and \ref{prop:Cauchy_L2}. Section \ref{sec:preliminaries} collects preliminary results that will be used throughout the paper. In particular, it states and proves the continuity lemma and contains results on Bessel processes and barrier estimates associated to 1D Brownian motion. Section \ref{sec:adding_good_events} proves Proposition \ref{prop:first_layer_good_event} and Lemma \ref{lem:bad_points} showing that we can safely add the two layers of good events. Section \ref{sec:L2} is dedicated to the $L^2$ estimates needed to prove Proposition \ref{prop:bdd_L2} and \ref{prop:Cauchy_L2}. Appendix \ref{sec:appendix_bessel_bridge} justifies the existence of the field $(h_{x,\delta}, x \in D, \delta \in (0,1])$ interpolating local times with zero-dimensional Bessel bridges. Finally, Appendix \ref{sec:app_continuity} sketches the proof of Proposition \ref{prop:process_subcritical}.

\medskip

We end this introduction with some notations that will be used throughout the paper. We will denote:

\begin{notation}\label{not:bessel}
For $x > 0$ and $d \geq 0$, $\Pb^d_x$ and $\Eb_x^d$ the law and the expectation under which $(X_t)_{t \geq 0}$ is a $d$-dimensional Bessel process starting from $x$ at time 0.  $\Pb_x$ and $\Eb_x$ will denote the law and the expectation of 1D Brownian motion starting at $x$. Note that under $\Pb_x$, the process $X$ takes negative and positive values, whereas the process stays nonnegative under $\Pb_x^1$.
\end{notation}

\begin{notation}\label{not:k_x}
For $x \in D$, $k_x$ the smallest nonnegative integer such that $e^{-k_x} \leq |x-x_0|$;
\end{notation}

\begin{notation}
\label{not:R_tau_x,R}
$R$ the diameter of the domain $D$ and for $x \in D$ and $r >0$, $\tau_{x,r}$ the first hitting time of $\partial D(x,r)$;
\end{notation}

\begin{notation}
For $a_\eps \in \R, b_\eps >0, \eps > 0$, we will denote $a_\eps \lesssim b_\eps$ (resp. $a_\eps = O(b_\eps)$, resp. $a_\eps = o(b_\eps)$) if there exists some constant $C>0$ such that for all $\eps >0$, $a_\eps \leq C b_\eps$ (resp. $|a_\eps| \leq C b_\eps$, resp. $a_\eps/b_\eps \to 0$ as $\eps \to 0$). Sometimes we will emphasise the dependency on some parameter $\eta$ by writing for instance $a_\eps = o_\eta(b_\eps)$;
\end{notation}

\begin{notation}
For $x \in \R$, $(x)_+ = \max(x,0)$.
\end{notation}

\noindent
In this paper, $C, c$, etc. will denote generic constants that may vary from line to line.

\section{High level proof of Theorems \ref{th:convergence_measures} and \ref{th:critical_from_subcritical}}
\label{sec:high_level}

To ease notations, we will prove the convergences stated in Theorem \ref{th:convergence_measures} along the radii $\eps \in \{e^{-k}, k \geq 0\}$. The proof extends naturally to all radii $\eps \in (0,1]$. In particular, in what follows we will write $\sup_{\eps >0}$, $\limsup_{\eps >0}$, etc. but we actually mean $\sup_{\eps \in \{e^{-k}, k \geq 0\} }$, $\limsup_{\eps \in \{e^{-k}, k \geq 0\} }$, etc. 

\medskip

We start off by defining the field $(h_{x,\delta}, x \in D, \delta \in (0,1])$ mentioned in Section \ref{subsec:proof_outline}. 
Recall Notation \ref{not:R_tau_x,R}. We will also denote for any $x=(x_1,x_2) \in \R^2$, $\floor{x} = (\floor{x_1}, \floor{x_2})$.

\begin{lemma}\label{lem:process_h}
By enlarging the probability space we are working on if necessary, we can construct a random field $(h_{x,\delta}, x \in D, \delta \in (0,1])$ such that
\begin{itemize}
\item
for all $x \in D$, and $n \geq 0$, conditionally on $\{ L_{x,\delta}(\tau_{x,R}), \delta = e^{-n}, e^{-n-1} \},$  $(h_{x,e^{-t}}, t \in [n,n+1])$ has the law of a zero-dimensional Bessel bridge from $\sqrt{e^n L_{x,e^{-n}}(\tau_{x,R})}$ to $\sqrt{e^{n+1} L_{x,e^{-n-1}}(\tau_{x,R})}$ that is independent of $(B_t, t \geq 0)$ and $(h_{y,\delta}, y \in D, \delta \notin [e^{-n-1}, e^{-n}])$;
\item
for all $n_0 \geq 0$ and $x, y \in D$, conditionally on $\{L_{z,\delta}(\tau_{z,R}), z=x,y, \delta=e^{-n}, n \geq n_0\}$, $(h_{x,\delta}, \delta \leq e^{-n_0})$ and $(h_{y,\delta}, \delta \leq e^{-n_0})$ are independent as soon as $|x-y| \geq 2 e^{-n_0}$;
\item
for all $n \geq 0$ and $z \in e^{-n-10} \Z^2 \cap D$, $(h_{x,\delta}, x \in D, \floor{e^{n+10} x} = e^{n+10}z, e^{-n-1} \leq \delta \leq e^{-n} )$ is continuous.
\end{itemize}
\end{lemma}

See Appendix \ref{sec:appendix_bessel_bridge} for a proof of the existence of such a process. Note that by \eqref{eq:local_times_and_Bessel_process}, for all $n_0 \geq 0$ and for all $x \in D$, conditionally on $L_{x,e^{-n_0}}(\tau_{x,R})$, $(h_{x,e^{-s-n_0}}, s \geq 0)$ has the law of a zero-dimensional Bessel process starting from $\sqrt{e^{n_0} L_{x,e^{-n_0}}(\tau_{x,R})}$.

We now introduce the good events that we will work with: let $\beta, M >0$ be large and define for all $x \in D$ and $\eps \leq |x-x_0|$, $\eps = e^{-k}$,
\[
G_\eps(x) := \left\{ \forall s \in [k_x,k], h_{x,e^{-s}} \leq 2 s + \beta \right\}
\]
and
\[
G_\eps'(x) := \Bigg\{ \forall s \in [k_x,k], h_{x,e^{-s}} \leq 2 s + \beta - \frac{\sqrt{s}}{M \log(2+s)^2 } \Bigg\}.
\]
If $|x-x_0| < \eps$, the above good events do not impose anything by convention.
Let us mention that if $\eps = e^{-k - t_0}$ for some $k \geq 0$ and $t_0 \in (0,1)$, one would need to consider the process
\[
s \mapsto \left\{
\begin{array}{cc}
h_{x,e^{-s}} & \mathrm{if~} s \in [k_x,k],  \\
\sqrt{ e^s L_{x,e^{-s}}(\tau_{x,R}) } & \mathrm{if~} s \in [k,k+t_0]
\end{array}
\right.
\]
instead of $s \mapsto h_{x,e^{-s}}$ to define the good events when $\eps \notin \{e^{-k}, k \geq 0\}$. Again, in what follows we will restrict ourselves to $\eps \in \{e^{-k}, k \geq 0\}$ to ease notations.

We now consider modified versions of the measures $m_\eps^\gamma, \gamma \in (0,2)$, and $m_\eps$ defined respectively in \eqref{eq:def_measure_subcritical_normalisation} and \eqref{eq:def_measure_SH}:
\begin{equation}
\label{eq:def_measure_subcritical_modified}
\hat{m}^\gamma_\eps(dx) := \mathbf{1}_{G_\eps(x)} m_\eps^\gamma(dx), \quad
\dhat{m}_\eps^\gamma(dx) := \mathbf{1}_{G_\eps'(x)} \indic{|x-x_0| \geq 1/M} \hat{m}_\eps^\gamma(dx)
\end{equation}
and
\begin{equation}
\label{eq:def_measure_SH_modified}
\hat{m}_\eps(dx) := \mathbf{1}_{G_\eps(x)} m_\eps(dx), \quad
\dhat{m}_\eps(dx) := \mathbf{1}_{G_\eps'(x)} \indic{|x-x_0| \geq 1/M} \hat{m}_\eps(dx).
\end{equation}
We also consider modified versions of the measure $\mu_\eps$ defined in \eqref{eq:def_measure_derivative}: for all Borel set $A$, set
\begin{equation}
\label{eq:def_measure_derivative_modified}
\hat{\mu}_\eps(A) := \sqrt{|\log \eps|} \eps^2 \int_A \left( - \sqrt{\frac{1}{\eps} L_{x,\eps}(\tau_{x,R})} + 2 \log \frac{1}{\eps} + \beta \right) e^{2 \sqrt{\frac{1}{\eps} L_{x,\eps}(\tau)} } \mathbf{1}_{G_\eps(x)} dx
\end{equation}
and we decompose further
\[
\dhat{\mu}_\eps(dx) := \mathbf{1}_{G_\eps'(x)} \indic{|x-x_0| \geq 1/M} \hat{\mu}_\eps(dx).
\]
We emphasise that in \eqref{eq:def_measure_derivative_modified} the local times are stopped at time $\tau$ or $\tau_{x,R}$ depending on whether the local time is in the exponential or not.

A first step towards the proof of Theorem \ref{th:convergence_measures} consists in showing that these changes of measures are harmless:

\begin{proposition}\label{prop:first_layer_good_event}
Let $A$ be a Borel set.
The following three limits hold in $\P_{x_0}$-probability:
\begin{equation}
\label{eq:lem_first_layer_good_event_a}
\limsup_{\beta \to \infty} \limsup_{\eps \to 0} \abs{\hat{m}_\eps(A) - m_\eps(A)} = 0,
\end{equation}
\begin{equation}
\label{eq:lem_first_layer_good_event_b}
\limsup_{\beta \to \infty} \limsup_{\eps \to 0} \abs{\hat{\mu}_\eps(A) - \mu_\eps(A)} = 0,
\end{equation}
\begin{equation}
\label{eq:lem_first_layer_good_event_c}
\limsup_{\beta \to \infty} \limsup_{\gamma \to 2^{-}} (2-\gamma)^{-1} \limsup_{\eps \to 0} \abs{\hat{m}_\eps^\gamma(A) - m_\eps^\gamma(A)} = 0.
\end{equation}
\end{proposition}

Once the good events $G_\eps(x)$ are introduced, we can perform $L^1$ computations.
Next, we will show:

\begin{lemma}\label{lem:bad_points}
Let $A$ be a Borel set and fix $\beta >0$. We have
\begin{equation}
\label{eq:lem_bad_points_1}
\limsup_{M \to \infty}
\limsup_{\eps \to 0} \EXPECT{x_0}{ \hat{m}_\eps(A) - \dhat{m}_\eps(A) } = 0,
\end{equation}
\begin{equation}
\label{eq:lem_bad_points_2}
\limsup_{M \to \infty}
\limsup_{\eps \to 0} \EXPECT{x_0}{ \hat{\mu}_\eps(A) - \dhat{\mu}_\eps(A) } = 0,
\end{equation}
\begin{equation}
\label{eq:lem_bad_points_subcritical}
\limsup_{M \to \infty}
\limsup_{\gamma \to 2} (2-\gamma)^{-1} \limsup_{\eps \to 0 } \EXPECT{x_0}{ \hat{m}_\eps^\gamma(A) - \dhat{m}_\eps^\gamma(A) } = 0.
\end{equation}
\end{lemma}

The second layer of good events makes the sequences $(\dhat{m}_\eps(D), \eps >0)$, $(\dhat{\mu}_\eps(D), \eps >0)$ and $((2-\gamma)^{-1} \dhat{m}_\eps^\gamma(D), \gamma \in [1,2), \eps < \eps_\gamma)$ bounded in $L^2$. Here
\begin{equation}
\label{eq:def_eps_gamma}
\eps_\gamma := \exp \left( - \exp (2/(2-\gamma)) \right)
\end{equation}
goes to zero very rapidly as $\gamma \to 2$.
We recall that a sequence $(\nu_n, n \geq 1)$ of random Borel measures on $D$ is tight for the topology of weak convergence on $D$ if, and only if, the sequence $(\nu_n(D), n \geq 1)$ of real-valued random variables is tight (see \cite[Exercise 3.8]{BiskupLectures} for instance).

\begin{proposition}\label{prop:bdd_L2}
Fix $\beta >0$ and $M > 0$. We have
\begin{equation}
\label{eq:prop_L2_a}
\int_{D \times D} \sup_{\eps >0} \EXPECT{x_0}{\dhat{m}_\eps(dx) \dhat{m}_\eps(dy) } < \infty,
\end{equation}
\begin{equation}
\label{eq:prop_L2_b}
\int_{D \times D} \sup_{\eps >0} \EXPECT{x_0}{\dhat{\mu}_\eps(dx) \dhat{\mu}_\eps(dy) } < \infty,
\end{equation}
\begin{equation}
\label{eq:prop_L2_subcritical}
\int_{D \times D} \sup_{\gamma \in [1,2)} (2-\gamma)^{-2} \sup_{\eps < \eps_\gamma} \EXPECT{x_0}{\dhat{m}_\eps^\gamma(dx) \dhat{m}_\eps^\gamma(dy) } < \infty.
\end{equation}
In particular, $\sup_{\eps>0} \EXPECT{x_0}{ \dhat{\mu}_\eps(D)^2} < \infty$ and $(\dhat{\mu}_\eps, \eps >0)$ is tight for the topology of weak convergence on $D$. Moreover, any subsequential limit $\dhat{\mu}$ of $(\dhat{\mu}_\eps, \eps >0)$ satisfies: $\P_{x_0}$-a.s. simultaneously for all $x \in D$, $\dhat{\mu}(\{x\})=0$.
\end{proposition}

Finally, we will show:

\begin{proposition}\label{prop:Cauchy_L2}
Fix $\beta>0$ and $M>0$ and let $A$ be a Borel set. Let $(\gamma_n, n \geq 1) \in [1,2)^\N$ be a sequence converging to 2.
\begin{enumerate}
\item
$(\dhat{m}_\eps(A), \eps >0)$, $(\dhat{\mu}_\eps(A), \eps >0)$ and for all $n \geq 1$, $(\dhat{m}_\eps^{\gamma_n}(A), \eps < \eps_{\gamma_n})$ are Cauchy sequences in $L^2$. 
Let $\dhat{m}(A)$, $\dhat{\mu}(A)$ and $\dhat{m}^{\gamma_n}(A), n \geq 1$, be the limiting random variables.
\item
$\dhat{m}(A) = \sqrt{2/\pi} \dhat{\mu}(A) \quad \P_{x_0}$-a.s.
\item
$(2-\gamma_n)^{-1} \dhat{m}^{\gamma_n}(A)$ converges in $L^2$ towards $2\dhat{\mu}(A)$ as $n \to \infty$.
\end{enumerate}
\end{proposition}

We now have all the ingredients to prove Theorems \ref{th:convergence_measures} and \ref{th:critical_from_subcritical}.

\begin{proof}[Proof of Theorems \ref{th:convergence_measures} and \ref{th:critical_from_subcritical}]
Let $A$ be a Borel set. Let $\beta >0$. For all $M>0$, we have
\begin{align*}
&\limsup_{\eps, \delta \to 0} \EXPECT{x_0}{ \abs{\hat{\mu}_\eps(A) - \hat{\mu}_\delta(A)} } \\
&\leq 2 \limsup_{\eps \to 0} \EXPECT{x_0}{ \abs{\hat{\mu}_\eps(A) - \dhat{\mu}_\eps(A)} } + \limsup_{\eps, \delta \to 0} \EXPECT{x_0}{ \abs{\dhat{\mu}_\eps(A) - \dhat{\mu}_\delta(A)}^2 }^{1/2}.
\end{align*}
By Proposition \ref{prop:Cauchy_L2}, the second right hand side term vanishes whereas by Lemma \ref{lem:bad_points} the first right hand side term goes to zero as $M \to \infty$. The left hand side term being independent of $M$, it has to vanish. In other words, $(\hat{\mu}_\eps(A), \eps >0)$ converges in $L^1$ towards some $\hat{\mu}(A, \beta)$ (we keep track of the dependence in $\beta$ here). Let $\hat{\mu}(A, \infty)$ be the almost sure limit of the nondecreasing sequence $\hat{\mu}(A,\beta)$ as $\beta \to \infty$. We now have for any small $\rho >0$ and large $\beta >0$,
\begin{align*}
& \limsup_{\eps \to 0} \PROB{x_0}{ \abs{\mu_\eps(A) - \hat{\mu}(A,\infty)} > \rho} \leq \limsup_{\eps \to 0} \PROB{x_0}{ \abs{\mu_\eps(A) - \hat{\mu}_\eps(A,\beta)} > \rho/3} \\
& + \limsup_{\eps \to 0} \PROB{x_0}{ \abs{\hat{\mu}_\eps(A,\beta) - \hat{\mu}(A,\beta)} > \rho/3} + \PROB{x_0}{ \abs{\hat{\mu}(A,\beta) - \hat{\mu}(A,\infty)} > \rho/3}.
\end{align*}
The second right hand side term vanishes since $(\hat{\mu}_\eps(A,\beta), \eps >0)$ converges (in $L^1$) towards $\hat{\mu}(A,\beta)$. The third term goes to zero as $\beta \to \infty$ since $(\hat{\mu}(A,\beta), \beta >0)$ converges (almost surely) to $\hat{\mu}(A,\infty)$. The first term goes to zero as $\beta \to \infty$ by Proposition \ref{prop:first_layer_good_event}. We have thus obtained the convergence in $\P_{x_0}$-probability of $(\mu_\eps(A),\eps >0)$.

Let $(\gamma_n, n \geq 1) \in [1,2)^\N$ be a sequence converging to 2.
By mimicking the above lines, Proposition \ref{prop:first_layer_good_event}, Lemma \ref{lem:bad_points} and Proposition \ref{prop:Cauchy_L2} imply that
\[
\lim_{\eps \to 0} \left( m_\eps(A) - \sqrt{\frac{2}{\pi}} \mu_\eps(A) \right) = 0
\quad \mathrm{and} \quad
\lim_{n \to \infty} \lim_{\eps \to 0} \left( \frac{1}{2-\gamma_n} m_\eps^{\gamma_n}(A) - 2 \mu_\eps(A) \right) = 0
\]
in $\P_{x_0}$-probability. By \cite{jegoGMC}, we already know that $(m_\eps^{\gamma_n}(A), \eps >0)$ converges to $m^{\gamma_n}(A)$ in probability. We have thus obtained the convergence in probability of
$(m_\eps(A), \eps >0), (\mu_\eps(A), \eps >0)$ and $((2-\gamma_n)^{-1} m^{\gamma_n}(A), n \geq 1)$ and the limits satisfy
\[
\lim_{\eps \to 0} m_\eps(A) = \sqrt{\frac{2}{\pi}} \lim_{\eps \to 0} \mu_\eps(A)
\quad \mathrm{and} \quad
\lim_{n \to \infty} \frac{1}{2-\gamma_n} m^{\gamma_n}(A) = 2 \lim_{\eps \to 0} \mu_\eps(A).
\]
Obtaining the convergence of the measures and the identification of the limiting measures as stated in Theorems \ref{th:convergence_measures} and \ref{th:critical_from_subcritical} is now routine.

The only points that remained to be checked are points \ref{thmthm:nondegeneracy}-\ref{thmthm:nonatomicity} of Theorem \ref{th:convergence_measures}. Point \ref{thmthm:nonatomicity} follows from the fact that any subsequential limit $\dhat{\mu}$ of $(\dhat{\mu}_\eps, \eps >0)$ are non-atomic (see Proposition \ref{prop:bdd_L2}) and that $\mu(D) - \dhat{\mu}(D)$ is as small as desired (in probability, by tuning the parameters $\beta$ and $M$) by Proposition \ref{prop:first_layer_good_event} and Lemma \ref{lem:bad_points}.
We now turn to Point \ref{thmthm:first_moment}. Since $(\hat{m}_\eps(D), \eps >0)$ converges in $L^1$ towards $\hat{m}(D)$, $\EXPECT{x_0}{\hat{m}(D)} = \lim_{\eps \to 0} \EXPECT{x_0}{\hat{m}_\eps(D)}$. Now, by monotonicity, $\EXPECT{x_0}{m(D)} \geq \lim_{\beta \to \infty} \lim_{\eps \to 0} \EXPECT{x_0}{\hat{m}_\eps(D)}$ which is infinite by \eqref{eq:lem_first_moment_measures_infinity}.

Finally, let us prove Point \ref{thmthm:nondegeneracy} of Theorem \ref{th:convergence_measures}. The fact that $\mu(D)$ is finite $\P_{x_0}$-a.s. follows directly from Proposition \ref{prop:first_layer_good_event} and Lemma \ref{lem:first_moment_measures}. We now want to show that it is positive $\P_{x_0}$-a.s. By Point \ref{thmthm:first_moment} of Theorem \ref{th:convergence_measures}, we already know that it is positive with a positive probability. We are going to bootstrap this to obtain a probability equal to 1.
Let $p \geq 1$ and consider the sequence of stopping times defined by $\sigma_0^{(2)} = 0$ and for all $i \geq 1$,
\[
\sigma_i^{(1)} := \inf \{ t > \sigma_{i-1}^{(2)}, |B_t - x_{i-1}| = 2^{-p} \},
\quad
\sigma_i^{(2)} := \inf \{ t > \sigma_i^{(1)}, |B_t - x_0| = 2^{-p+1} i \}
\]
and $x_i := B_{\sigma_i^{(2)}}$.
For $i \geq 0$, let $\mu_i$ be the critical Brownian multiplicative chaos in the domain $(D(x_i,2^{-p}),x_i)$  between the times $\sigma_{i}^{(2)}$ and $\sigma_{i+1}^{(1)}$.
Let $I:= \floor{ \d(x_0, \partial D) 2^p/10}$. Since $\mu \leq \sum_{i=0}^I \mu_i$, we have
\[
\PROB{x_0}{\mu(D) = 0} \leq \PROB{x_0}{\forall i =0 \dots I, \mu_i(D(x_i,2^{-p})) = 0}.
\]
By Markov property and translation invariance, the probability on the right hand side is equal to
\[
\PROB{x_0}{\mu_0(D(x_0,2^{-p})) = 0}^{I+1}.
\]
By scaling of critical Brownian multiplicative chaos coming from Corollary \ref{cor:conformal}, the probability $\PROB{x_0}{\mu_0(D(x_0,2^{-p})) = 0}$ does not depend on $p$. Moreover, thanks to Theorem \ref{th:convergence_measures}, Point \ref{thmthm:first_moment}, it is strictly less than one. By letting $p \to \infty$, we thus deduce that $\PROB{x_0}{\mu(D) = 0} = 0$ concluding the proof.
\end{proof}

Proposition \ref{prop:subcritical_vanishes} now follows:

\begin{proof}[Proof of Proposition \ref{prop:subcritical_vanishes}]
Recall that $m_\eps^{\gamma=2}(D) = m_\eps(D) / \sqrt{\log \eps|}$. By Theorem \ref{th:convergence_measures}, $(m_\eps(D), \eps >0)$ converges in $\P_{x_0}$-probability towards a nondegenerate random variable. Hence $(m_\eps^{\gamma=2}(D), \eps >0)$ converges in $\P_{x_0}$-probability to zero as desired.
\end{proof}

The remaining of the paper is devoted to the proof of the above intermediate statements.

\section{Preliminaries}
\label{sec:preliminaries}

\subsection{Local times as exponential random variables}

In this short section we recall some results of \cite{jegoGMC} that allow us to approximate local times of circles by exponential random variables. We start by recalling the behaviour of the Green function.

\begin{lemma}[\cite{jegoGMC}, Lemma 2.1]
\label{lem:Green_function_estimate}
For all $x \in \C$, $r > \eps >0$ and $y \in \partial D(x, \eps)$, we have:
\begin{align}
\EXPECT{y}{L_{x,\eps}(\tau_{\partial D(x,r)})} & = 2 \eps \log \frac{r}{\eps}. \label{eq:lem_Green_estimate_circle}
\end{align}
\end{lemma}

In the following lemma, we denote by $\CR(x,D)$ the conformal radius of $D$ seen from $x$ and by $G_D$ the Green function of $D$ with Dirichlet boundary conditions normalised so that $G_D(x,y) \sim - \log |x-y|$ as $x \to y$. Recall also Notation \ref{not:R_tau_x,R}.

\begin{lemma}\label{lem:first_moment}
Let $\eta >0$, $x \in D$ and $\eps >0$ such that the disc $D(x,\eps)$ is included in $D$ and is at distance at least $\eta$ from $\partial D$. Let $y \in \partial D(x,\eps)$. Then $L_{x,\eps}(\tau)$ under $\P_y$ stochastically dominates and is stochastically dominated by exponential variables with mean
\[
2 \eps \log \frac{\CR(x,D)}{\eps} + o_\eta(\eps).
\]
In particular,
\begin{equation}
\label{eq:lem_first_moment}
\EXPECT{y}{e^{2 \sqrt{\frac{1}{\eps} L_{x,\eps}(\tau)}}} = (1+o_\eta(1)) 2 \sqrt{2\pi} \CR(x,D)^2 \sqrt{|\log \eps|} \eps^{-2}.
\end{equation}
Moreover, if $x_0 \notin D(x,\eps)$,
\begin{equation}
\label{eq:lem_hitting_proba}
\PROB{x_0}{\tau_{x,\eps} < \tau} = (1+o_\eta(1)) \frac{G_D(x_0,x)}{|\log \eps|}.
\end{equation}
\end{lemma}

\begin{proof}
\eqref{eq:lem_hitting_proba} is part of \cite[Lemma 2.2]{jegoGMC}. The claim about the stochastic dominations is a consequence of \cite[Section 2]{jegoGMC} as explained at the beginning of the proof of \cite[Proposition 3.1]{jegoGMC}. \eqref{eq:lem_first_moment} is then an easy computation with exponential variables.
\end{proof}

\subsection{Continuity lemma}

We now state a refinement of Lemma 5.1 of \cite{jegoGMC}. We indeed need a quantitative estimate on the error that we make when we forget about the exit point of the excursion.

\begin{lemma}\label{lem:independence local times and exit point}
Let $k,k',n \geq 0$ with $k' \geq k+1$ and $n \geq k' - k$. Denote $\eta = e^{-k}$, $\eta' = e^{-k'}$ and for all $i=1 \dots k' - k$, $r_i = \eta e^{-i}$. Consider $0 < r_n < \dots < r_{k'-k+1} < r_{k'-k} = \eta'$ and for $i = 1 \dots n$, $T_i \in \Bc([0,\infty))$. For any $y \in \partial D(0,\eta/e)$, we have
\begin{align}
1 - p(\eta'/\eta) \leq 
\frac{ \PROB{y}{ \forall i=1 \dots n, L_{0,r_i}(\tau_{0,\eta}) \in T_i \vert \tau_{0,\eta'} < \tau_{0,\eta}, B_{\tau_{0,\eta}}} }{ \PROB{y}{ \forall i=1 \dots n, L_{0,r_i}(\tau_{0,\eta}) \in T_i \vert \tau_{0,\eta'} < \tau_{0,\eta}} }
\leq 1 + p(\eta'/\eta)
\end{align}
with $p(u) \leq \frac{1}{c} \exp \left(- c |\log u|^{1/2} \right)$ for some universal constant $c >0$.
\end{lemma}

\begin{remark}\label{rem:dyadic_radii}
It is crucial that we consider dyadic radii $r \in \{\eta e^{-i}, i =1 \dots k' -k\}$ between $\eta'$ and $\eta/e$ since there is no hope to obtain such a result if we were looking at the local times $L_{0,r}(\tau_{0,\eta})$ for all $r \leq \eta/e$. Indeed, if we condition the Brownian motion to spend very little time in the disc $D(0,\eta/e)$ before hitting $\partial D(0,\eta)$ (which is a function of $L_{0,r}(\tau_{0,\eta}), r \leq \eta/e$), $B_{\tau_{0,\eta}}$ will favour points on $\partial D(0,\eta)$ close to the starting position $y$, even if we condition further the trajectory to visit $D(0,\eta')$ before exiting $D(0,\eta)$.
\end{remark}

\begin{proof}[Proof of Lemma \ref{lem:independence local times and exit point}]
The proof is inspired from the one of \cite[Lemma 5.1]{jegoGMC}.
In this proof, we will write $u = \pm v$ when we mean $-v \leq u \leq v$.
To ease notations, we will denote $\tau_\eta := \tau_{0,\eta}, \tau_{\eta'} := \tau_{0,\eta'}$ and for all $i=1 \dots n, L_{r_i} := L_{0,r_i}(\tau_{0,\eta})$. Take $C \in \Bc \left( \partial D(0,\eta) \right)$. We will denote $\mathrm{Leb}(C)$ for the Lebesgue measure on $\partial D(0,\eta)$ of $C$. It is enough to show that
\begin{align}
\label{eq:proof indep1}
& \PROB{y}{B_{\tau_\eta} \in C ,\tau_{\eta'} < \tau_\eta, \forall i=1 \dots n, L_{r_i} \in T_i} \\
& ~~~~ = \left( 1 \pm \frac{1}{c} \exp \left(- c \abs{\log \frac{\eta'}{\eta}}^{1/3} \right) \right) \frac{\PROB{y}{B_{\tau_\eta} \in C, \tau_{\eta'} < \tau_\eta }}{\PROB{y}{\tau_{\eta'} < \tau_\eta}} \PROB{y}{\tau_{\eta'} < \tau_\eta, \forall i=1 \dots n, L_{r_i} \in T_i}. \nonumber
\end{align}
Moreover, establishing \eqref{eq:proof indep1} can be reduced to show that
\begin{align}
\label{eq:proof indep2}
& \PROB{y}{B_{\tau_\eta} \in C, \tau_{\eta'} < \tau_\eta, \forall i=1 \dots n, L_{r_i} \in T_i} \\
& ~~~~~~~ = \left( 1 \pm \frac{1}{c} \exp \left(- c \abs{\log \frac{\eta'}{\eta}}^{1/3} \right) \right) \frac{\mathrm{Leb}(C)}{2 \pi \eta} ~ \PROB{y}{\tau_{\eta'} < \tau_\eta, \forall i=1 \dots n, L_{r_i} \in T_i}. \nonumber
\end{align}
Indeed, applying \eqref{eq:proof indep2} to $T_i = [0,\infty)$ for all $i$ gives
\[
\PROB{y}{B_{\tau_\eta} \in C, \tau_{\eta'} < \tau_\eta }
= \left( 1 \pm \frac{1}{c} \exp \left(- c \abs{\log \frac{\eta'}{\eta}}^{1/3} \right) \right) \PROB{y}{\tau_{\eta'} < \tau_\eta} \frac{\mathrm{Leb}(C)}{2 \pi \eta},
\]
which combined with \eqref{eq:proof indep2} leads to \eqref{eq:proof indep1} with slightly different constants.
Finally, after reformulation of \eqref{eq:proof indep2}, to finish the proof we only need to prove that
\begin{equation}
\label{eq:proof prop local times indep exit point}
\PROB{y}{B_{\tau_\eta} \in C \vert \tau_{\eta'} < \tau_\eta, \forall i=1 \dots n, L_{r_i} \in T_i}
= \left( 1 \pm \frac{1}{c} \exp \left(- c \abs{\log \frac{\eta'}{\eta}}^{1/3} \right) \right) \frac{\mathrm{Leb}(C)}{2 \pi \eta}.
\end{equation}

The skew-product decomposition of Brownian motion (see \cite{kallenberg2002foundations}, Corollary 16.7 for instance) tells us that we can write
\[
(B_t,t \geq 0) \overset{\mathrm{(d)}}{=} (\abs{B_t} e^{i \theta_t}, t \geq 0)
\mathrm{~with~}
(\theta_t, t \geq 0) = (w_{\sigma_t}, t \geq 0)
\]
where $(w_t, t \geq 0)$ is a one-dimensional Brownian motion independent of the radial part $(\abs{B_t}, t \geq 0)$ and $(\sigma_t, t \geq 0)$ is a time-change that is adapted to the filtration generated by $(\abs{B_t}, t \geq 0)$:
\[
\sigma_t = \int_0^t \frac{1}{\abs{B_s}^2} ds.
\]
In particular, under $\prob_y$, we have the following equality in law
\begin{equation}
\label{eq:proof prop local times indep exit point2}
\left( \tau_\eta, \abs{B_t}, t < \tau_\eta, B_{\tau_\eta} \right)
\overset{\mathrm{(d)}}{=}
\left( \tau_\eta, \abs{B_t}, t < \tau_\eta, \eta e^{i \theta_0 + i \varsigma \Nc} \right)
\end{equation}
where $\theta_0$ is the argument of $y$, $\Nc$ is a standard normal random variable independent of the radial part $(\abs{B_t}, t \geq 0)$ and
\[
\varsigma = \sqrt{\int_0^{\tau_\eta} \frac{1}{\abs{B_s}^2} ds}.
\]

We now investigate a bit the distribution of $e^{i \theta_0 + it \Nc}$ for some $t>0$. More precisely, we want to give a quantitative description of the fact that if $t$ is large, the previous distribution should approximate the uniform distribution on the unit circle. Using the probability density function of $\Nc$ and then using Poisson summation formula, we find that the probability density function $f_t(\theta)$ of $e^{i \theta_0 + it \Nc}$ at a given angle $\theta$ is given by
\begin{align*}
f_t(\theta) & = \frac{1}{\sqrt{2 \pi} t} \sum_{n \in \Z} e^{-(\theta - \theta_0 + 2\pi n)^2/(2t^2)} = \frac{1}{2 \pi} \sum_{p \in \Z} e^{ip(\theta - \theta_0)} e^{- p^2 t^2/2} \\
& = \frac{1}{2 \pi} \left( 1 + 2 \sum_{p=1}^\infty \cos(p (\theta - \theta_0)) e^{- p^2 t^2 / 2} \right).
\end{align*}
In particular, we can control the error in the approximation mentioned above by: for all $\theta \in [0,2 \pi]$,
\[
\abs{f_t(\theta) - \frac{1}{2\pi} } \leq \frac{1}{\pi} \sum_{p=1}^\infty e^{-p^2 t^2/2} \leq C_1 \max \left( 1, \frac{1}{t} \right) e^{-t^2/2}
\]
for some universal constant $C_1>0$.

We now come back to the objective \eqref{eq:proof prop local times indep exit point}. Using the identity \eqref{eq:proof prop local times indep exit point2} and because the local times $L_{r_i}$ are measurable with respect to the radial part of Brownian motion, we have by triangle inequality
\begin{align*}
& \abs{ \PROB{y}{B_{\tau_\eta} \in C \vert \tau_{\eta'} < \tau_\eta, \forall i=1 \dots n, L_{r_i} \in T_i} - \frac{\mathrm{Leb}(C)}{2 \pi \eta} } \\
& \leq
\EXPECT{y}{ \left. \int_0^{2\pi} \abs{ f_\varsigma(\theta) - \frac{1}{2 \pi} } \indic{ \eta e^{i \theta} \in C} d \theta \right\vert \tau_{\eta'} < \tau_\eta, \forall i =1 \dots n, L_{r_i} \in T_i } \\
& \leq
C_1 \frac{\mathrm{Leb}(C)}{\eta} \EXPECT{y}{ \left. \max \left( 1, \frac{1}{\varsigma} \right) e^{-\varsigma^2/2} \right\vert \tau_{\eta'} < \tau_\eta, \forall i =1 \dots n, L_{r_i} \in T_i } \\
& \leq
C_1 \frac{\mathrm{Leb}(C)}{\eta} \EXPECT{y}{ \left. \max \left( 1, \frac{1}{\varsigma'} \right) e^{-(\varsigma')^2/2} \right\vert \tau_{\eta'} < \tau_\eta, \forall i =1 \dots n, L_{r_i} \in T_i }
\end{align*}
where
\[
\varsigma' := \sqrt{\int_{\tau_{r_n}}^{\tau_\eta} \frac{1}{\abs{B_s}^2} ds}.
\]
To conclude the proof, we want to show that
\[
\EXPECT{y}{ \left. \max \left( 1, \frac{1}{\varsigma'} \right) e^{-(\varsigma')^2/2} \right\vert \tau_{\eta'} < \tau_\eta, \forall i =1 \dots n, L_{r_i} \in T_i } \leq \frac{1}{c} \exp \left(- c \abs{\log \frac{\eta'}{\eta}}^{1/2} \right).
\]
By conditioning on the trajectory up to $\tau_{\eta'}$, it is enough to show that for any $T_i' \in \Bc([0,\infty)), i = 1 \dots n$, for any $z \in \partial D(0,\eta')$,
\begin{equation}
\label{eq:proof indep3}
\EXPECT{z}{ \left. \max \left( 1, \frac{1}{\varsigma'} \right) e^{-(\varsigma')^2/2} \right\vert \forall i =1 \dots n, L_{r_i} \in T_i' } \leq \frac{1}{c} \exp \left(- c \abs{\log \frac{\eta'}{\eta}}^{1/2} \right).
\end{equation}
In the following, we fix such $T_i'$ and such a $z$.

Consider the sequence of stopping times defined by: $\sigma_0^{(2)} :=0$ and for all $i = 1 \dots k' + k$,
\[
\sigma_i^{(1)} := \inf \left\{ t > \sigma_{i-1}^{(2)}: \abs{B_t} = \eta' e^{i-1/2} \right\}
\mathrm{~and~}
\sigma_i^{(2)} := \inf \left\{ t > \sigma_i^{(1)}: \abs{B_t} \in \{ \eta' e^{i}, \eta' e^{i-1} \} \right\}.
\]
We only keep track of the portions of trajectories during the intervals $\left[ \sigma_i^{(1)}, \sigma_i^{(2)} \right]$ by bounding from below $\varsigma'$ by
\[
(\varsigma')^2 \geq \sum_{i=1}^{k'-k} \frac{\sigma_i^{(2)} - \sigma_i^{(1)}}{(\eta' e^i)^2}.
\]
Notice that by Markov property, conditioning on $\{ \forall i =1 \dots n, L_{r_i} \in T_i' \}$ impacts the variables $\sigma_i^{(2)} - \sigma_i^{(1)}$ only through $\abs{B_{\sigma_i^{(2)}}}$.
Since
\[
v \mapsto \max \left( 1, \frac{1}{v^{1/2}} \right) e^{-v/2}
\]
is convex, we deduce by Jensen's inequality that
\begin{align*}
& \EXPECT{z}{ \left. \max \left( 1, \frac{1}{\varsigma'} \right) e^{-(\varsigma')^2/2} \right\vert \forall i =1 \dots n, L_{r_i} \in T_i' } \\
& \leq \frac{1}{k'-k} \sum_{i=1}^{k'-k} \EXPECT{z}{ \left. \max \left( 1, \frac{1}{k'-k} \frac{(\eta'e^i)^2}{\sigma_i^{(2)}-\sigma_i^{(1)}} \right)^{1/2} \exp \left( - \frac{k'-k}{2} \frac{\sigma_i^{(2)}-\sigma_i^{(1)}}{(\eta'e_i)^2} \right) \right\vert \abs{B_{\sigma_i^{(2)}}} }.
\end{align*}
By Markov property and Brownian scaling, we have obtained
\begin{align*}
& \EXPECT{z}{ \left. \max \left( 1, \frac{1}{\varsigma'} \right) e^{-(\varsigma')^2/2} \right\vert \forall i =1 \dots n, L_{r_i} \in T_i' } \\
& \leq \max_{r = 1,e^{-1}} \EXPECT{e^{-1/2}}{ \left. \max \left( 1, \frac{1}{(k'-k)\sigma_*} \right)^{1/2} \exp \left( - \frac{(k'-k)\sigma_*}{2} \right) \right\vert \abs{B_{\sigma_*}} = r }.
\end{align*}
where $\sigma_* := \inf \{ t > 0 : |B_t| \in \{ 1, e^{-1} \} \}$. Now, one can show (see \cite[Section 14]{Doob55} for instance) that there exists a universal constant $c>0$ such that for all $s \geq 1$,
\[
\EXPECT{e^{-1/2}}{e^{-s \sigma_*}} \leq e^{-c \sqrt{s}}.
\]
Since $\min_{r=1,e^{-1}} \PROB{e^{-1/2}}{|B_{\sigma_*}| = r|} \geq c$ for some universal constant $c>0$, we also have
\[
\max_{r=1,e^{-1}} \EXPECT{e^{-1/2}}{\left. e^{-s \sigma_*} \right\vert \abs{B_{\sigma_*}} = r } \leq C e^{-c \sqrt{s}}.
\]
From this, we deduce that
\[
\max_{r = 1,e^{-1}} \EXPECT{e^{-1/2}}{ \left. \max \left( 1, \frac{1}{(k'-k)\sigma_*} \right) \right\vert \abs{B_{\sigma_*}} = r } \leq C
\]
and therefore, by Cauchy--Schwarz, we obtain that
\[
\max_{r = 1,e^{-1}} \EXPECT{e^{-1/2}}{ \left. \max \left( 1, \frac{1}{(k'-k)\sigma_*} \right)^{1/2} \exp \left( - \frac{(k'-k)\sigma_*}{2} \right) \right\vert \abs{B_{\sigma_*}} = r } \leq C e^{-c \sqrt{k'-k} }.
\]
Recalling that $k'-k = \log \eta'/\eta$, this shows \eqref{eq:proof indep3} which finishes the proof of Lemma \ref{lem:independence local times and exit point}.
\end{proof}

\subsection{Bessel process}

The purpose of this section is to collect properties of Bessel processes that will be needed in this paper. Recall Notation \ref{not:bessel}.

We start off by recalling the following result that can be found for instance in the lecture notes \cite{LawlerBessel}, Proposition 2.2.

\begin{lemmaa}\label{lem:bessel_RN_derivative}
For each $x, t >0$ and $d \geq 0$, the measures $\Pb_x$ and $\Pb^d_x$, considered as measures on paths $\{X_s, s \leq t\}$, restricted to the event $\{ \forall s \leq t, X_s >0\}$ are mutually absolutely continuous with Radon-Nikodym derivative
\[
\frac{d \Pb^d_x}{d \Pb_x} = \left( \frac{X_t}{x} \right)^a \exp \left( - \frac{a(a-1)}{2} \int_0^t \frac{ds}{X_s^2} \right)
\]
where $a = (d-1)/2$.
\end{lemmaa}

We now state a consequence of Lemma \ref{lem:bessel_RN_derivative} and Girsanov's theorem that will allow us to transfer computations on zero-dimensional Bessel process over to 1D Brownian motion and 3D Bessel process.
Let us mention that since 0 is absorbing for the zero-dimensional Bessel process $X$, we will very often write $\indic{X_t >0}$ instead of $\indic{\forall s \leq t, X_s >0}$ for this specific process.

\begin{lemma}\label{lem:Bessel0_Bessel3}
Let $\gamma \in (0,2]$, $t>0$, $r>0$ and let $f: \Cc([0,t],[0,\infty)) \to [0,\infty)$ be a nonnegative measurable function. Then
\begin{align}
\label{eq:lem_Bessel0_BM}
& \sqrt{t} e^{-\frac{\gamma^2}{2} t} \Eb^0_r \left[ e^{\gamma X_t} \indic{X_t >0} f(X_s,s \leq t) \right] \\
\nonumber
& = \sqrt{r} e^{\gamma r} \Eb_r \left[ \left( \frac{t}{X_t + \gamma t} \right)^{1/2} \exp \left( - \frac{3}{8} \int_0^t \frac{ds}{(X_s + \gamma s)^2} \right) \indic{\forall s \leq t, X_s + \gamma s > 0} f(X_s + \gamma s, s \leq t) \right].
\end{align}
In particular,
\begin{align}
\label{eq:lem_Bessel0_BM_bound}
& \sqrt{t} e^{-\frac{\gamma^2}{2} t} \Eb^0_r \left[ e^{\gamma X_t} \indic{X_t >0} f(X_s,s \leq t) \right] \leq \sqrt{r} e^{\gamma r} \Eb_r \left[ \left( \frac{t}{X_t + \gamma t} \right)_+^{1/2} f(X_s + \gamma s, s \leq t) \right].
\end{align}
Moreover,
\begin{align}
\label{eq:lem_Bessel0_Bessel3_gamma=2}
& \sqrt{t} e^{-2t} \Eb^0_r \left[ e^{2X_t} \indic{ X_t > 0} \indic{\forall s \leq t, X_s < 2s + \beta} f \left( X_s, s \leq t \right) \right] \\
& = 2^{-1/2} \sqrt{r} e^{2r} (\beta-r) \Eb_{\beta-r}^3 \Bigg[ \frac{1}{X_t} \left( 1 - \frac{X_t-\beta}{2t} \right)^{-1/2} \indic{\forall s \leq t, 2s - X_s + \beta > 0} \nonumber \\
& ~~~~~~~~~~~~~~~~~~~~~~~~~~\times f(2s - X_s + \beta, s \leq t) \exp \left( - \frac{3}{8} \int_0^t \frac{ds}{(2s-X_s+\beta)^2} \right) \Bigg]
\nonumber
\end{align}
and
\begin{align}
\label{eq:lem_Bessel0_Bessel3_bound}
& \sqrt{t} e^{-2 t} \Eb^0_r \left[ e^{2 X_t} \indic{X_t >0} \indic{\forall s \leq t, X_s < 2s + \beta} f(X_s,s \leq t) \right] \\
& \nonumber \leq \sqrt{r} e^{2 r} (\beta -r) \Eb^3_{ \beta - r} \left[ \frac{1}{X_t} \left( \frac{t}{2t + \beta - X_t} \right)_+^{1/2} f( \beta + 2s - X_s , s \leq t) \right].
\end{align}
Finally,
\begin{equation}
\label{eq:lem_Bessel_limit}
\lim_{\beta \to \infty}
\lim_{t \to \infty} t e^{-2t} \Eb^0_r \left[ e^{2X_t} \indic{\forall s \leq t, X_s < 2s + \beta} \right] = \infty.
\end{equation}
\end{lemma}

\begin{proof}[Proof of Lemma \ref{lem:Bessel0_Bessel3}]
By Lemma \ref{lem:bessel_RN_derivative}, the left hand side of \eqref{eq:lem_Bessel0_BM} is equal to
\begin{align*}
\sqrt{r} \sqrt{t} e^{-\frac{\gamma^2}{2} t} \Eb_r \left[ X_t^{-1/2} \exp \left( - \frac{3}{8} \int_0^t \frac{ds}{X_s^2} \right) e^{\gamma X_t} \indic{\forall s \leq t, X_s >0} f(X_s,s \leq t) \right].
\end{align*}
Girsanov's theorem concludes the proof of \eqref{eq:lem_Bessel0_BM}. 
\eqref{eq:lem_Bessel0_BM_bound} follows directly from \eqref{eq:lem_Bessel0_BM}.
Now, by \eqref{eq:lem_Bessel0_BM}, the left hand side of \eqref{eq:lem_Bessel0_Bessel3_gamma=2} is equal to
\begin{align*}
& \sqrt{r}e^{2 r} \Eb_r \Bigg[ \left( \frac{t}{X_t + 2 t} \right)_+^{1/2} \exp \left( - \frac{3}{8} \int_0^t \frac{ds}{(X_s + 2s)^2} \right) \indic{\forall s \leq t, X_s + 2 s >0} \\
& \times \indic{\forall s \leq t, X_s < \beta} f(X_s + 2 s, s \leq t) \Bigg] \\
& = \sqrt{r}e^{2 r} \Eb_{\beta -r} \Bigg[ \left( \frac{t}{2t + \beta - X_t} \right)_+^{1/2} \exp \left( - \frac{3}{8} \int_0^t \frac{ds}{(2s + \beta - X_s)^2} \right) \indic{\forall s \leq t, 2s + \beta - X_s >0} \\
& \times \indic{\forall s \leq t, X_s > 0} f( 2 s + \beta - X_s , s \leq t) \Bigg].
\end{align*}
By Lemma \ref{lem:bessel_RN_derivative}, this is in turn equal to the right hand side of \eqref{eq:lem_Bessel0_Bessel3_gamma=2}. \eqref{eq:lem_Bessel0_Bessel3_bound} is an easy consequence of \eqref{eq:lem_Bessel0_Bessel3_gamma=2} and we now turn to the proof of \eqref{eq:lem_Bessel_limit}.  We use \eqref{eq:lem_Bessel0_Bessel3_gamma=2} and we add the stronger constraint that $\{ \forall s \leq t, 2s -X_s+\beta > r/2 +s\}$ in order to have a lower bound. On this event, we can bound
\[
\exp \left( - \frac{3}{8} \int_0^t \frac{ds}{(2s-X_s + \beta)^2} \right)
\geq \exp \left( - \frac{3}{8} \int_0^\infty \frac{ds}{(r/2 + s)^2} \right) = c_r.
\]
Moreover, we simply bound
\[
\left( 1 - \frac{X_t-\beta}{2t} \right)^{-1/2} \geq \left( 1 + \frac{\beta}{2t} \right)^{-1/2},
\]
which overall shows that
\begin{align*}
t e^{-2t} \Eb^0_r \left[ e^{2X_t} \indic{\forall s \leq t, X_s < 2s + \beta} \right]
\geq c_r (\beta-r)\left( 1 + \frac{\beta}{2t} \right)^{-1/2} \sqrt{t} \Eb^3_{\beta-r} \left[ \frac{1}{X_t} \indic{\forall s \leq t, 2s -X_s+\beta > r/2 +s} \right].
\end{align*}
Since $X_t$ under $\Pb^3_{\beta-r}(\cdot \vert \forall s \leq t, 2s -X_s+\beta > r/2 +s)$ is stochastically dominated by $X_t$ under $\Pb^3_{\beta-r}$, we can further bound
\[
\Eb^3_{\beta-r} \left[ \frac{1}{X_t} \indic{\forall s \leq t, 2s -X_s+\beta > r/2 +s} \right]
\geq \Eb^3_{\beta-r} \left[ \frac{1}{X_t} \right] \Pb^3_{\beta-r} \left( \forall s \leq t, 2s -X_s+\beta > r/2 +s \right).
\]
Lemma \ref{lem:Bessel3_barrier}, Point \ref{lemlem:1/X_t}, shows that $\Eb^3_{\beta-r} \left[ \frac{\sqrt{t}}{X_t} \right] \to \sqrt{2/\pi}$ as $t \to \infty$. Therefore
\[
\liminf_{t \to \infty} t e^{-2t} \Eb^0_r \left[ e^{2X_t} \indic{\forall s \leq t, X_s < 2s + \beta} \right]
\geq c_r (\beta -r) \Pb^3_{\beta-r} \left( \forall s \geq 0, 2s -X_s+\beta > r/2 +s \right).
\]
To see that the above probability remains bounded away from zero as $\beta \to \infty$, we can for instance notice that a three-dimensional Bessel process which starts at $\beta-r$ is stochastically dominated by the sum of three independent one-dimensional Bessel processes $X^{(i)}$, $i=1,2,3$, starting at the origin, plus $\beta -r$ (this follows by bounding $\sqrt{a^2 + b^2 + c^2} \leq |a| + |b| + |c|$). Therefore
\begin{align*}
\Pb^3_{\beta-r} \left( \forall s \geq 0, 2s -X_s+\beta > r/2 +s \right)
\geq \Prob{ \forall s \geq 0,  \sum_{i=1}^3 X_s^{(i)} < r/2 + s} > 0.
\end{align*}
This concludes the proof of \eqref{eq:lem_Bessel_limit}. 
\end{proof}

We now collect some properties of three-dimensional Bessel process.

\begin{lemma}\label{lem:Bessel3_barrier}
Let $K >0$.
\begin{enumerate}
\item
\label{lemlem:barrier}
Uniformly over $r \in [0,K]$,
\[
\Pb_r^3 \left( \forall t \geq 0, X_t \geq \frac{\sqrt{t}}{M \log (2 + t)^2}  \right) \to 1
\]
as $M \to \infty$.
\item
\label{lemlem:1/X_t}
$\Eb^3_r \left[ \frac{1}{X_t} \right] = \sqrt{\frac{2}{\pi t}} + o \left( \frac{1}{\sqrt{t}} \right)$ as $t \to \infty$, where the error is uniform over $r \in [0,K]$.
\item
For any $q \in (0,3)$,
\label{lemlem:1/X_t^2}
$\sup_{t \geq 1} \sup_{r >0} \Eb^3_r \left[ \frac{t^{q/2}}{X_t^q} \right]$ is finite.
\item
For any $q \in (0,1)$,
\label{lemlem:1/X_t^2bis}
$\sup_{t \geq 1} \sup_{K \geq 0} \sup_{r \in [0,K]} \Eb_r^3 \left[ \left( 1 - \frac{X_t - K}{2t} \right)_+^{-q} \right]$ is finite.
\end{enumerate}
\end{lemma}

\begin{proof}[Proof of Lemma \ref{lem:Bessel3_barrier}]
Points \ref{lemlem:barrier}-\ref{lemlem:1/X_t} are part of \cite[Lemma 2.9]{powell2018}. To verify Point \ref{lemlem:1/X_t^2}, notice that $X_t$ under $\Pb^3_0$ is stochastically dominated by $X_t$ under $\Pb^3_r$ for any $r>0$. By scaling, we deduce that
\[
\sup_{t \geq 1} \sup_{r >0} \Eb^3_r \left[ \frac{t^{q/2}}{X_t^q} \right] \leq \sup_{t \geq 1} \Eb^3_0 \left[ \frac{t^{q/2}}{X_t^q} \right] = \Eb^3_0 \left[ \frac{1}{X_1^q} \right].
\]
The density of $X_1$ under $\Pb_0^3$ is explicit (see \cite[Proposition 2.5]{LawlerBessel} for instance) and is given by
\[
\sqrt{\frac{2}{\pi}} y^2 e^{-y^2/2} dy.
\]
We can therefore directly check that $\Eb^3_0 \left[ X_1^{-q} \right]$ is finite as soon as $q < 3$. This concludes the proof of Point \ref{lemlem:1/X_t^2}. Point \ref{lemlem:1/X_t^2bis} follows from a similar direct computation. 
\end{proof}

We conclude this section on Bessel processes with estimates that will be used repeatedly in the paper.

\begin{lemma}\label{lem:Bessel_first_moment}
There exists a universal constant $C>0$ such that the following estimates hold true. For all $K \geq 1, r \in [0,K]$ and $t \geq 1$,
\begin{equation}
\label{eq:lem_Bessel_first_moment_a}
t e^{-2t} \Eb^0_r \left[ e^{2X_t} \indic{\forall s \leq t, X_s \leq 2s + K} \indic{X_t >0} \right] \leq C \sqrt{r}(K-r) e^{2r}
\end{equation}
and
\begin{equation}
\label{eq:lem_Bessel_first_moment_b}
\sqrt{t} e^{-2t} \Eb^0_r \left[ (-X_t + 2t + K) e^{2X_t} \indic{\forall s \leq t, X_s \leq 2s + K} \indic{X_t >0} \right] \leq C \sqrt{r}(K-r) e^{2r}.
\end{equation}
Moreover, for all $K \geq 1, r \in [0,K]$, $\gamma \in (1,2)$ and $t \geq \exp(1/(2-\gamma))$,
\begin{equation}
\label{eq:lem_Bessel_first_moment_c}
\frac{1}{2-\gamma} \sqrt{t} e^{-\gamma^2 t/2} \Eb^0_r \left[ e^{\gamma X_t} \indic{\forall s \leq t, X_s \leq 2s + K} \indic{X_t >0} \right] \leq C \sqrt{r}(K-r) e^{\gamma r}.
\end{equation}
\end{lemma}

\begin{proof}[Proof of Lemma \ref{lem:Bessel_first_moment}]
By \eqref{eq:lem_Bessel0_Bessel3_bound}, the left hand side of \eqref{eq:lem_Bessel_first_moment_b} is at most
\[
2^{-1/2} \sqrt{r}(K-r) e^{2r} \Eb^3_{K-r} \left[ \left( 1 - \frac{X_t - K}{2t} \right)_+^{-1/2} \right].
\]
The expectation with respect to the three-dimensional Bessel process is bounded uniformly in $r \in [0,K], K>0, t\geq 1$ by Lemma \ref{lem:Bessel3_barrier}, point \ref{lemlem:1/X_t^2bis}. This concludes the proof of \eqref{eq:lem_Bessel_first_moment_b}. Now, by \eqref{eq:lem_Bessel0_Bessel3_bound} and then by Cauchy--Schwarz inequality, the left hand side of \eqref{eq:lem_Bessel_first_moment_a} is at most
\begin{align*}
& 2^{-1/2} \sqrt{r}(K-r) e^{2r} \Eb^3_{K-r} \left[ \frac{\sqrt{t}}{X_t} \left( 1 - \frac{X_t - K}{2t} \right)_+^{-1/2} \right] \\
& \leq 2^{-1/2} \sqrt{r}(K-r) e^{2r} \Eb^3_{K-r} \left[ \frac{t}{X_t^2} \right]^{1/2} \Eb^3_{K-r} \left[ \left( 1 - \frac{X_t - K}{2t} \right)_+^{-1} \right]^{1/2}.
\end{align*}
Lemma \ref{lem:Bessel3_barrier}, points \ref{lemlem:1/X_t^2} and \ref{lemlem:1/X_t^2bis}, then concludes the proof of \eqref{eq:lem_Bessel_first_moment_a}. We now turn to the proof of \eqref{eq:lem_Bessel_first_moment_c}.
By \eqref{eq:lem_Bessel0_BM_bound}, the left hand side of \eqref{eq:lem_Bessel_first_moment_c} is at most
\begin{align*}
& \frac{1}{2-\gamma} \sqrt{r} e^{\gamma r} \Eb_r \left[ \left( \frac{t}{X_t + \gamma t} \right)_+^{1/2} \indic{\forall s \leq t, X_s \leq (2-\gamma)s + K} \right] \\
& \leq \frac{1}{2-\gamma} \sqrt{r} e^{\gamma r} \Eb_r \left[ \left( \frac{t}{X_t + \gamma t} \right)_+^{1/2} \indic{ X_t \leq -\gamma t/2 } \right] \\
& + \sqrt{\frac{2}{\gamma}} \frac{1}{2-\gamma} \sqrt{r} e^{\gamma r} \Pb_r \left( \forall s \leq t, X_s \leq (2-\gamma)s + K \right).
\end{align*}
By H\"{o}lder's inequality and an analogue of Lemma \ref{lem:Bessel3_barrier}, Point \ref{lemlem:1/X_t^2bis}, for Brownian motion rather than 3D Bessel process, we see that the last expectation above is at most
\[
\Eb_0 \left[ \left(1+ \frac{X_t + r}{\gamma t} \right)_+^{-2/3} \right]^{3/4} \Pb_0 \left( X_t \leq -r -\gamma t/2 \right)^{1/4} \lesssim e^{-\gamma^2 t/32} \leq 2 - \gamma
\]
by recalling that $t \geq \exp(1/(2-\gamma))$.
On the other hand (see \cite[Proposition 6.8.1]{Resnick1992} for instance),
\[
\Pb_0 \left( \forall s \geq 0, X_s < (2-\gamma) s + K - r \right) = 1 - e^{-2 (K - r) (2-\gamma)} \leq 2 (K - r) (2-\gamma).
\]
Since
\[
\Pb_0 \left( \exists s \geq t, X_s \geq (2-\gamma) s + K - r \right) \lesssim e^{-\frac{(2-\gamma)^2}{16} t} \leq 2-\gamma,
\]
it implies that
\[
\Pb_0 \left( \forall s \leq t, X_s < (2-\gamma) s + K - r \right) \lesssim (K - r) (2-\gamma).
\]
Putting things together yields \eqref{eq:lem_Bessel_first_moment_c}. This concludes the proof.
\end{proof}

\subsection{Barrier estimates for 1D Brownian motion}

The purpose of this section is to prove the following lemma.

\begin{lemma}\label{lem:BM_estimates}
There exists $C >0$ such that the following claims hold true. For all $K, H \geq 1$ and all integer $n \geq 1$,
\begin{equation}\label{eq:lem_BM_estimates_1}
\Pb_0 \left( \forall k = 0 \dots n-1, \min_{[k,k+1]} X \leq 2 \log (1+k) + K \right) \leq \frac{CK^2}{\sqrt{n}}
\end{equation}
and
\begin{equation}
\label{eq:lem_BM_estimates_2}
\Pb_0 \left( \forall k =0 \dots n-1, \min_{[k,k+1]} X \leq 2 \log(k+1) + K, \exists s \in [0,n], X_s > K + H \right) \leq \frac{C K^2 e^{-H/64}}{\sqrt{n}}.
\end{equation}
Moreover, for all $K, H \geq 1$, $\gamma \in [1,2)$ and all integer $n \geq (2-\gamma)^{-4}$,
\begin{equation}\label{eq:lem_BM_estimates_3}
\Pb_0 \left( \forall k = 0 \dots n-1, \min_{[k,k+1]} X \leq (2-\gamma) k + 2 \log (1+k) + K \right) \leq C K^2 (2-\gamma)
\end{equation}
and
\begin{align}\label{eq:lem_BM_estimates_4}
\Pb_0 \Big( & \forall k = 0 \dots n-1, \min_{[k,k+1]} X \leq (2-\gamma)k +  2 \log (1+k) + K, \\
& \exists s \leq n, X_s \geq (2-\gamma)s + K + H \Big) \leq C K^2 e^{-H/64} (2-\gamma). \nonumber
\end{align}
\end{lemma}

We start off with the following intermediate result.

\begin{lemma}\label{lem:BM_barrier_intermediate}
Let $c >0$. There exists $C>0$ such that the following estimates hold. For all $n \geq 1$ and $K \geq 1$,
\begin{equation}
\label{eq:lem_BM_barrier_intermediate_1}
\Pb_0 \left( \forall s \leq n, X_s \leq c \log(1+s) + K \right) \leq C K^2 / \sqrt{n}.
\end{equation}
Moreover, for all $\gamma \in [1,2)$, for all $n \geq (2-\gamma)^{-4}$ and $K \geq 1$,
\begin{equation}
\label{eq:lem_BM_barrier_intermediate_2}
\Pb_0 \left( \forall s \leq n, X_s \leq (2-\gamma)s + c \log(1+s) + K \right) \leq C K^2 (2-\gamma).
\end{equation}
\end{lemma}

\begin{proof}
We start by proving \eqref{eq:lem_BM_barrier_intermediate_1}.
If $K > n^{1/4}$, then the result is clear by bounding the probability by one. In the rest of the proof we thus assume that $K \leq n^{1/4}$.
Let us denote $K_n = c \log(1+n) + K$.
By the reflection principle, 
\begin{align*}
\Pb_0 \left( \forall s \leq n, X_s \leq K_n, X_n \geq - n^{1/4} \right) & = \int_{-n^{1/4}}^{K_n} \frac{1}{\sqrt{2\pi n}} \left( e^{-\frac{x^2}{2n}} - e^{-\frac{(2K_n - x)^2}{2n}} \right) dx
\end{align*}
For all $x \in [-n^{1/4},K_n]$, we can bound
\[
e^{-\frac{x^2}{2n}} - e^{-\frac{(2K_n - x)^2}{2n}}
= e^{-\frac{x^2}{2n}} \left( 1 - e^{-2\frac{(K_n^2 - K_n x)}{n}} \right)
\lesssim e^{-\frac{x^2}{2n}} \frac{K_n^2 - K_n x}{n} \lesssim e^{-\frac{x^2}{2n}} \frac{K_n n^{1/4}}{n},
\]
implying that
\begin{align*}
\Pb_0 \left( \forall s \leq n, X_s \leq K_n, X_n \geq - n^{1/4} \right) \lesssim \frac{K_n n^{1/4}}{n} \int_{-n^{1/4}}^{K_n} \frac{1}{\sqrt{2\pi n}} e^{-\frac{x^2}{2n}} dx \lesssim \frac{ K_n}{n}.
\end{align*}
Another similar consequence of the reflection principle is that
\[
\Pb_0 \left( \forall s \leq n, X_s \leq K_n \right) \gtrsim K_n/\sqrt{n}.
\]
Therefore
\begin{align*}
& \Pb_0 \left( X_n \geq -n^{1/4} \vert \forall s \leq n, X_s \leq c \log(1+s) + K \right) \\
& \leq \Pb_0 \left( X_n \geq -n^{1/4} \vert \forall s \leq n, X_s \leq c \log(1+n) + K \right) \lesssim 1/\sqrt{n}
\end{align*}
and
\begin{align*}
& \Pb_0 \left( \forall s \leq n, X_s \leq c \log(1+s) + K \right) \\
& \lesssim n^{-1/2} \Pb_0 \left( \forall s \leq n, X_s \leq c \log(1+s) + K \right)
+ \Pb_0 \left( \forall s \leq n, X_s \leq c \log(1+s) + K, X_n \leq -n^{1/4} \right) \\
& \lesssim n^{-1/2} \Pb_0 \left( \forall s \leq n, X_s \leq c \log(1+s) + K \right)
+ \Pb_0 \left( X_n \leq -n^{1/4}, \exists s \in [n,n + n^{1/4}], X_s > K \right) \\
& + \Pb_0 \left( \forall s \leq n+n^{1/4}, X_s \leq c' (s \wedge (n+n^{1/4} -s))^{1/20} + K \right).
\end{align*}
By equation (25) of \cite{BramsonDingZeitouni}, the last right hand side term is at most $C K^2 / \sqrt{n}$. The second right hand side term being at most
\[
\Pb_0 \left( \max_{[0,n^{1/4}]} X \geq K + n^{1/4} \right) \lesssim e^{-c n^{1/4}},
\]
we deduce that
\begin{align*}
& \Pb_0 \left( \forall s \leq n, X_s \leq c \log(1+s) + K \right) \\
& \lesssim n^{-1/2} \Pb_0 \left( \forall s \leq n, X_s \leq c \log(1+s) + K \right) + K^2 / \sqrt{n}
\end{align*}
which concludes the proof of \eqref{eq:lem_BM_barrier_intermediate_1}.

We now turn to the proof of \eqref{eq:lem_BM_barrier_intermediate_2}.
Since $n \geq (2-\gamma)^{-4}$,
\begin{align*}
\Pb_0 \left( \exists s \geq n, X_s > (2-\gamma) s \right) & \leq \Pb_0 \left( \exists s \geq n, X_s > s^{3/4} \right) \leq \sum_{k \geq n} \Pb_0 \left( \max_{[k,k+1]} X > k^{3/4} \right) \\
& \leq \sum_{k\geq n} e^{-c \sqrt{k}} \lesssim e^{-c \sqrt{n}} \lesssim 2 - \gamma.
\end{align*}
Hence
\begin{align*}
& \Pb_0 \left( \forall s \leq n, X_s \leq (2-\gamma)s + c \log(1+s) + K \right) \\
& \lesssim 2-\gamma + \Pb_0 \left( \forall s \leq 2n, X_s \leq (2-\gamma) s + C (s \wedge (2n -s))^{1/20} + K \right) \\
& = 2 - \gamma + e^{-(2-\gamma)^2 n } \Eb_0 \left[ e^{-(2-\gamma) X_{2n}} \indic{\forall s \leq 2n, X_s \leq C (s \wedge (2n -s))^{1/20} + K} \right]
\end{align*}
by Girsanov's theorem. Now, by equation (25) of \cite{BramsonDingZeitouni}, we conclude that
\begin{align*}
& \Pb_0 \left( \forall s \leq n, X_s \leq (2-\gamma)s + 3 \log(1+s) + K \right) \\
& \lesssim 2-\gamma + K e^{-(2-\gamma)^2 n } \int_{-\infty}^K (K-x) n^{-3/2} e^{-x^2/(4n)} e^{-(2-\gamma)x } dx \\
& = 2-\gamma + K \int_{-\infty}^{K + 2(2-\gamma)n} (K + 2(2-\gamma)n -y) n^{-3/2} e^{-y^2/(4n)} dy \\
& \lesssim K^2(2-\gamma).
\end{align*}
This finishes the proof of \eqref{eq:lem_BM_barrier_intermediate_2}.
\end{proof}

\begin{proof}[Proof of Lemma \ref{lem:BM_estimates}]
We start by proving \eqref{eq:lem_BM_estimates_1}.
By Lemma \ref{lem:BM_barrier_intermediate}, there exists some universal constant $C_1 >0$ such that for all $t \geq 1$,
\begin{equation}
\label{eq:proof_lem_BM_estimate_b}
\Pb_0 \left( \forall s \in [1,t], X_s \leq 3 \log(1+s) + 2K \right) \leq C_1 K^2/\sqrt{t+1}.
\end{equation}
We thus aim to take care of the minima in \eqref{eq:lem_BM_estimates_1}.
Let $n \geq 1$ and define
\[
p_n := \sup_{ 0 \leq t_0 < 1 } \Pb_0 \left( \forall k \leq n-1, \min_{[k+t_0,k+1+t_0]} X \leq 2 \log(k+1) + K \right).
\]
Let $0 \leq t_0 < 1$. Set
$\tau := \inf\{ s >t_0: X_s \geq 3 \log (1+s) + 2K \}$.
We are going to decompose the above probability according to the value of $\tau$.
Let $k \geq 1$. Notice that on the event $\{ k + t_0 \leq \tau < k+1 + t_0, \min_{[k-1+t_0,k+t_0]} X \leq 2 \log k + K \}$, we have $\max_{u,v \in [k-1+t_0,\tau]} |X_u-X_v| \geq \log(k+1) + K$.
If $k=0$, on the event $\{t_0 \leq \tau < 1+t_0\}$, we simply have $\max_{u \in [0,\tau]} |X_u-X_0| \geq K$ when $X$ starts at 0. Hence
\begin{align*}
& \Pb_0 \left( \forall k \leq n-1, \min_{[k+t_0,k+1+t_0]} X \leq 2 \log(k+1) + K \right) \\
& \leq \Pb_0(\tau \geq n+t_0) + \sum_{k=0}^{n-1} \Pb_0 \Big( k+t_0 \leq \tau < k+1+t_0, \max_{u, v \in [k-1+t_0,\tau]} |X_u - X_v| \geq \log(k+1) + K, \\
& ~~~~~~~~~~~~~~~~~~~~~~~~~~~~~~~~~~ \forall j=k+1 \dots n, \min_{[j+t_0,j+1+t_0]} X \leq 2 \log (j+1) + K \Big).
\end{align*}
By applying Markov's property to the stopping time $\tau$, and by writing $\tilde{X}$ a Brownian motion independent of $\tau$, we see that the last probability written above is equal to
\begin{align*}
& \Eb_0 \Bigg[ \indic{k+t_0 \leq \tau < k+1+t_0, \max_{u, v \in [k-1+t_0,\tau]} |X_u - X_v| \geq \log(k+1) + K} \\
& ~~~~ \times \Pb_0 \left( \forall j =k+1 \dots n, \min_{[j-\tau+t_0,j+1-\tau+t_0]} \tilde{X} \leq 2 \log(j+1) - 3 \log(1+\tau) - K \Big\vert \tau \right)   \Bigg] \\
& \leq \Pb_0 \left( k+t_0 \leq \tau < k+1+t_0, \max_{u, v \in [k-1+t_0,\tau]} |X_u - X_v| \geq \log(k+1) + K \right) \\
& ~~~~ \times \sup_{0 \leq t_0' < 1} \Pb_0 \left( \forall j =0 \dots n-k-1, \min_{[j+t_0',j+1+t_0']} X \leq 2 \log(1 + j) + K \right) \\
& \leq \Pb_0 \left( \max_{u, v \in [0,2]} |X_u - X_v| \geq \log(k+1) + K \right) p_{n-k-1} \\
& \leq \Pb_0 \left( 2 \max_{[0,2]} |X| \geq \log(k+1) + K \right) p_{n-k-1}
\leq e^{-(\log(k+1) + K)^2/16}p_{n-k-1}.
\end{align*}
Moreover, by \eqref{eq:proof_lem_BM_estimate_b},
\[
\sup_{0 \leq t_0 < 1} \Pb_0 (\tau \geq n+t_0) \leq \Pb_0 \left( \forall s \in [1,n], X_s \leq 3 \log(1+s) + 2K \right) \leq C_1 K^2/\sqrt{n+1}.
\]
We have thus proven that
\begin{equation}
\label{eq:proof_lem_BM_estimate_a}
p_n \leq \frac{C_1 K^2}{\sqrt{n+1}} + \sum_{k=0}^{n-1} e^{-(\log(k+1) + K)^2/16}p_{n-k-1}.
\end{equation}
This recursive relation allows us to conclude the proof of \eqref{eq:lem_BM_estimates_1}. We detail the arguments.
Define
\[
C_2 := \sup_{n \geq 1} \sqrt{n+1} \sum_{k=0}^{n-1} e^{-(\log(k+1))^2/16} \frac{1}{\sqrt{n-k}} < \infty
\]
and assume that $K$ is large enough so that we can define
\[
C_K = C_1 / (1 - e^{-K^2/16} C_2).
\]
We clearly have $p_0 \leq 1 \leq C_K K^2 / \sqrt{1+0}$. Let $n \geq 1$ and assume now that for all $k \leq n-1$, $p_k \leq C_K K^2 / \sqrt{k+1}$. By \eqref{eq:proof_lem_BM_estimate_a}, we have
\begin{align*}
p_n & \leq \frac{C_1 K^2}{\sqrt{n+1}} + \sum_{k=0}^{n-1} e^{-(\log(k+1) + K)^2/16} \frac{C_K K^2}{\sqrt{n-k}} \leq \frac{K^2}{\sqrt{n+1}} \left( C_1 + e^{-K^2/16} C_2 C_K  \right)
= \frac{C_K K^2}{\sqrt{n+1}}.
\end{align*}
This concludes the proof by induction of the fact that $p_n \leq C_K K^2 / \sqrt{n+1}$ for all $n \geq 1$. Since $C_K$ does not grow with $K$, this concludes the proof of \eqref{eq:lem_BM_estimates_1}.

We now turn to the proof of \eqref{eq:lem_BM_estimates_2}.
We are first going to show that
\begin{align}\label{eq:proof_lem_BM_barrier}
\Pb_0 \Big( & \forall k = 0 \dots n-1, \min_{[k,k+1]} X \leq 2 \log (1+k) + K, \\
& \exists s \leq n, X_s \geq 3 \log(1+s) + H + K \Big) \leq C e^{-H^2/16} K^2 / \sqrt{n}. \nonumber
\end{align}
By considering the stopping time
\[
\inf \left\{ s >0, X_s \geq 3 \log(1+s) + H +K \right\},
\]
and by following almost the same arguments as above, one can show that the probability in \eqref{eq:proof_lem_BM_barrier} is at most
\begin{align*}
& \sum_{k=0}^{n-1} e^{-(\log(k+1)+H)^2/16} \sup_{0 \leq t_0 < 1} \Pb_0 \left( \forall j \leq n-1-k, \min_{[k+t_0,k+1+t_0]} X \leq 2 \log(k+1) + K \right) \\
& \leq \sum_{k=0}^{n-1} e^{-(\log(k+1)+H)^2/16} p_{n-k}
\lesssim \sum_{k=0}^{n-1} e^{-(\log(k+1)+H)^2/16}  \frac{K^2}{\sqrt{n-k}} \lesssim e^{-H^2/16} \frac{K^2}{\sqrt{n}}
\end{align*}
thanks to the estimates on $p_n$. This shows \eqref{eq:proof_lem_BM_barrier}.
Now, it implies that
\begin{align*}
& \Pb_0 \left( \forall k =0 \dots n-1, \min_{[k,k+1]} X \leq 2 \log(k+1) + K, \exists s \in [0,n], X_s > K + H \right) \\
& \leq C e^{-H^2/64} K^2 / \sqrt{n}
+ \Pb_0 \left( \forall s \in [0,n], X_s \leq 3 \log(s+1) + K + H/2, \exists s \in [0,n], X_s > K + H \right).
\end{align*}
If $H$ is larger than $6 \log(n+1)$, then the probability on the right hand side vanishes and we directly obtain \eqref{eq:lem_BM_estimates_2}. Let us now assume that $H \leq 6 \log (n+1)$ and denote $k_0 := \floor{ e^{\frac{H}{6}} - 1} \leq n$ and consider the stopping time $\tau = \inf \{ s >0: X_s > K + H \}$. By Markov property, the last probability written above is at most equal to
\begin{align*}
& \sum_{k=k_0}^{n-1} \Pb_0 \left( k \leq \tau < k+1, \forall s \in (\tau,n), X_s \leq (c+1) \log(s+1) + K + H/2 \right) \\
& \leq \sum_{k=k_0}^{n-1} \Pb_0 \left( k \leq \tau < k+1 \right) \\
& \times \Pb_{K+H} \left( \forall s \in [0,n-k-1], X_s \leq 3 \log (s+1) + 2 \log(k+1) + K+ H/2 \right) \\
& \lesssim \sum_{k=k_0}^{n-1} \Pb_0 \left( k \leq \tau < k+1 \right) \log(k+1)^2 / \sqrt{n-k}
\end{align*}
by Lemma \ref{lem:BM_barrier_intermediate}. Now, using the explicit density of $\tau$ (which is a consequence of the reflection principle), we see that
\[
\Pb_0 \left( k \leq \tau < k+1 \right)
= \int_k^{k+1} \frac{K+H}{\sqrt{2\pi t^3}} \exp \left( - \frac{(K+H)^2}{2t} \right)
\lesssim \frac{K+H}{(k+1)^{3/2}}.
\]
Hence,
\begin{align*}
& \Pb_0 \left( \forall k =0 \dots n-1, \min_{[k,k+1]} X \leq 2 \log(k+1) + K, \exists s \in [0,n], X_s > K + H \right) \\
& \lesssim e^{-H^2/64} K^2 \frac1{ \sqrt{n} }
+ (K+H) \sum_{k=k_0}^{n-1} \frac{\log(k+1)^2}{\sqrt{k^3(n-k)}}.
\end{align*}
The behaviour of the above sum is given by
\[
\sum_{k=k_0}^{n-1} \frac{\log(k+1)^2}{\sqrt{k^3(n-k)}}
\lesssim \frac{1}{\sqrt{n}} \sum_{k=k_0}^{\floor{n/2}} \frac{\log(k)^2}{k^{3/2}} + \frac{\log(n)^2}{n^{3/2}} \sum_{k=\floor{n/2}+1}^{n-1} \frac{1}{\sqrt{n-k}}
\lesssim \frac{1}{\sqrt{n}} \left( \frac{\log(k_0)^2}{\sqrt{k_0}} + \frac{\log(n)^2}{\sqrt{n}} \right).
\]
By recalling that $k_0 = \floor{ e^{\frac{H}{6}} - 1} \leq n$, we have therefore obtained that
\begin{align*}
& \Pb_0 \left( \forall k =0 \dots n-1, \min_{[k,k+1]} X \leq 2 \log(k+1) + K, \exists s \in [0,n], X_s > K + H \right) \\
& \lesssim e^{-H^2/64} K^2 \frac1{ \sqrt{n} }
+ (K+H) H^2 e^{-H/12} \frac1{ \sqrt{n} }.
\end{align*}
This concludes the proof of \eqref{eq:lem_BM_estimates_2}.

We now turn to the proof of \eqref{eq:lem_BM_estimates_3}. This time we define for $n \geq 1$, 
\[
q_n := \sup_{0 \leq t_0 < 1} \Pb_0 \left( \forall k \leq n-1, \min_{[k+t_0,k+1+t_0]} X \leq (2-\gamma) k + 2 \log (1+k) + K \right).
\]
By considering for $0 \leq t_0 < 1$, the stopping time
\[
\inf \left\{ s > t_0: X_s > (2-\gamma) s + 3 \log (1+s) + 2K \right\},
\]
we can show using a reasoning very similar to the one above that
\begin{align*}
q_n \leq \Pb_0 \left( \forall s \in [1,n], X_s \leq (2-\gamma)s + 3 \log(1+s) + 2K \right) + \sum_{k=0}^{n-1} e^{-(\log(1+k)+K)^2/16} q_{n-k-1}.
\end{align*}
Take $n \geq (2-\gamma)^{-4}$. By \eqref{eq:lem_BM_barrier_intermediate_2}, the first right hand side term above is at most $CK^2 (2-\gamma)$. Moreover, for all $k \in [n/2,n]$,
\[
q_k - q_n \leq \Pb_0 \left( \exists s \geq n/2, X_s > (2-\gamma) s \right) \leq \Pb_0 \left( \exists s \geq n/2, X_s > 2^{-1/4} s^{3/4} \right) \lesssim e^{-\sqrt{n}/4} \lesssim 2-\gamma.
\]
Therefore,
\begin{align*}
q_n & \lesssim K^2 (2-\gamma) + \sum_{k=0}^{\floor{n/2}} e^{-(\log(k+1)+K)^2/16} q_n + \sum_{k=\floor{n/2}+1}^{n-1} e^{-(\log(k+1)+K)^2/16} \\
& \lesssim K^2(2-\gamma) + e^{-K^2/16} q_n
\end{align*}
which shows that $q_n \lesssim K^2 (2-\gamma)$ as soon as $K$ is large enough.
This finishes the proof of \eqref{eq:lem_BM_estimates_3}. \eqref{eq:lem_BM_estimates_4} follows from \eqref{eq:lem_BM_estimates_3} in a similar manner that \eqref{eq:lem_BM_estimates_2} follows from \eqref{eq:lem_BM_estimates_1}. This concludes the proof.
\end{proof}

\section[Adding good events: proof of Proposition \ref{prop:first_layer_good_event} and Lemma \ref{lem:bad_points}]{Adding good events: proof of Proposition \ref{prop:first_layer_good_event} \\ and Lemma \ref{lem:bad_points}}
\label{sec:adding_good_events}

The purpose of this section is to prove Proposition \ref{prop:first_layer_good_event} and Lemma \ref{lem:bad_points}. We start by discussing Proposition \ref{prop:first_layer_good_event}. As mentioned in Section \ref{subsec:proof_outline}, it is natural to expect the introduction of the good events $G_\eps(x)$ to be harmless. Indeed, in analogy with the case of log-correlated Gaussian fields (see \cite[Corollary 2.4]{powell2018} for instance), the following should hold true:
\begin{equation}
\label{eq:supremum_local_times}
\sup_{x \in D} \sup_{\eps > 0} \left( \sqrt{\frac{1}{\eps} L_{x,\eps}(\tau_{x,R})} - 2 \log \frac{1}{\eps} \right) < \infty
\quad \quad \P_{x_0}-\mathrm{a.s.}
\end{equation}
which would imply (forgetting about the Bessel bridges) that
$
\PROB{x_0}{\bigcap_{x \in D} \bigcap_{\eps>0} G_\eps(x) } \to 1
$
as $\beta \to \infty$. We have not been able to prove such a statement because of the following two main reasons. 

1) For a fixed radius $\eps$, we would like to be able to compare
\begin{equation}
\label{eq:comparison}
\sup_{x \in D} \sqrt{\frac{1}{\eps} L_{x,\eps}(\tau_{x,R})}
\quad \mathrm{and} \quad
\sup_{x \in \eps \Z^2 \cap D} \sqrt{\frac{1}{\eps} L_{x,\eps}(\tau_{x,R})},
\end{equation}
the latter supremum being a supremum over a finite number of elements. To do so, we would need to be able to precisely control the way the local times vary with respect to the centre of the circle. Obtaining estimates precise enough turns out to be difficult to achieve (the estimates of Section C of \cite{jegoGMC} leading to the continuity of the local time process $(x,\eps) \mapsto L_{x,\eps}(\tau)$ are too rough). We resolve this problem by first considering local time of annuli rather than circles. Indeed, comparing local times of annuli is much easier since if an annulus is included in another one, then the local time of the former is not larger than the local time of the latter. 

2) Assuming that we are able to make the comparison \eqref{eq:comparison}, the next step would be to be able to bound from above
\[
\PROB{x_0}{ \sup_{x \in \eps \Z^2 \cap D} \sqrt{\frac{1}{\eps} L_{x,\eps}(\tau_{x,R})} \geq 2 \log \frac{1}{\eps} }.
\]
If the bound is good enough, Borel-Cantelli lemma would allow us to conclude the proof of \eqref{eq:supremum_local_times}, at least along dyadic radii $\eps$. Estimating accurately this probability is again challenging (a union bound is not good enough for instance). In the case of log-correlated Gaussian fields, the estimation of such probabilities is heavily based on the Gaussianity of the process. For instance, in \cite{DRSV2014a}, Kahane's convexity inequalities allow the authors to import computations from cascades (Theorem 1.6 of \cite{HuShi2009}).
We resolve this problem by asking the local times to stay under $2 \log \frac{1}{\eps} + 2 \log \log \frac{1}{\eps}$ instead of $2\log \frac{1}{\eps}$. Indeed, here we can do very naive computations using for instance union bounds. Importantly, this restriction is enough to turn the variables that we consider bounded in $L^1$. We can then make $L^1$ computations and use repulsion estimates to get rid of the extra $2 \log \log \frac{1}{\eps}$ term.

\subsection{Supremum of local times of annuli}

\begin{lemma}
\label{lem:sup_local_time_annulus}
For $x \in D$ and $\eps>0$, let
\[
\ell_{x,\eps}(\tau_{x,R}) := \int_0^{\tau_{x,R}} \indic{\eps \leq |B_t - x| \leq e \eps} dt
= \int_\eps^{e\eps} L_{x,r}(\tau_{x,R}) dr
\]
be the amount of time the Brownian trajectory has spent in the annulus $D(x,e \eps) \backslash D(x,\eps)$ before hitting $\partial D(x,R)$. Then,
\[
\sup_{\eps \in \{e^{-n}, n \geq 1\}} \sup_{\substack{x \in D\\|x-x_0| \geq e \eps}} \sqrt{ \frac{2}{(e^2-1)\eps^2} \ell_{x,\eps}(\tau_{x,R}) } - 2 \log \frac{1}{\eps} - 2 \log \log \frac{1}{\eps} < \infty \quad \quad \P_{x_0}-\mathrm{a.s.}
\]
\end{lemma}

\begin{proof}[Proof of Lemma \ref{lem:sup_local_time_annulus}]
For $x \in D$ and $\eps>0$, define
\[
\ell_{x,\eps} := \int_0^{\tau_{x,eR}} \indic{ \eps - \frac{\eps}{|\log \eps|} \leq |B_t - x| \leq e \eps + \frac{\eps}{|\log \eps|} } dt
\]
and notice that if $|x-y| \leq \eps/|\log \eps|$, then $\ell_{x,\eps}(\tau_{x,R}) \leq \ell_{y,\eps}$ $\P_{x_0}$-a.s. Hence
\begin{align*}
& \sup_{\eps \in \{e^{-n}, n \geq 1\}} \sup_{\substack{x \in D\\|x-x_0| \geq e \eps}} \sqrt{ \frac{2}{(e^2-1)\eps^2} \ell_{x,\eps}(\tau_{x,R}) } - 2 \log \frac{1}{\eps} - 2 \log \log \frac{1}{\eps} \\
& \leq \sup_{\eps \in \{e^{-n}, n \geq 1\}} \sup_{\substack{x \in \frac{\eps}{|\log \eps|} \Z^2 \cap D \\ |x-x_0| \geq e \eps + \eps/|\log \eps|}} \sqrt{ \frac{2}{(e^2-1)\eps^2} \ell_{x,\eps} } - 2 \log \frac{1}{\eps} - 2 \log \log \frac{1}{\eps}
\end{align*}
$\P_{x_0}$-a.s. By Borel-Cantelli lemma, to conclude the proof it is now enough to show that
\[
\sum_{\eps \in \{e^{-n}, n \geq 1\}} \PROB{x_0}{ \sup_{\substack{x \in \frac{\eps}{|\log \eps|} \Z^2 \cap D \\ |x-x_0| \geq e \eps + \eps/|\log \eps|}} \sqrt{ \frac{2}{(e^2-1)\eps^2} \ell_{x,\eps} } - 2 \log \frac{1}{\eps} - 2 \log \log \frac{1}{\eps} \geq 0 } < \infty.
\]
After a union bound, we want to estimate
\[
\PROB{x_0}{ \sqrt{ \frac{2}{(e^2-1)\eps^2} \ell_{x,\eps} } - 2 \log \frac{1}{\eps} - 2 \log \log \frac{1}{\eps} \geq 0 }
\]
for a given $\eps \in \{e^{-n}, n \geq 1\}$ and $x \in \frac{\eps}{|\log \eps|} \Z^2 \cap D$ such that $|x-x_0| \geq e \eps + \eps/|\log \eps|$. Let $z \in \partial D(x,e\eps + \eps/|\log \eps|)$. By \eqref{eq:local_times_and_Bessel_process}, starting from $z$ and conditioned on
\[
\ell := \frac{1}{e \eps + \eps/|\log \eps|} L_{x,e \eps + \eps/|\log \eps|}(\tau_{x,eR}),
\]
\begin{align*}
\ell_{x,\eps} = \int_{\eps - \eps/|\log \eps|}^{e\eps + \eps/|\log \eps|} L_{x,r}(\tau_{x,eR}) dr
\overset{(d)}{=} \left( e \eps + \frac{\eps}{|\log \eps|} \right)^2 \int_0^{\log \frac{e\eps + \eps/|\log \eps|}{\eps - \eps/|\log \eps|}} e^{-2s} X_s^2 ds
\end{align*}
where $X_s$ is a zero-dimensional Bessel process starting at $\sqrt{\ell}$. By bounding
\[
\log \frac{e\eps + \eps/|\log \eps|}{\eps - \eps/|\log \eps|} \leq 1 + \frac{3}{ |\log \eps|}, \quad
\left( e \eps + \frac{\eps}{|\log \eps|} \right)^2 \leq e^2 \eps^2 \left( 1 + \frac{1}{|\log \eps|} \right)
\]
(if $\eps$ is small enough) and 
\[
\frac{2}{1-e^{-2}} \left(1+\frac{1}{|\log \eps|} \right) \int_0^{1 + 3/|\log \eps|} e^{-2s} X_s^2 ds
\leq \left( 1 + \frac{2}{|\log \eps|} \right) \max_{s \leq 1+3/|\log \eps|} X_s^2
\]
we deduce that
\begin{align*}
& \PROB{z}{ \sqrt{ \frac{2}{(e^2-1)\eps^2} \ell_{x,\eps} } - 2 \log \frac{1}{\eps} - 2 \log \log \frac{1}{\eps} \geq 0 } \\
& \leq \EXPECT{z}{ \Pb^0_{\sqrt{\ell}} \left( \left(1+\frac{2}{|\log \eps|} \right) \max_{s \leq 1+3/|\log \eps|} X_s^2 \geq \left(2 \log \frac{1}{\eps} + 2 \log \log \frac{1}{\eps} \right)^2 \right) }.
\end{align*}
Since $(X_s, s \geq 0)$ under $\Pb_{\sqrt{\ell}}^0$ is stochastically dominated by $(X_s, s \geq 0)$ under $\Pb_{\sqrt{\ell}}$ (zero-dimensional Bessel process has a negative drift), we obtain that
\begin{align*}
& \Pb^0_{\sqrt{\ell}} \left( \left(1+\frac{2}{|\log \eps|} \right) \max_{s \leq 1+3/|\log \eps|} X_s^2 \geq \left(2 \log \frac{1}{\eps} + 2 \log \log \frac{1}{\eps} \right)^2 \right) \\
& \leq \Pb^0_{\sqrt{\ell}} \left( \max_{s \leq 1+3/|\log \eps|} X_s \geq 2 \log \frac{1}{\eps} + 2 \log \log \frac{1}{\eps} - 3 \right) \\
& \leq \Pb_0 \left( \max_{s \leq 1+3/|\log \eps|} X_s \geq 2 \log \frac{1}{\eps} +2 \log \log \frac{1}{\eps} - 3 - \sqrt{\ell} \right) \\
& \leq \indic{\sqrt{\ell} \geq 2 \log \frac{1}{\eps} + 2 \log \log \frac{1}{\eps} - 3} \\
& + 2 \times \indic{\sqrt{\ell} < 2 \log \frac{1}{\eps} + 2 \log \log \frac{1}{\eps} - 3} \exp \left( - \frac{1}{2(1+3/|\log \eps|)} \left( 2 \log \frac{1}{\eps} + 2 \log \log \frac{1}{\eps} - 3 - \sqrt{\ell} \right)^2 \right).
\end{align*}
We used reflection principle in the last inequality.
Recalling that under $\P_z$ $\ell$ is an exponential variable with mean equal to $2 |\log \eps| + O(1)$ (see \eqref{eq:lem_Green_estimate_circle}), we see that
\[
\PROB{z}{\sqrt{\ell} \geq 2 \log \frac{1}{\eps} + 2 \log \log \frac{1}{\eps} - 3} \lesssim \eps^2 |\log \eps|^{-4}.
\]
Moreover, by denoting $A := \frac{2 |\log \eps| + 2 \log |\log \eps| - 3}{\sqrt{\EXPECT{z}{\ell}}}$ and $\lambda = \frac{\EXPECT{z}{\ell}}{2(1+3/|\log \eps|)}$, we have
\begin{align*}
&
\EXPECT{z}{ \indic{\sqrt{\ell} < 2 \log \frac{1}{\eps} + 2 \log \log \frac{1}{\eps} - 3} \exp \left( - \frac{1}{2(1+3/|\log \eps|)} \left( 2 \log \frac{1}{\eps} + 2 \log \log \frac{1}{\eps} - 3 - \sqrt{\ell} \right)^2 \right)} \\
& \leq
\int_0^{A^2} e^{- \lambda(A-\sqrt{t})^2 } e^{-t} dt
=  2 e^{-\lambda A^2 /(\lambda+1)} \int_{-\infty}^{A/\sqrt{\lambda +1}} e^{-u^2} \max \left(0, \frac{u}{\sqrt{\lambda+1}} + \frac{\lambda}{\lambda+1} A \right) du \\
& \lesssim A  e^{-\lambda A^2 /(\lambda+1)}
\lesssim \sqrt{|\log \eps|} \eps^2 |\log \eps|^{-4}.
\end{align*}
Wrapping things up, we have proven that
\[
\PROB{x_0}{ \sqrt{ \frac{2}{(e^2-1)\eps^2} \ell_{x,\eps} } - 2 \log \frac{1}{\eps} - 2 \log \log \frac{1}{\eps} \geq 0 }
\lesssim |\log \eps|^{-7/2} \eps^2
\]
and summing over $x \in \frac{\eps}{|\log \eps|} \Z^2 \cap D$, $|x-x_0| \geq e \eps + \eps/|\log \eps|$,
\[
\PROB{x_0}{ \sup_{\substack{x \in \frac{\eps}{|\log \eps|} \Z^2 \cap D \\ |x-x_0| \geq e \eps + \eps/|\log \eps|}} \sqrt{ \frac{2}{(e^2-1)\eps^2} \ell_{x,\eps} } - 2 \log \frac{1}{\eps} - 2 \log \log \frac{1}{\eps} \geq 0 }
\lesssim |\log \eps|^{-3/2}.
\]
This is summable over $\eps \in \{ e^{-n}, n \geq 1 \}$ as required. It concludes the proof.
\end{proof}

\subsection{First layer of good events: proof of Proposition \ref{prop:first_layer_good_event}}

We now have all the ingredients to prove Proposition \ref{prop:first_layer_good_event}. During the course of the proof, we will obtain intermediate results that we gather in the following lemma. Recall the definition \eqref{eq:def_eps_gamma} of $\eps_\gamma$.

\begin{lemma}\label{lem:first_moment_measures}
Firstly,
\begin{equation}
\label{eq:lem_first_moment_measures_infinity}
\lim_{\beta \to \infty} \lim_{\eps \to 0} \EXPECT{x_0}{\hat{m}_\eps(D)} = \infty.
\end{equation}
Secondly, we have for $\beta >0$ fixed,
\begin{equation}
\label{eq:lem_first_moment_measures_1}
\sup_{\eps >0 } \EXPECT{x_0}{ \hat{m}_\eps(D) }
\leq \int_D \sup_{\eps >0} \EXPECT{x_0}{ \hat{m}_\eps(dx) } < \infty,
\end{equation}
\begin{equation}
\label{eq:lem_first_moment_measures_2}
\sup_{\eps >0 } \EXPECT{x_0}{ \hat{\mu}_\eps(D) }
\leq \int_D \sup_{\eps >0} \EXPECT{x_0}{ \hat{\mu}_\eps(dx) } < \infty
\end{equation}
and 
\begin{equation}
\label{eq:lem_first_moment_measures_3}
\sup_{\gamma \in [1,2)} (2-\gamma)^{-1} \sup_{\eps < \eps_\gamma } \EXPECT{x_0}{ \hat{m}_\eps^\gamma(D) }
\leq \int_D \sup_{\gamma \in [1,2)} (2-\gamma)^{-1} \sup_{\eps < \eps_\gamma } \EXPECT{x_0}{ \hat{m}_\eps^\gamma(dx) }  < \infty.
\end{equation}
\end{lemma}

\begin{proof}[Proof of Proposition \ref{prop:first_layer_good_event} and Lemma \ref{lem:first_moment_measures}]

Let $\beta'>0$ be large. In light of Lemma \ref{lem:sup_local_time_annulus} we introduce for all $\eps = e^{-k} >0$ and $x \in D$ at distance at least $e \eps$ from $x_0$, the good event
\begin{align*}
H_\eps(x) := \Bigg\{ & \forall n = k_x + 1 \dots k, \sqrt{ \frac{2}{(e^2-1)} e^{2n} \ell_{x,e^{-n}}(\tau_{x,R}) } - 2 n - 2 \log n \leq \beta' \Bigg\}
\end{align*}
and set
\[
H := \bigcap_{x \in D} \bigcap_{\eps >0} H_\eps(x).
\]
Lemma \ref{lem:sup_local_time_annulus} asserts that $\PROB{x_0}{ H } \to 1$ as $\beta' \to \infty$.

\paragraph*{Seneta--Heyde norming.}

We are first going to show that for a fixed $\beta' >0$,
\begin{equation}
\label{eq:proof_prop_first_layer_a}
\int_A \sup_{\eps >0} |\log \eps| \eps^2 \EXPECT{x_0}{e^{2 \sqrt{\frac{1}{\eps} L_{x,\eps}(\tau)}} \mathbf{1}_{H_\eps(x)}} dx < \infty.
\end{equation}
First of all, if $|x-x_0| < 1/|\log \eps|$, then we simply bound
\[
|\log \eps| \eps^2 \EXPECT{x_0}{e^{2 \sqrt{\frac{1}{\eps} L_{x,\eps}(\tau)}} \mathbf{1}_{H_\eps(x)}}
\leq |\log \eps| \eps^2 \EXPECT{x_0}{e^{2 \sqrt{\frac{1}{\eps} L_{x,\eps}(\tau_{x,R})}} } \lesssim |\log \eps|^{3/2}
\]
by \eqref{eq:lem_first_moment}. Take now $x \in D$ at distance at least $1/|\log \eps|$ from $x_0$. We again bound $L_{x,\eps}(\tau)$ by $L_{x,\eps}(\tau_{x,R})$ to be able to use the link \eqref{eq:local_times_and_Bessel_process} between local times and zero-dimensional Bessel process:
\begin{align*}
|\log \eps| \eps^2 \EXPECT{x_0}{e^{2 \sqrt{\frac{1}{\eps} L_{x,\eps}(\tau)}} \mathbf{1}_{H_\eps(x)}}
\leq |\log \eps| \eps^2 \EXPECT{x_0}{e^{2 \sqrt{\frac{1}{\eps} L_{x,\eps}(\tau_{x,R})}} \mathbf{1}_{H_\eps(x)}}.
\end{align*}
Denote by $r_x := \sqrt{e^{k_x} L_{x,e^{-k_x}}(\tau_{x,R})}$. 
\eqref{eq:local_times_and_Bessel_process} tells us that, conditionally on $r_x$, the process
\[
X_s := \sqrt{ e^{k_x +s} L_{x,e^{-k_x-s}}(\tau_{x,R}) }, \quad s \geq 0,
\]
is a zero-dimensional Bessel process starting at $r_x$. The event $H_\eps(x)$ requires
\begin{align*}
\min_{u \in [s-1,s]} X_u
& \leq \left( \frac{2}{e^2 - 1} \int_{s-1}^s e^{2s-2u} X_u^2 du \right)^{1/2} \\
& = \left( \frac{2}{e^2 - 1} e^{2(k_x+s)} \int_{s-1}^s e^{-k_x -u} L_{x, e^{-k_x-u}}(\tau_{x,R}) du \right)^{1/2} \\
& = \left( \frac{2}{e^2 - 1} e^{2(k_x+s)} \int_{e^{-k_x-s}}^{e^{-k_x-s+1}} L_{x,\delta}(\tau_{x,R}) d \delta \right)^{1/2} \\
& \leq 2 s + 2 k_x + 2 \log(s + k_x) + \beta' \leq 2s + 2\log s + \beta' + 4 k_x
\end{align*}
for all $s= 1 \dots k-k_x$. Hence
\begin{align*}
& |\log \eps| \eps^2 \EXPECT{x_0}{e^{2 \sqrt{\frac{1}{\eps} L_{x,\eps}(\tau_{x,R})}} \mathbf{1}_{H_\eps(x)}} \\
& \leq |\log \eps| \eps^2 \E_{x_0} \Bigg[ \Eb^0_{r_x} \Bigg[ e^{2X_{k-k_x}}
\indic{ \forall s = 1 \dots k-k_x, \min_{u \in [s-1,s]} X_u \leq 2s + 2\log s + \beta' + 4 k_x } \Bigg] \Bigg].
\end{align*}
Now, with \eqref{eq:lem_Bessel0_BM}, we have
\begin{align*}
& |\log \eps| \eps^2 \EXPECT{x_0}{e^{2 \sqrt{\frac{1}{\eps} L_{x,\eps}(\tau_{x,R})}} \mathbf{1}_{H_\eps(x)}}
\leq |\log \eps| \eps^2 + \frac{k}{\sqrt{k-k_x}} e^{-2k_x} \\
& \times \EXPECT{x_0}{e^{2r_x} \sqrt{r_x} \Eb_{r_x} \Bigg[ \left( \frac{k-k_x}{X_{k-k_x} + 2(k-k_x)} \right)_+^{1/2} \indic{ \forall s= 1 \dots k-k_x, \min_{u \in [s-1,s]} X_u \leq 2 \log s + \beta' + 4k_x + 2 }  \Bigg] }.
\end{align*}
We now bound 
\begin{align*}
& \Eb_{r_x} \Bigg[ \left( \frac{k-k_x}{X_{k-k_x} + 2(k-k_x)} \right)_+^{1/2} \indic{ \forall s= 1 \dots k-k_x, \min_{u \in [s-1,s]} X_u \leq 2 \log s + \beta' + 4k_x + 2 }  \Bigg] \\
& \leq \Pb_{r_x} \left( \forall s = 1 \dots k-k_x, \min_{u \in [s-1,s]} X_u \leq 2 \log s + \beta' + 4k_x + 2 \right) \\
& + \Eb_{r_x} \Bigg[ \left( \frac{k-k_x}{X_{k-k_x} + 2(k-k_x)} \right)_+^{1/2} \indic{ X_{k-k_x} \leq -(k-k_x) }  \Bigg].
\end{align*}
By \eqref{eq:lem_BM_estimates_1}, the first right hand side term is at most $C (k_x)^2 k^{-1/2}$. The second right hand side term decays much faster and we have obtained
\[
|\log \eps| \eps^2 \EXPECT{x_0}{e^{2 \sqrt{\frac{1}{\eps} L_{x,\eps}(\tau_{x,R})}} \mathbf{1}_{H_\eps(x)}}
\lesssim |\log \eps| \eps^2 + (k_x)^2 e^{-2k_x} \EXPECT{x_0}{e^{2r_x} \sqrt{r_x} } \lesssim (k_x)^3
\]
where we have used \eqref{eq:lem_first_moment} in the last inequality (or more precisely, the stochastic domination stated in Lemma \ref{lem:first_moment} in order to also handle $\sqrt{r_x}$). To wrap things up, we have proven that
\[
|\log \eps| \eps^2 \EXPECT{x_0}{e^{2 \sqrt{\frac{1}{\eps} L_{x,\eps}(\tau_{x,R})}} \mathbf{1}_{H_\eps(x)}}
\lesssim \left\{ \begin{array}{cc}
|\log \eps|^{3/2} & \mathrm{if~} |x-x_0| \leq 1/|\log \eps| \\
|\log |x-x_0||^3 & \mathrm{if~} |x-x_0| \geq 1 / |\log \eps|
\end{array}
\right.
\]
which concludes the proof of \eqref{eq:proof_prop_first_layer_a}. Very few arguments need to be changed in order to show \eqref{eq:lem_first_moment_measures_1}. The only difference is that, compared to the event $H_\eps(x)$, the event $G_\eps(x)$ ensures (in particular) the Bessel process $X$ to stay below $s \mapsto 2s + \beta + 2k_x $ at every integer $s$. This is more restrictive than asking $\min_{[s,s+1]} X$ to be not larger than $2s + 2 \log s + \beta + 4k_x$, we can thus conclude using the reasoning above.

We now turn to the proof of \eqref{eq:lem_first_layer_good_event_a}. Fix $\beta' >0$. We are going to show that
\begin{equation}
\label{eq:proof_prop_first_layer_d}
\sup_{\eps >0} |\log \eps| \eps^2 \int_A \EXPECT{x_0}{e^{2 \sqrt{\frac{1}{\eps}L_{x,\eps}(\tau)}} \mathbf{1}_{H_\eps(x)} \mathbf{1}_{G_\eps(x)^c} } dx
\end{equation}
goes to zero as $\beta \to \infty$.
Let $\eta_0 >0$ be small. By \eqref{eq:proof_prop_first_layer_a},
\[
\sup_{\eps >0} |\log \eps| \eps^2 \int_A \indic{|x-x_0| \leq \eta_0} \EXPECT{x_0}{e^{2 \sqrt{\frac{1}{\eps}L_{x,\eps}(\tau)}} \mathbf{1}_{H_\eps(x)} } dx = o_{\eta_0 \to 0}(1).
\]
Fix now $\eta_0>0$. In what follows the constants underlying the bounds may depend on $\eta_0$.
Recall the definition of $h_{x,\delta}$ constructed in Lemma \ref{lem:process_h}.
By a reasoning very similar to what we did above and using \eqref{eq:lem_BM_estimates_2}, one can show that
\[
\sup_{\substack{\eps=e^{-k}\\k \geq 1}} |\log \eps| \eps^2 \int_A \indic{|x-x_0| > \eta_0} \EXPECT{x_0}{e^{2 \sqrt{\frac{1}{\eps}L_{x,\eps}(\tau_{x,R})}} \mathbf{1}_{H_\eps(x)} \indic{\exists s \in \{k_x, \dots, k\}, h_{x,e^{-s}} \geq 2s + \beta/2} } dx
\]
goes to zero as $\beta \to \infty$. We are thus left to control
\[
|\log \eps| \eps^2 \EXPECT{x_0}{e^{2 \sqrt{\frac{1}{\eps}L_{x,\eps}(\tau_{x,R})}} \indic{\forall s \in \{k_x,\dots,k \}, h_{x,e^{-s}} < 2s + \beta/2} \mathbf{1}_{G_\eps(x)^c} }
\]
for some $x \in D$ at distance at least $\eta_0$ from $x_0$.
Denote $r_x = \sqrt{ e^{-k_x} L_{x,e^{-k_x}}(\tau_{x,R})}$.
By \eqref{eq:local_times_and_Bessel_process} and then by \eqref{eq:lem_Bessel0_BM_bound}, this is equal to 
\begin{align*}
& |\log \eps| \eps^2 \EXPECT{x_0}{ \Eb_{r_x}^0 \left[ e^{2X_t} \indic{\forall s=0 \dots k-k_x, X_s < 2s + \beta/2 + 2k_x} \indic{\exists s \leq k-k_x, X_s \geq 2s + \beta + 2k_x} \right] } \\
& \lesssim \sqrt{k} \E_{x_0} \Bigg[ \sqrt{r_x} e^{2r_x} \Eb_{r_x} \Bigg[ \left( \frac{k-k_x}{X_{k-k_x} + 2(k-k_x)} \right)_+^{1/2} \\
& \times \indic{\forall s= 0 \dots k-k_x, X_s < \beta/2 + 2k_x} \indic{\exists s \leq k-k_x, X_s \geq \beta + 2k_x} \Bigg] \Bigg] \\
& \lesssim \sqrt{k} \EXPECT{x_0}{ \sqrt{r_x} e^{2r_x} \Pb_r \left( \forall s = 0 \dots k-k_x, X_s < \beta/2 + 2k_x, \exists s \leq k-k_x, X_s \geq \beta + 2k_x \right) } \\
& \lesssim \beta^2 e^{-\beta/256}
\end{align*}
by \eqref{eq:lem_BM_estimates_2}. This concludes the proof of \eqref{eq:proof_prop_first_layer_d}.
We now have for any small $\rho>0$,
\begin{align*}
& \limsup_{\eps \to 0} \PROB{x_0}{ \abs{m_\eps(A) - \hat{m}_\eps(A)} \geq \rho} \\
& \leq \PROB{x_0}{H^c} + \frac{1}{\rho} \limsup_{\eps \to 0} |\log \eps| \eps^2 \int_A \EXPECT{x_0}{e^{2 \sqrt{\frac{1}{\eps}L_{x,\eps}(\tau)}} \mathbf{1}_{H_\eps(x)} \mathbf{1}_{G_\eps(x)^c} } dx.
\end{align*}
By letting $\beta \to \infty$ and then $\beta' \to \infty$, we see that
\[
\limsup_{\beta \to \infty} \limsup_{\eps \to 0}
\PROB{x_0}{ \abs{m_\eps(A) - \hat{m}_\eps(A)} \geq \rho} = 0
\]
as desired in \eqref{eq:lem_first_layer_good_event_a}.

To show \eqref{eq:lem_first_moment_measures_infinity}, take $r>0$ small enough so that $\{x \in D: D(x,r) \subset D\}$ has positive Lebesgue measure and notice that
\begin{align*}
\EXPECT{x_0}{\hat{m}_\eps(D)}
\geq |\log \eps| \eps^2 \int_D \indic{D(x,r) \subset D} \EXPECT{x_0}{e^{2\sqrt{\frac{1}{\eps} L_{x,\eps}(\tau_{x,r})}} \indic{\forall s \in [k_x,k], h^r_{x,e^{-s}} \leq 2s + \beta} } dx
\end{align*}
where $h^r$ is defined in a similar manner as $h$ expect that we consider local times up to time $\tau_{x,r}$ rather than $\tau_{x,R}$. Using \eqref{eq:local_times_and_Bessel_process}, we see that \eqref{eq:lem_first_moment_measures_infinity} is a direct consequence of \eqref{eq:lem_Bessel_limit} and Fatou's lemma.

\paragraph*{Subcritical measures}

We have finished the part of the proof concerning the Seneta--Heyde normalisation and we now turn to the justification of \eqref{eq:lem_first_layer_good_event_c} and \eqref{eq:lem_first_moment_measures_3}.
This is very similar to what we have just done.
The only difference is that after using the link \eqref{eq:local_times_and_Bessel_process} between local times and zero-dimensional Bessel process and the relation \eqref{eq:lem_Bessel0_BM} to transfer computations to 1D Brownian motion, we have
\begin{align*}
& \frac{1}{2-\gamma} \sqrt{|\log \eps|} \eps^{\gamma^2/2} \EXPECT{x_0}{e^{2 \sqrt{\frac{1}{\eps} L_{x,\eps}(\tau_{x,R})}} \mathbf{1}_{H_\eps(x)}} \\
& \leq \frac{1}{2-\gamma} \sqrt{|\log \eps|} \eps^{\gamma^2/2} + \frac{1}{2-\gamma} e^{-\gamma^2 k_x /2} \E_{x_0} \Bigg[ \sqrt{r_x} e^{\gamma r_x} \\
& \times \Eb_{r_x} \left[ \left( \frac{k-k_x}{X_{k-k_x} + \gamma (k-k_x)}  \right)^{1/2} \indic{ \forall s=1 \dots k-k_x, \min_{u \in [s-1,s]} X_u \leq (2-\gamma)s + 2 \log s + \beta' + 4 k_x + 2} \right] \Bigg].
\end{align*}
We conclude as before by using \eqref{eq:lem_BM_estimates_3} and \eqref{eq:lem_BM_estimates_4} (note here that $k -k_x \geq (2-\gamma)^{-4}$ since $\eps_\gamma$ has been chosen small enough) instead of \eqref{eq:lem_BM_estimates_1} and \eqref{eq:lem_BM_estimates_2}.

\paragraph*{Derivative martingale}

We finish with the justification of \eqref{eq:lem_first_layer_good_event_b} and \eqref{eq:lem_first_moment_measures_2}. Recall that in the modified measure $\hat{\mu}_\eps$, the Brownian motion is stopped either at time $\tau$ or at time $\tau_{x,R}$ depending on whether the local time $L_{x,\eps}$ is in the exponential or not. Part of \eqref{eq:lem_first_layer_good_event_b} consists in saying that, in the limit, this modification does not change the measure with high probability.
We thus start by proving that
\begin{equation}
\label{eq:proof_prop_first_layer_c}
\limsup_{\eps \to 0} \int_D \sqrt{|\log \eps|} \eps^2 \EXPECT{x_0}{
\left( \sqrt{\frac{1}{\eps} L_{x,\eps}(\tau_{x,R})} - \sqrt{\frac{1}{\eps} L_{x,\eps}(\tau)} \right) e^{2 \sqrt{\frac{1}{\eps} L_{x,\eps}(\tau)} } } dx  = 0.
\end{equation}
Let $x \in D$.
By applying Markov property to the first exit time $\tau$ of $D$, the integrand in \eqref{eq:proof_prop_first_layer_c} is at most equal to
\begin{align*}
\sqrt{|\log \eps|} \eps^2 \EXPECT{x_0}{ e^{2 \sqrt{\frac{1}{\eps}L_{x,\eps}(\tau)}} \sup_{z \in \partial D} \EXPECT{z}{ \left. \sqrt{ \frac{1}{\eps} L_{x,\eps}(\tau_{x,R}) + \frac{1}{\eps} L_{x,\eps}(\tau) } - \sqrt{ \frac{1}{\eps} L_{x,\eps}(\tau) } \right\vert L_{x,\eps}(\tau) }  }.
\end{align*}
We decompose this expectation in two parts, the first one integrating on the event that $\sqrt{\frac{1}{\eps}L_{x,\eps}(\tau)} < |\log \eps|/2$ and the second one integrating on the complement event. The first part decays quickly to zero and we explain how to deal with the second part.
Recall that starting from any point of $\partial D(x,\eps)$, $L_{x,\eps}(\tau_{x,R})$ is a random variable with mean $2 \eps \log (R/\eps)$ (see Lemma \ref{lem:Green_function_estimate}). By Cauchy--Schwarz inequality and then by bounding $\sqrt{a+b} - \sqrt{a} \leq C b/\sqrt{a}$ for $a>2b>1$, and using \eqref{eq:lem_hitting_proba}, we thus obtain that on the event that $\sqrt{\frac{1}{\eps}L_{x,\eps}(\tau)} \geq |\log \eps|/2$,
\begin{align*}
& \sup_{z \in \partial D} \EXPECT{z}{ \left. \sqrt{ \frac{1}{\eps} L_{x,\eps}(\tau_{x,R}) + \frac{1}{\eps} L_{x,\eps}(\tau) } - \sqrt{ \frac{1}{\eps} L_{x,\eps}(\tau) } \right\vert L_{x,\eps}(\tau) } \\
& \leq \sup_{z \in \partial D} \PROB{z}{\tau_{x,\eps} < \tau_{x,R}} \left( \sqrt{ \EXPECT{z}{ \left. \frac{1}{\eps} L_{x,\eps}(\tau_{x,R}) \right\vert \tau_{x,\eps} < \tau} + \frac{1}{\eps} L_{x,\eps}(\tau) } - \sqrt{ \frac{1}{\eps} L_{x,\eps}(\tau) } \right) \\
& \lesssim \frac{|\log \d(x,\partial D)|}{|\log \eps|} \left( \sqrt{ \frac{1}{\eps} L_{x,\eps}(\tau) + 2 \log \frac{R}{\eps} } - \sqrt{ \frac{1}{\eps} L_{x,\eps}(\tau) } \right) \\
& \lesssim |\log \d(x,\partial D)| \left( \frac{1}{\eps}L_{x,\eps}(\tau) \right)^{-1/2}
\lesssim \frac{|\log \d(x,\partial D)|}{ |\log \eps| }.
\end{align*}
The integrand in \eqref{eq:proof_prop_first_layer_c} is therefore at most
\[
o(1) +
O(1) \frac{ |\log \d(x,\partial D)|}{\sqrt{ |\log \eps| }} \eps^2 \EXPECT{x_0}{ e^{2 \sqrt{\frac{1}{\eps} L_{x,\eps}(\tau) }} }
\leq o(1) + O(1) \frac{|\log \d(x,\partial D)| |\log |x-x_0|| }{|\log \eps| }
\]
by \eqref{eq:lem_first_moment} and \eqref{eq:lem_hitting_proba}.
This concludes the proof of \eqref{eq:proof_prop_first_layer_c}.

Now, let $\beta >0$. For any small $\rho >0$ and large $\beta'>0$, we have
\begin{align*}
& \limsup_{\eps \to 0} \PROB{x_0}{ \abs{ \mu_\eps(A) - \hat{\mu}_\eps(A) } > \rho }
\leq \PROB{x_0}{H^c} \\
& + \frac{3}{\rho} \limsup_{\eps \to 0} \int_D \sqrt{|\log \eps|} \eps^2 \EXPECT{x_0}{
\left( \sqrt{\frac{1}{\eps} L_{x,\eps}(\tau_{x,R})} - \sqrt{\frac{1}{\eps} L_{x,\eps}(\tau)} \right) e^{2 \sqrt{\frac{1}{\eps} L_{x,\eps}(\tau)} } } dx \\
& + \frac{3}{\rho} \limsup_{\eps \to 0} \int_D \left( 3 \log \log \frac{1}{\eps} + \beta \right) \sqrt{|\log \eps|} \eps^2 \EXPECT{x_0}{ e^{2 \sqrt{\frac{1}{\eps} L_{x,\eps}(\tau)} } \mathbf{1}_{H_\eps(x)} } dx \\
& + \frac{3}{\rho} \limsup_{\eps \to 0} \int_D \sqrt{|\log \eps|} \eps^2 \EXPECT{x_0}{
\abs{ - \sqrt{\frac{1}{\eps} L_{x,\eps}(\tau_{x,R})} + 2 \log \frac{1}{\eps} } e^{2 \sqrt{\frac{1}{\eps} L_{x,\eps}(\tau)} } \mathbf{1}_{H_\eps(x)} \mathbf{1}_{G_\eps(x)^c} } dx.
\end{align*}
\eqref{eq:proof_prop_first_layer_c} and \eqref{eq:proof_prop_first_layer_a} tell us that the second and respectively third right hand side terms vanish. When $\beta '>0$ and $\rho >0$ are fixed, one can show using a method very similar to what we did with the Seneta--Heyde normalisation that the last right hand side term goes to zero as $\beta \to \infty$. Hence
\[
\limsup_{\beta \to \infty} \limsup_{\eps \to 0} \PROB{x_0}{ \abs{ \mu_\eps(A) - \hat{\mu}_\eps(A) } > \rho }
\leq \PROB{x_0}{H^c}.
\]
The left hand side term is independent of $\beta'$ whereas the right hand side term goes to zero as $\beta' \to 0$. Therefore, for any small $\rho >0$,
\[
\limsup_{\beta \to \infty} \limsup_{\eps \to 0} \PROB{x_0}{ \abs{ \mu_\eps(A) - \hat{\mu}_\eps(A) } > \rho }
= 0
\]
as desired in \eqref{eq:lem_first_layer_good_event_b}. The proof of \eqref{eq:lem_first_moment_measures_2} is very similar to that of \eqref{eq:lem_first_moment_measures_1}. We omit the details and it concludes the proof.
\end{proof}

\subsection{Second layer of good events: proof of Lemma \ref{lem:bad_points}}

\begin{proof}[Proof of Lemma \ref{lem:bad_points}]
We start by proving \eqref{eq:lem_bad_points_2}. Let $\eta_0 >0$. By Lemma \ref{lem:first_moment_measures}, it is enough to show that
\begin{align}
& \int_D \sqrt{|\log \eps|} \eps^2 \EXPECT{x_0}{ \left( - \sqrt{\frac{1}{\eps} L_{x,\eps}(\tau_{x,R})} + 2 \log \frac{1}{\eps} + \beta \right) e^{2 \sqrt{\frac{1}{\eps} L_{x,\eps}(\tau)} } \mathbf{1}_{G_\eps(x)} \mathbf{1}_{G_\eps'(x)^c}  } \indic{|x-x_0| > \eta_0} dx
\label{eq:proof_lem_bad_points_a}
\end{align}
goes to zero as $\eps \to 0$ and then $M \to \infty$. The constants underlying the following estimates may depend on $\eta_0$. We start off by bounding $L_{x,\eps}(\tau)$ by $L_{x,\eps}(\tau_{x,R})$ in the exponential above.
By letting $t = k - k_x$, $\beta_x = \beta + 2k_x$ and $r = \sqrt{e^{k_x} L_{x,e^{-k_x}}(\tau_{x,R})}$ and by using \eqref{eq:local_times_and_Bessel_process}, we are left to estimate
\begin{align*}
\sqrt{t} e^{-2t} \EXPECT{x_0}{ \Eb_r^0 \left[ (- X_t + 2 t + \beta_x) e^{2X_t} \indic{\forall s \leq t, X_s < 2s + \beta_x, \exists s \leq t, X_s \geq 2s + \beta_x - \frac{\sqrt{s}}{M \log(2+s)^2} } \right] }.
\end{align*}
By \eqref{eq:lem_Bessel0_Bessel3_bound} and then by Cauchy--Schwarz inequality, this is at most
\begin{align*}
& \EXPECT{x_0}{ \sqrt{r} (\beta_x - r) e^{2r} \Eb^3_{\beta_x -r} \left[ \left( \frac{t}{2t + \beta_x - X_t} \right)_+^{1/2} \indic{\exists s \leq t, X_s \leq \frac{\sqrt{s}}{M \log(2+s)^2} } \right] } \\
& \lesssim \EXPECT{x_0}{ \sqrt{r} (\beta_x - r) e^{2r} \Eb^3_{\beta_x -r} \left[ \left( \frac{t}{2t + \beta_x - X_t} \right)_+ \right]^{1/2} \Pb^3_{\beta_x -r} \left( \exists s \geq 0, X_s \leq \frac{\sqrt{s}}{M \log(2+s)^2} \right)^{1/2} }
\end{align*}
which goes to zero as $M \to \infty$ uniformly in $t$ by Lemma \ref{lem:Bessel3_barrier}, Points \ref{lemlem:barrier} and \ref{lemlem:1/X_t^2bis}. We have thus proven that the contribution of points at distance at least $\eta_0$ from $x_0$ to the integral \eqref{eq:proof_lem_bad_points_a} goes to zero as $\eps \to 0$ and then $M \to 0$. This concludes the proof of \eqref{eq:lem_bad_points_2}.

The proof of \eqref{eq:lem_bad_points_1} is very similar: the presence of an extra $\sqrt{|\log \eps|}$ in the normalisation as well as the absence of the derivative term $(-X_t + 2t + \beta)$ makes an extra multiplicative term $\sqrt{t}/X_t$ popping up in the expectation with respect to the 3D Bessel process. We conclude as before using Cauchy--Schwarz inequality and Lemma \ref{lem:Bessel3_barrier}, Point \ref{lemlem:1/X_t^2}.

We finish with the proof of \eqref{eq:lem_bad_points_subcritical}. With the same notations as above, it is again enough to estimate
\[
(2-\gamma)^{-1}\sqrt{t} e^{-\frac{\gamma^2}{2}t} \EXPECT{x_0}{ \Eb_r^0 \left[ e^{\gamma X_t} \indic{\forall s \leq t, X_s < 2s + \beta_x, \exists s \leq t, X_s \geq 2s +\beta_x - \frac{\sqrt{s}}{M \log(2+s)^2} } \right] }.
\]
By \eqref{eq:lem_Bessel0_BM_bound}, this is at most
\begin{align}
\nonumber
& (2-\gamma)^{-1} \EXPECT{x_0}{ \sqrt{r} e^{\gamma r} \Eb_r \left[ \left( \frac{t}{X_t + \gamma t} \right)_+^{1/2} \indic{\forall s \leq t, X_s < (2-\gamma)s + \beta_x, \exists s \leq t, X_s \geq (2-\gamma)s + \beta_x - \frac{\sqrt{s}}{M \log(2+s)^2} }  \right] } \\
& \lesssim o_{t \to \infty} (1) + 
(2-\gamma)^{-1} \E_{x_0} \bigg[ \sqrt{r} e^{\gamma \sqrt{r}}  \nonumber \\
& ~~~~~~ \Pb_r \left( \forall s \leq t, X_s < (2-\gamma)s + \beta_x, \exists s \leq t, X_s \geq (2-\gamma)s + \beta_x - \frac{\sqrt{s}}{M \log(2+s)^2}  \right) \bigg]
\label{eq:proof_lem_bad_points_b}
\end{align}
where we obtained the above estimate by decomposing the expectation according to whether $X_t \leq -\gamma t/2$ or not.
By Girsanov's theorem and then by Lemma \ref{lem:bessel_RN_derivative}, the above probability with respect to the one-dimensional Brownian motion is equal to
\begin{align*}
& e^{-(2-\gamma)^2 t/2}
\Eb_0 \left[ e^{-(2-\gamma)X_t} \indic{\forall s \leq t, X_s < \beta_x -r, \exists s \leq t, X_s \geq \beta_x - r - \frac{\sqrt{s}}{M \log(2+s)^2}} \right] \\
& = e^{-(2-\gamma)^2 t/2}
\Eb_{\beta_x-r}^3 \left[ \frac{\beta_x -r}{X_t} e^{-(2-\gamma)(\beta_x -r - X_t)} \indic{\exists s \leq t, X_s \leq \frac{\sqrt{s}}{M \log(2+s)^2}} \right].
\end{align*}
By decomposing the above expectation according to whether $X_t \geq (2-\gamma) t / 4$ or not, we see that it is at most, up to a multiplicative constant,
\begin{align*}
& e^{-(2-\gamma)^2 t / 4} + e^{-(2-\gamma)^2 t/2} \Eb^3_{\beta_x-r} \left[ \frac{\beta_x -r}{(2-\gamma)t} e^{(2-\gamma) X_t} \indic{ \exists s \leq t, X_s \leq \frac{\sqrt{s}}{M \log(2+s)^2}} \right].
\end{align*}
Now, by Lemma \ref{lem:Bessel3_barrier} point \ref{lemlem:barrier} and because $X_t$ under $\Pb_{\beta_x-r}^3 \left( \cdot \left\vert \exists s \leq t, X_s \leq \frac{\sqrt{s}}{M \log(2+s)^2} \right. \right)$ is stochastically dominated by $X_t$ under $\Pb^3_{\beta_x-r}$, we see that the probability in \eqref{eq:proof_lem_bad_points_b} is at most, up to a multiplicative constant,
\begin{align*}
& e^{-(2-\gamma)^2 t / 4} + o_{M \to \infty}(1) e^{-(2-\gamma)^2 t/2} \Eb^3_{\beta_x-r} \left[ \frac{\beta_x -r}{(2-\gamma)t} e^{(2-\gamma) X_t} \right].
\end{align*}
By a similar procedure as above we can reintroduce $\frac{\beta_x-r}{X_t}$ in the expectation above in place of $\frac{\beta_x-r}{(2-\gamma) t}$ and reverse the computations using Lemma \ref{lem:bessel_RN_derivative} and then Girsanov's theorem to obtain that
\begin{align*}
& e^{-(2-\gamma)^2 t/2} \Eb^3_{\beta_x-r} \left[ \frac{\beta_x -r}{(2-\gamma)t} e^{(2-\gamma) X_t} \right] \\
& \lesssim e^{-(2-\gamma)^2 t/4} + \Pb_r \left( \forall s \leq t, X_s < (2-\gamma)s + \beta_x \right) \lesssim 2- \gamma
\end{align*}
by \eqref{eq:lem_BM_estimates_3}.
Wrapping things up, we have obtained that the probability in \eqref{eq:proof_lem_bad_points_b} is at most
\[
o_{M \to \infty}(1) (2-\gamma)
\]
as desired. This concludes the proof.
\end{proof}

\section{\texorpdfstring{$L^2$}{L2}-estimates}
\label{sec:L2}

\subsection{Uniform integrability: proof of Proposition \ref{prop:bdd_L2}}

This section is devoted to the proof of Proposition \ref{prop:bdd_L2}.
We first state the following result for ease of reference.

\begin{lemma}\label{lem:sum_squared_Bessel}
Let $I$ be a finite set of indices, $(r_i, i \in I) \in [0,\infty)^I$ and let $(X^{(i)}, i \in I) \sim \otimes_{i \in I} \Pb^0_{r_i}$ be independent zero-dimensional Bessel processes starting at $r_i$. Define the process $(X_s, s \geq 0)$ as follows: for all $n \geq 0$, let $X_n = \sqrt{ \sum_{i \in I} (X_n^{(i)})^2 }$ and conditionally on $(X_n^{(i)}, n \geq 1, i \in I)$, let $(X_s, s \in (n,n+1)), n \geq 0,$ be independent zero-dimensional Bessel bridges between $X_n$ and $X_{n+1}$. Then $X \sim \Pb^0_r$ with $r = \sqrt{\sum_{i \in I} r_i^2}$.
\end{lemma}

\begin{proof}
This is a direct consequence of the fact that the sum of independent zero-dimensional squared Bessel processes is again distributed as a zero-dimensional squared Bessel process.
\end{proof}

\begin{proof}[Proof of Proposition \ref{prop:bdd_L2}]
The constants underlying this proof may depend on $\beta$ and $M$. We start by proving \eqref{eq:prop_L2_b}. We will then see that very few arguments need to be modified to obtain \eqref{eq:prop_L2_a} and \eqref{eq:prop_L2_subcritical}.
Let $\eps'$ be the only real number in $\{e^{-n}, n \geq 1\}$ be such that
\begin{equation}
\label{eq:proof_prop_L2_b}
1 \leq \frac{\eps'}{e^4 M \eps \exp \left( (\log |\log \eps|)^6 \right)} < e.
\end{equation}
We are first going to control the contribution of points $x, y \in D$ at distance at least $1/M$ from $x_0$ such that $|x-y| \leq \eps'$. Let $x$ and $y$ be such points.
On $G_\eps'(y)$,
\[
\sqrt{\frac{1}{\eps}L_{y,\eps}(\tau)} \leq 2 \log \frac{1}{\eps} - \frac{\sqrt{|\log \eps|}}{M\log(2+|\log \eps|)^2} + \beta.
\]
We thus have
\begin{align*}
(\eps')^2 \EXPECT{x_0}{\dhat{\mu}_\eps(x) \dhat{\mu}_\eps(y) }
& \lesssim (\eps')^2  |\log \eps|^3 \exp \left( - \frac{2\sqrt{|\log \eps|}}{M\log(2+|\log \eps|)^2} \right) \EXPECT{x_0}{ e^{2 \sqrt{\frac{1}{\eps}L_{x,\eps}(\tau)} } } \\
& \lesssim \left( \frac{\eps'}{\eps} \right)^2 |\log \eps|^{7/2} \exp \left( - \frac{2\sqrt{|\log \eps|}}{M\log(2+|\log \eps|)^2} \right)
\end{align*}
using \eqref{eq:lem_first_moment} in the last inequality.
This shows that
\[
\int_{D \times D} \sup_{\eps >0} \EXPECT{x_0}{\dhat{\mu}_\eps(dx) \dhat{\mu}_\eps(dy) } \indic{|x-y| \leq \eps'} < \infty.
\]

We now focus on the remaining contribution.
Let $x,y \in D$ at distance at least $1/M$ from $x_0$ be such that $|x-y| \geq \eps'$. Without loss of generality, assume that the diameter of $D$ is at most 1 so that we can define $\alpha = e^{-k_\alpha}, \eta = e^{-k_\eta} \in \{ e^{-n} , n \geq 1 + \floor{\log M} \}$ to be the only real numbers satisfying
\begin{equation}
\label{eq:proof_prop_L2_d}
\frac{1}{e^2 M} \leq \frac{\alpha}{|x-y|} < \frac{1}{e M}
\quad \mathrm{and} \quad
\frac{1}{e^4 M} \leq
\frac{\eta}{|x-y| \exp \left( - (\log |\log |x-y||)^6 \right) } < \frac{1}{e^3 M}.
\end{equation}
Notice that $D(x,\alpha) \cap D(y,\alpha) = \varnothing$ (as soon as $M$ is at least $2/e$), that $\eta \geq \eps$ because $|x-y| \geq \eps'$, that $k_\eta \geq 1 + \log M \geq k_x$ and that $\eta < \alpha/e$.
Define
\[
G_{\eta,\eps}(x) := \left\{ \forall s \in [k_\eta,k], h_{x,e^{-s}} \leq 2 s + \beta \right\}.
\]
Importantly, the event $G_{\eta, \eps}(x)$ is contained in $G_\eps(x)$ and only cares about what happens inside the disc $D(x,\alpha/e)$.
We similarly define $G_{\eta,\eps}(y)$.
We can bound $\EXPECT{x_0}{\dhat{\mu}_\eps(x) \dhat{\mu}_\eps(y) }$ by
\begin{align}
\nonumber
& |\log \eps| \eps^4 \E_{x_0} \Bigg[ \left( - \sqrt{\frac{1}{\eps}L_{x,\eps}(\tau_{x,R})} + 2 \log \frac{1}{\eps} + \beta \right) \left( - \sqrt{\frac{1}{\eps}L_{y,\eps}(\tau_{y,R})} + 2 \log \frac{1}{\eps} + \beta \right) \\
\label{eq:proof_prop_L2_c}
&  e^{2 \sqrt{\frac{1}{\eps}L_{x,\eps}(\tau_{x,R})}} e^{2 \sqrt{\frac{1}{\eps}L_{y,\eps}(\tau_{y,R})}} \mathbf{1}_{G_{\eta, \eps}(x)} \mathbf{1}_{G_{\eta, \eps}(y)} \indic{ \sqrt{\frac{1}{\eta}L_{y,\eta}(\tau_{y,R})} \leq 2 \log \frac{1}{\eta} + \beta - \frac{\sqrt{|\log \eta|}}{M \log(2+|\log \eta|)^2} } \Bigg].
\end{align}
In broad terms, our strategy now is to condition on $L_{x,\eta}(\tau_{x,R})$ and $L_{y,\eta}(\tau_{y,R})$ and integrate everything else. Let $N_x$ be the number of excursions from $\partial D(x,\alpha/e)$ to $\partial D(x,\alpha)$ before hitting $\partial D(x,R)$. For $i=1 \dots N_x$ and $\delta \leq \alpha/e$, let $L_{x,\delta}^i$ be the local time of $\partial D(x,\delta)$ accumulated during the $i$-th excursion. We also write $r_{x,\eta}^i := \sqrt{\frac{1}{\eta} L_{x,\eta}^i}$ and $r_{x,\eta} := \sqrt{\frac{1}{\eta} L_{x,\eta}(\tau_{x,R})}$.
Let $I_x$ be the subset of $\{1, \dots, N_x\}$ corresponding to the above excursions that hit $\partial D(x,\eta)$.
Define similar notations with $x$ replaced by $y$ et let $\Fc_{x,y}$ be the sigma algebra generated by
$N_x, N_y, I_x, I_y$
and the successive initial and final positions of the above-mentioned excursions (around both $x$ and $y$).

Conditionally on the initial and final positions of the above excursions,
\[
\left( L_{x,\delta}^i, i =1 \dots N_x, \delta \leq \alpha/e \right)
\quad \mathrm{and} \quad
\left( L_{y,\delta}^i, i =1 \dots N_y, \delta \leq \alpha/e \right)
\]
are independent. Moreover, for all $i=1 \dots N_x$, conditioned on $\{ i \in I_x \}$, $\left( L_{x,e^{-n}}^i, n \geq k_\alpha + 1 \right)$ is close to be independent of the initial and final positions of the given excursion: this is the content of the continuity Lemma \ref{lem:independence local times and exit point}. The Bessel bridges that we use to interpolate the local times between dyadic radii smaller than $\alpha$ around $x$ and $y$ do not create any further dependence since $D(x,\alpha) \cap D(y,\alpha) = \varnothing$. Hence, recalling \eqref{eq:local_times_and_Bessel_process} and Lemma \ref{lem:sum_squared_Bessel}, we see that by paying a multiplicative price $\left( 1 + p \left( \frac{\eta}{\alpha} \right) \right)^{|I_x|+|I_y|}$ and conditionally on $\Fc_{x,y}$, we can approximate the joint law of $(h_{x,\eta e^{-s}}, s \geq 0)$ and $(h_{y,\eta e^{-s}}, s \geq 0)$ by $\Pb^0_{r_{x,\eta}} \otimes \Pb^0_{r_{y,\eta}}$.
Letting $t = \log \frac{\eta}{\eps} = k-k_\eta$ and $\beta' := \beta + 2k_\eta$, we deduce that
\begin{align*}
& |\log \eps| \eps^4 \E_{x_0} \Bigg[ \left( - \sqrt{\frac{1}{\eps}L_{x,\eps}(\tau_{x,R})} + 2 \log \frac{1}{\eps} + \beta \right) \left( - \sqrt{\frac{1}{\eps}L_{y,\eps}(\tau_{y,R})} + 2 \log \frac{1}{\eps} + \beta \right) \\
&  e^{2 \sqrt{\frac{1}{\eps}L_{x,\eps}(\tau_{x,R})}} e^{2 \sqrt{\frac{1}{\eps}L_{y,\eps}(\tau_{y,R})}} \mathbf{1}_{G_{\eta, \eps}(x)} \mathbf{1}_{G_{\eta, \eps}(y)} \indic{ \sqrt{\frac{1}{\eta}L_{y,\eta}(\tau_{y,R})} \leq 2 \log \frac{1}{\eta} + \beta - \frac{\sqrt{|\log \eta|}}{M \log(2+|\log \eta|)^2} } \Bigg\vert \Fc_{x,y} \Bigg] \\
& \leq 
\left( 1 + p \left( \frac{\eta}{\alpha} \right) \right)^{|I_x|+|I_y|}
|\log \eps| \eps^4 \E_{x_0} \Bigg[  \indic{ \sqrt{\frac{1}{\eta}L_{y,\eta}(\tau_{y,R})} \leq 2 \log \frac{1}{\eta} + \beta - \frac{\sqrt{|\log \eta|}}{M \log(2+|\log \eta|)^2} }
\\
& \times \Eb^0_{r_{x,\eta}}
\Bigg[ \left( - X_{t} + 2 t + \beta' \right) e^{2 X_{t}} \indic{ \forall s \leq t, X_s \leq 2s + \beta'} \Bigg] \\
& \times \Eb^0_{r_{y,\eta}}
\Bigg[ \left( - X_{t} + 2 t + \beta' \right) e^{2 X_{t}} \indic{ \forall s \leq t, X_s \leq 2s + \beta'} \Bigg]
\Bigg\vert \Fc_{x,y} \Bigg].
\end{align*}
Now, by \eqref{eq:lem_Bessel_first_moment_b},
\begin{align}
\label{eq:proof_prop_L2_f}
& \sqrt{|\log \eps|} \eps^2 \Eb^0_{r_{x,\eta}} \left[ \left( - X_{t} + 2t + \beta' \right) e^{2 X_{t}} \indic{ \forall s \leq t, X_s \leq 2s + \beta'} \right] \\
\nonumber
& \lesssim \frac{\sqrt{|\log \eps|} \eps^2}{\sqrt{t} e^{-2t}} |\log \eta|^{3/2} e^{2 r_{x,\eta}}
\lesssim |\log \eta|^{3/2} \eta^2 e^{2 \sqrt{\frac{1}{\eta}L_{x,\eta}(\tau_{x,R})}}.
\end{align}
We have a similar estimate for the expectation around the point $y$ and we further bound
\begin{align*}
& \sqrt{|\log \eps|} \eps^2 \Eb^0_{r_{y,\eta}} \left[ \left( - X_{t} + 2t + \beta' \right) e^{2 X_{t}} \indic{ \forall s \leq t, X_s < 2s + \beta'} \right] \\
& \times \indic{ \sqrt{\frac{1}{\eta}L_{y,\eta}(\tau_{x,R})} \leq 2 \log \frac{1}{\eta} + \beta - \frac{\sqrt{|\log \eta|}}{M \log(2+|\log \eta|)^2} } \\
& \lesssim |\log \eta|^{3/2} \eta^{-2} \exp \left( - \frac{2\sqrt{|\log \eta|}}{M \log(2+|\log \eta|)^2} \right).
\end{align*}
To wrap things up, we have proven that
\begin{align}
\label{eq:proof_prop_L2_e}
& \EXPECT{x_0}{\dhat{\mu}_\eps(x) \dhat{\mu}_\eps(y) \vert ~ |I_x| + |I_y| } \\
\nonumber
& \lesssim \left( 1 + p \left( \frac{\eta}{\alpha} \right) \right)^{|I_x| + |I_y|}
|\log \eta|^3 \exp \left( - \frac{2\sqrt{|\log \eta|}}{M \log(2+|\log \eta|)^2} \right) \EXPECT{x_0}{ \left. e^{2 \sqrt{\frac{1}{\eta}L_{x,\eta}(\tau_{x,R})}} \right\vert ~ |I_x| + |I_y| }.
\end{align}
By the continuity Lemma \ref{lem:independence local times and exit point} and recalling \eqref{eq:proof_prop_L2_d}, there exists $c_* >0$ such that
\[
\log \left( 1 + p \left( \frac{\eta}{\alpha} \right) \right)
\lesssim \exp \left( - c_* \left( \log \frac{\alpha}{\eta} \right)^{1/3} \right)
\lesssim \exp \left( - c_* \left( \log |\log |x-y|| \right)^2 \right).
\]
If we take $N$ to be equal to $\exp \left(  c_* \left( \log |\log |x-y|| \right)^2 /2 \right)$, we thus have
\[
\left( 1 + p \left( \frac{\eta}{\alpha} \right) \right)^N \lesssim 1
\]
and \eqref{eq:proof_prop_L2_e} together with \eqref{eq:lem_first_moment} yield
\begin{align*}
& \EXPECT{x_0}{\dhat{\mu}_\eps(x) \dhat{\mu}_\eps(y) \indic{|I_x| + |I_y| \leq N} } \\
& \lesssim |\log \eta|^3 \exp \left( - \frac{2\sqrt{|\log \eta|}}{M \log(2+|\log \eta|)^2} \right) \EXPECT{x_0}{ e^{2 \sqrt{\frac{1}{\eta}L_{x,\eta}(\tau_{x,R})}} } \\
& \lesssim |\log \eta|^{7/2} \eta^{-2} \exp \left( - \frac{2\sqrt{|\log \eta|}}{M \log(2+|\log \eta|)^2} \right) \\
& \lesssim |\log |x-y||^{7/2} |x-y|^{-2} \exp \left( 2 (\log |\log |x-y||)^6 \right) \exp \left( - \frac{2\sqrt{|\log |x-y||}}{M \log(2+|\log |x-y||)^2} \right) \\
& \lesssim |x-y|^{-2} \exp \left( - \frac{\sqrt{|\log |x-y||}}{M \log(2+|\log |x-y||)^2} \right).
\end{align*}
We now explain how to bound $\EXPECT{x_0}{\dhat{\mu}_\eps(x) \dhat{\mu}_\eps(y) \indic{|I_x| + |I_y| > N} }$. $|I_x|$ is smaller than the number of excursions from $\partial D(x,\alpha/e)$ to $\partial D(x,\eta)$ before hitting $\partial D(x,R)$ and the probability for a Brownian trajectory starting at $\partial D(x,\alpha/e)$ to hit $\partial D(x,\eta)$ before hitting $\partial D(x,R)$ is given by
\[
\frac{\log(\alpha/eR)}{\log(\eta/R)}.
\]
By strong Markov property, we then obtain that for all $M>0$,
\begin{align*}
\PROB{x_0}{|I_x| > M} \leq \left( \frac{\log(\alpha/eR)}{\log(\eta/R)} \right)^M
& \leq \exp \left( - c \frac{(\log |\log |x-y||)^6}{|\log |x-y||} M \right).
\end{align*}
Using \eqref{eq:proof_prop_L2_e}, Cauchy--Schwarz and \eqref{eq:lem_first_moment}, we deduce that
\begin{align}
\nonumber
& \EXPECT{x_0}{\dhat{\mu}_\eps(x) \dhat{\mu}_\eps(y) \indic{|I_x| + |I_y| > N} } \\
\nonumber
& \leq \sum_{p \geq \floor{\log_2 N}} \EXPECT{x_0}{\dhat{\mu}_\eps(x) \dhat{\mu}_\eps(y) \indic{2^p \leq |I_x| + |I_y| < 2^{p+1}} } \\
\nonumber
& \lesssim |\log \eta|^3 \exp \left( - \frac{2\sqrt{|\log \eta|}}{M \log(2+|\log \eta|)^2} \right) \EXPECT{x_0}{ e^{4 \sqrt{\frac{1}{\eta}L_{x,\eta}(\tau_{x,R})}} }^{1/2} \\
\nonumber
& \times \sum_{p \geq \floor{\log_2 N}} \left( 1 + p \left( \frac{\eta}{\alpha} \right) \right)^{2^{p+1}} \left( \PROB{x_0}{ |I_x| \geq 2^{p-1}} + \PROB{x_0}{ |I_y| \geq 2^{p-1}} \right)^{1/2} \\
\label{eq:proof_prop_bdd_L2}
& \lesssim |x-y|^{-4} \exp \left( - c \frac{(\log |\log |x-y||)^6}{|\log |x-y||} N \right) \\
\nonumber
& \leq |x-y|^{-4} \exp \left( - c \exp \left( \frac{c_*}{4} (\log |\log |x-y||)^2 \right) \right) \lesssim 1.
\end{align}
This concludes the proof of \eqref{eq:prop_L2_b}. 

Let $\dhat{\mu}$ be any subsequential limit of $(\dhat{\mu}_\eps, \eps >0)$. 
The claim about the non-atomicity of $\dhat{\mu}$ follows from the following energy estimate which is a consequence of what we did before:
\begin{align*}
& \EXPECT{x_0}{\int_{D \times D} \exp \left( |\log |x-y||^{1/3} \right) \dhat{\mu}(dx) \dhat{\mu}(dy) } \\
& \leq \limsup_{\eps \to 0} \EXPECT{x_0}{\int_{D \times D} \exp \left( |\log |x-y||^{1/3} \right) \dhat{\mu}_\eps(dx) \dhat{\mu}_\eps(dy) } < \infty.
\end{align*}

\medskip

For the proof of \eqref{eq:prop_L2_a}, resp. \eqref{eq:prop_L2_subcritical}, we proceed in the exact same way as before. The only difference is that, instead of \eqref{eq:proof_prop_L2_f}, we need to bound from above
\[
|\log \eps| \eps^2 \Eb^0_{r_{x,\eta}} \left[ e^{2X_t} \indic{\forall s \leq t, -X_s + 2s + \beta' >0} \right],
\]
resp.
\[
\frac{1}{2-\gamma} \sqrt{|\log \eps|} \eps^{\gamma^2/2} \Eb^0_{r_{x,\eta}} \left[e^{\gamma X_t} \indic{\forall s \leq t, -X_s + 2s + \beta' >0} \right].
\]
This is done in \eqref{eq:lem_Bessel_first_moment_a}, resp. \eqref{eq:lem_Bessel_first_moment_c}, and we conclude the proof of \eqref{eq:prop_L2_a}, resp. \eqref{eq:prop_L2_subcritical}, along the same lines as above.

\end{proof}

\subsection{Cauchy sequence in \texorpdfstring{$L^2$}{L2}: proof of Proposition \ref{prop:Cauchy_L2}}
\label{subsec:Cauchy}

This section is devoted to the proof of Proposition \ref{prop:Cauchy_L2}.

\begin{proof}[Proof of Proposition \ref{prop:Cauchy_L2}]
Let $A$ be a Borel set of $\R^2$. Let $\eta = e^{-k_\eta} \in \{e^{-n}, n \geq 1\}$ be small and consider
\begin{equation}
\label{eq:proof_prop_Cauchy_g}
(A \times A)_\eta := \left\{ (x,y) \in A \times A: \forall n \geq 1, D(x,\eta) \cap \partial D(y,e^{-n}) = D(y,\eta) \cap \partial D(x,e^{-n}) = \varnothing \right\}.
\end{equation}
If $(x,y) \in (A \times A)_\eta$, the two sequences of circles $(\partial D(x,e^{-n}), n \geq 1)$ and $(\partial D(y,e^{-n}), n \geq 1)$ will not interact between each other inside $D(x,\eta)$ and $D(y,\eta)$. We can write
\begin{align*}
\limsup_{\eps, \eps' \to 0} \EXPECT{x_0}{ \left( \dhat{\mu}_\eps(A) - \dhat{\mu}_{\eps'}(A) \right)^2 }
& \leq 2 \limsup_{\eps \to 0} \int_{(A \times A) \backslash (A \times A)_\eta} \EXPECT{x_0}{\dhat{\mu}_\eps(dx) \dhat{\mu}_\eps(dy)} \\
& + 2 \limsup_{\eps, {\eps'} \to 0} \int_{(A \times A)_\eta} \EXPECT{x_0}{\dhat{\mu}_\eps(x) \left( \dhat{\mu}_\eps(y) - \dhat{\mu}_{\eps'}(y) \right)} dx dy.
\end{align*}
Thanks to \eqref{eq:prop_L2_b} and because the Lebesgue measure of $(A \times A) \backslash (A \times A)_\eta$ goes to zero as $\eta \to 0$, we know that the first right hand side term goes to zero as $\eta \to 0$. We are going to show that for a fixed $\eta$ the second right hand side term vanishes. \eqref{eq:prop_L2_b} provides the upper bound required to apply dominating convergence theorem and we are left to show the pointwise convergence
\begin{equation}
\label{eq:proof_prop_Cauchy_a}
\limsup_{\eps,{\eps'} \to 0}
\EXPECT{x_0}{\dhat{\mu}_\eps(x) \left( \dhat{\mu}_\eps(y) - \dhat{\mu}_{\eps'}(y) \right)} = 0
\end{equation}
for a fixed $(x,y) \in (A \times A)_\eta$. Let $\eta' = e^{-k_{\eta'}} \in \{e^{-n}, n \geq 0\}$ be much smaller than $\eta$. Let $N_y$ (resp. $N_y'$) be the number of excursions from $\partial D(y, \eta/e)$ to $\partial D(y,\eta)$ before hitting $\partial D$ (resp. before hitting $\partial D(y,R)$). For $i=1 \dots N_y'$ and $\delta \leq \eta/e$, we will denote $L_{y,\delta}^i$ the local time of $\partial D(y,\delta)$ accumulated during the $i$-th such excursion. Denote by $I$ (resp. $I'$) the subset of $\{1, \dots, N_y\}$ (resp. $\{1, \dots, N_y'\}$) corresponding to the excursions that visited $\partial D(y,\eta')$. First of all, one can show that there exists $N \geq 1$ depending on $\eta$ such that
\[
\limsup_{\eps, \eps' \to 0} \EXPECT{x_0}{\dhat{\mu}_\eps(x) \dhat{\mu}_{\eps'}(y) \indic{N_y' > N} } \leq \eta.
\]
This is a direct consequence of the bound \eqref{eq:proof_prop_bdd_L2}. Let $\Fc_{x,y}$ be the sigma-algebra generated by $(L_{x,e^{-n}}(\tau), L_{x,e^{-n}}(\tau_{x,R}), n \geq 0)$, $(L_{y,e^{-n}}(\tau), L_{y,e^{-n}}(\tau_{y,R}), n = 0 \dots k_\eta -1)$, $N_y, N_y', I, I'$, $(L_{y, e^{-n}}^i, i \notin I', n = k_\eta \dots k_{\eta'})$ as well as the starting and exiting point of the excursions from $\partial D(y, \eta/e)$ to $\partial D(y,\eta)$ before hitting $\partial D(y,R)$. Denote $(e/\eta) (r_{y,\eta/e})^2$ (resp. $(e/\eta) (r_{y,\eta/e}')^2$) the local time $L_{y,\eta/e}(\tau)$ (resp. $L_{y,\eta/e}(\tau_{y,R}) - L_{y,\eta/e}(\tau)$), $t = \log (\eta/(e\eps))$, $\beta' = \beta - 2\log(\eta/e)$, $t_0 = \log (e/\eta)$, $t_1 = \log(e \eta'/\eta)$. With a reasoning similar as what we did in the proof of Proposition \ref{prop:bdd_L2}, Lemma \ref{lem:independence local times and exit point}, \eqref{eq:local_times_and_Bessel_process} and Lemma \ref{lem:sum_squared_Bessel} imply that $\EXPECT{x_0}{\dhat{\mu}_\eps(x) \dhat{\mu}_\eps(y) \indic{N_y' \leq N} }$ is equal to
\begin{align*}
& (1 \pm p(\eta'/\eta))^N \E_{x_0} \Bigg[ \sqrt{|\log \eps|} \eps^2 \left( - \sqrt{\frac{1}{\eps} L_{x,\eps}(\tau_{x,R}) } + 2 |\log \eps| + \beta \right) e^{2 \sqrt{\frac{1}{\eps} L_{x,\eps}(\tau)} } \mathbf{1}_{G_\eps(x) \cap G_\eps'(x)} \\
& \times \mathbf{1}_{G_{\eta/e}(y) \cap G_{\eta/e}'(y)} \indic{N_y' \leq N}
\Eb^0_{r_{y,\alpha/e}} \otimes \Eb^0_{r_{y,\alpha/e}'} \Bigg[ \sqrt{|\log \eps|} \eps^2 \left( - \sqrt{X_t^2 + (X_t')^2} + 2t + \beta' \right) e^{2X_t} \\
& \times f_t(X_s,X'_s,s \leq t)
\Bigg\vert \Fc_{x,y} \Bigg] \Bigg]
\end{align*}
where
\begin{align*}
f_t(X_s, X_s', s \leq t) & :=
\indic{\forall s \leq t_1 - t_0, \sqrt{ X_s^2 + (X_s')^2 + \sum_{i \notin I'} \frac{e}{\eta} e^s L_{y,\eta e^{-s-1}}^i} \leq 2s + \beta' - \frac{\sqrt{s+t_0}}{M \log (2+t_0+s)^2} } \\
& \times \indic{\forall s \in [t_1 - t_0, t], \sqrt{ X_s^2 + (X_s')^2} \leq 2s + \beta' - \frac{\sqrt{s+t_0}}{M \log (2+t_0+s)^2} }.
\end{align*}
Now, by \eqref{eq:lem_Bessel0_Bessel3_gamma=2}, we have
\begin{align*}
& \Eb^0_{r_{y,\alpha/e}} \otimes \Eb^0_{r_{y,\alpha/e}'} \Bigg[ \sqrt{|\log \eps|} \eps^2 \left( - \sqrt{X_t^2 + (X_t')^2} + 2t + \beta' \right) e^{2X_t} f_t(X_s,X'_s,s \leq t) \indic{X_t >0}
\Bigg\vert \Fc_{x,y} \Bigg] \\
& = \frac{\sqrt{r_{y,\alpha/e}}}{\sqrt{2}} e^{2 r_{y,\alpha/e}} \frac{\sqrt{|\log \eps|}}{\sqrt{|\log (e\eps/\eta)|}} \left( \frac{\eta}{e} \right)^2 (\beta' - r_{y,\eta/e})
\Eb^3_{\beta' - r_{y,\alpha/e}} \otimes \Eb^0_{r_{y,\alpha/e}'}
\Bigg[
\left( 1 - \frac{X_t-\beta'}{2t} \right)^{-1/2} \\
& \times \frac{- \sqrt{(2t- X_t + \beta')^2 + (X'_t)^2} + 2t + \beta'}{ X_t } \exp \left( - \frac{3}{8} \int_0^t \frac{ds}{(2s-X_s+\beta')^2} \right) \\
& \times f_t(2s - X_s + \beta', X_s', s \leq t)
\Bigg\vert \Fc_{x,y}
\Bigg]
\end{align*}
which converges as $\eps \to 0$ (and hence $t \to \infty$) towards
\begin{align*}
& \frac{\sqrt{r_{y,\alpha/e}}}{\sqrt{2}} e^{2 r_{y,\alpha/e}} \left( \frac{\eta}{e} \right)^2 (\beta' - r_{y,\eta/e})
\Eb^3_{\beta' - r_{y,\alpha/e}} \otimes \Eb^0_{r_{y,\alpha/e}'}
\Bigg[ \exp \left( - \frac{3}{8} \int_0^\infty \frac{ds}{(2s - X_s + \beta')^2} \right) \\
& \times f_\infty(2s - X_s + \beta', X_s', s \geq 0)
\Bigg\vert \Fc_{x,y}
\Bigg].
\end{align*}
This shows that
\[
\limsup_{\eps,\eps' \to 0} (1+p(\eta'/\eta))^{-N} \EXPECT{x_0}{\dhat{\mu}_\eps(x) \dhat{\mu}_\eps(y) \indic{N_y' \leq N} } - (1-p(\eta'/\eta))^{-N} \EXPECT{x_0}{\dhat{\mu}_\eps(x) \dhat{\mu}_{\eps'}(y) \indic{N_y' \leq N} }
\]
is at most zero.
The only quantity depending on $\eta'$ in the above expression is $p(\eta'/\eta)$ which goes to zero as $\eta' \to 0$. By letting $\eta' \to 0$, we thus obtain
\[
\limsup_{\eps,\eps' \to 0} \EXPECT{x_0}{\dhat{\mu}_\eps(x) \dhat{\mu}_\eps(y) \indic{N_y' \leq N} } - \EXPECT{x_0}{\dhat{\mu}_\eps(x) \dhat{\mu}_{\eps'}(y) \indic{N_y' \leq N} } \leq 0.
\]
This concludes the proof of the fact that $(\dhat{\mu}_\eps(A), \eps >0)$ is Cauchy in $L^2$.

\medskip

We move on to the proof of the convergence of $(\dhat{m}_\eps(A), \eps>0)$ together with the identification of the limit with $\sqrt{2/\pi} \dhat{\mu}(A)$. Since we already know that $(\dhat{\mu}_\eps(A), \eps>0)$ converges in $L^2$ towards $\dhat{\mu}(A)$, it is enough to show that
\[
\limsup_{\eps \to 0} \EXPECT{x_0}{ |\dhat{m}_\eps(A) - \sqrt{2/\pi} \dhat{\mu}_\eps(A) |^2 } = 0.
\]
In particular, we don't need to consider ``mixed moments'' with $\eps' \neq \eps$.
As before, we bound
\begin{align*}
& \limsup_{\eps \to 0} \EXPECT{x_0}{ |\dhat{m}_\eps(A) - \sqrt{2/\pi} \dhat{\mu}_\eps(A) |^2 } \\
& \leq \limsup_{\eps \to 0} \int_{(A \times A) \backslash (A \times A)_\eta} \EXPECT{x_0}{\dhat{m}_\eps(dx) \dhat{m}_\eps(dy)} + \frac{2}{\pi} \EXPECT{x_0}{\dhat{\mu}_\eps(dx) \dhat{\mu}_\eps(dy)} \\
& + \limsup_{\eps \to 0} \int_{(A \times A)_\eta} \EXPECT{x_0}{\dhat{m}_\eps(dx) \left( \dhat{m}_\eps(dy) - \sqrt{\frac{2}{\pi}} \dhat{\mu}_\eps(dy) \right)} \\
& + \sqrt{\frac{2}{\pi}} \limsup_{\eps \to 0} \int_{(A \times A)_\eta} \EXPECT{x_0}{\dhat{\mu}_\eps(dx) \left( \sqrt{\frac{2}{\pi}} \dhat{\mu}_\eps(dy) - \dhat{m}_\eps(dy) \right)}.
\end{align*}
As before, we only need to care about the two last right hand side terms and thanks to \eqref{eq:prop_L2_a} and \eqref{eq:prop_L2_b}, we only need to show the two following pointwise convergences:
\begin{equation}
\label{eq:proof_prop_Cauchy_c}
\lim_{\eps \to 0} \EXPECT{x_0}{\dhat{m}_\eps(dx) \left( \dhat{m}_\eps(dy) - \sqrt{\frac{2}{\pi}} \dhat{\mu}_\eps(dy) \right)}
= \lim_{\eps \to 0} \EXPECT{x_0}{\dhat{\mu}_\eps(dx) \left( \sqrt{\frac{2}{\pi}} \dhat{\mu}_\eps(dy) - \dhat{m}_\eps(dy) \right)}
= 0
\end{equation}
where $(x,y) \in (A \times A)_\eta$ is fixed. In both cases, we employ the same technique as before by decomposing the Brownian trajectory according to what happens close to the point $y$ and \eqref{eq:proof_prop_Cauchy_c} follows from the fact that
\[
\Eb^0_{r_{y,\alpha/e}} \otimes \Eb^0_{r_{y,\alpha/e}'} \Bigg[ |\log \eps| \eps^2 e^{2X_t} f_t(X_s,X'_s,s \leq t)
\Bigg\vert \Fc_{x,y} \Bigg]
\]
converges to the same limit as
\[
\sqrt{\frac{2}{\pi}} \Eb^0_{r_{y,\alpha/e}} \otimes \Eb^0_{r_{y,\alpha/e}'} \Bigg[ \sqrt{|\log \eps|} \eps^2 \left( - \sqrt{X_t^2 + (X_t')^2} + 2t + \beta' \right) e^{2X_t} f_t(X_s,X'_s,s \leq t)
\Bigg\vert \Fc_{x,y} \Bigg].
\]
Let us justify this last claim. After using \eqref{eq:lem_Bessel0_Bessel3_gamma=2}, we see that we only need to show that
\begin{align}
\label{eq:proof_prop_Cauchy_e}
& \lim_{t \to \infty}
\Eb^3_{\beta' - r_{y,\alpha/e}} \otimes \Eb^0_{r_{y,\alpha/e}'}
\Bigg[
\frac{\sqrt{t}}{X_t}
\left( 1 - \frac{X_t-\beta'}{2t} \right)_+^{-1/2} \\
\nonumber
& \times \exp \left( - \frac{3}{8} \int_0^t \frac{ds}{(2s-X_s+\beta')^2} \right) f_t(2s - X_s + \beta', X_s', s \leq t)
\Bigg\vert \Fc_{x,y}
\Bigg] \\
\nonumber
& = \sqrt{\frac{2}{\pi}} \lim_{t \to \infty}
\Eb^3_{\beta' - r_{y,\alpha/e}} \otimes \Eb^0_{r_{y,\alpha/e}'}
\Bigg[
\frac{- \sqrt{(2t- X_t + \beta')^2 + (X'_t)^2} + 2t + \beta'}{ X_t }
\left( 1 - \frac{X_t-\beta'}{2t} \right)_+^{-1/2} \\
\nonumber
& \times \exp \left( - \frac{3}{8} \int_0^t \frac{ds}{(2s-X_s+\beta')^2} \right) f_t(2s - X_s + \beta', X_s', s \leq t)
\Bigg\vert \Fc_{x,y}
\Bigg].
\end{align}
Take $t_2>t_1 -t_0$ large. We can bound
\begin{align*}
& \abs{ f_{t_2}(2s - X_s + \beta', X_s', s \leq t_2) - f_t(2s - X_s + \beta', X_s', s \leq t) } \\
& \leq \indic{\exists s \geq t_2, X_s < \frac{\sqrt{s+t_0}}{M\log(2+t_0+s)^2} \mathrm{~or~} X_s \geq 2s + \beta' } + \indic{ \exists s \geq t_2, X_s' >0 }.
\end{align*}
The difference between the expectation on the left hand side of \eqref{eq:proof_prop_Cauchy_e} and
\begin{align*}
& \Eb^3_{\beta' - r_{y,\alpha/e}} \otimes \Eb^0_{r_{y,\alpha/e}'}
\Bigg[
\frac{\sqrt{t}}{X_t}
\left( 1 - \frac{X_t-\beta'}{2t} \right)_+^{-1/2} \\
& \times
\exp \left( - \frac{3}{8} \int_0^t \frac{ds}{(2s-X_s+\beta')^2} \right) f_{t_2}(2s - X_s + \beta', X_s', s \leq t_2)
\Bigg\vert \Fc_{x,y}
\Bigg]
\end{align*}
is thus at most
\begin{align*}
& \Eb^3_{\beta' - r_{y,\alpha/e}} \otimes \Eb^0_{r_{y,\alpha/e}'}
\Bigg[
\frac{\sqrt{t}}{X_t}
\left( 1 - \frac{X_t-\beta'}{2t} \right)_+^{-1/2} \\
& ~~~~~~ \times \Bigg\{ \indic{\exists s \geq t_2, X_s < \frac{\sqrt{s+t_0}}{M\log(2+t_0+s)^2} \mathrm{~or~} X_s \geq 2s + \beta' } + \indic{ \exists s \geq t_2, X_s' >0 } \Bigg\} \Bigg].
\end{align*}
Let $q_1 \in (1,3), q_2 \in (1,2)$ and $q_3 > 1$ be such that $1/q_1 + 1/q_2 + 1/q_3 = 1$. By H\"{o}lder's inequality, we can bound the above expression by
\begin{align*}
& \Eb^3_{\beta' - r_{y,\alpha/e}} \left[ \frac{t^{q_1/2}}{X_t^{q_1}} \right]^{1/q_1} \Eb^3_{\beta' - r_{y,\alpha/e}} \left[ \left( 1 - \frac{X_t-\beta'}{2t} \right)_+^{-q_2/2} \right]^{1/q_2} \Bigg\{ \Pb^0_{r_{y,\alpha/e}'} \left( \exists s \geq t_2, X_s' >0 \right) \\
& + \Pb^3_{\beta' - r_{y,\alpha/e}} \left( \exists s \geq t_2, X_s < \frac{\sqrt{s+t_0}}{M\log(2+t_0+s)^2} \mathrm{~or~} X_s \geq 2s + \beta' \right) \Bigg\}^{1/q_3}.
\end{align*}
The first two expectations are bounded by a universal constant by Lemma \ref{lem:Bessel3_barrier} Points \ref{lemlem:1/X_t^2} and \ref{lemlem:1/X_t^2bis}. The last term containing the two probabilities goes to zero as $t_2 \to \infty$. Similarly, we can replace
\[
\int_0^t \frac{ds}{(2s - X_s + \beta')^2}
\quad \mathrm{by} \quad
\int_0^{t_2} \frac{ds}{(2s - X_s + \beta')^2}
\]
and
\[
\left( 1 - \frac{X_t-\beta'}{2t} \right)_+^{-1/2}
\quad \mathrm{by} \quad
1.
\]
We have shown that the left hand side term of \eqref{eq:proof_prop_Cauchy_e} is equal to $o_{t_2 \to \infty}(1)$ plus
\begin{align*}
\Eb^3_{\beta' - r_{y,\alpha/e}} \otimes \Eb^0_{r_{y,\alpha/e}'}
\Bigg[
\frac{\sqrt{t}}{X_t} 
\exp \left( - \frac{3}{8} \int_0^{t_2} \frac{ds}{(2s-X_s+\beta')^2} \right) f_{t_2}(2s - X_s + \beta', X_s', s \leq t_2)
\Bigg\vert \Fc_{x,y}
\Bigg].
\end{align*}
By conditioning up to $t_2$ and then by using Lemma \ref{lem:Bessel3_barrier} point \ref{lemlem:1/X_t}, we see that the above expectation converges as $t \to \infty$ to
\[
\sqrt{\frac{2}{\pi}} \Eb^3_{\beta' - r_{y,\alpha/e}} \otimes \Eb^0_{r_{y,\alpha/e}'}
\Bigg[
\exp \left( - \frac{3}{8} \int_0^{t_2} \frac{ds}{(2s-X_s+\beta')^2} \right) f_{t_2}(2s - X_s + \beta', X_s', s \leq t_2)
\Bigg\vert \Fc_{x,y}
\Bigg].
\]
With a similar reasoning as above, one can show that the expectation on the right hand side of \eqref{eq:proof_prop_Cauchy_e}
converges as $t \to \infty$ to $o_{t_2 \to \infty}(1)$ plus
\[
\Eb^3_{\beta' - r_{y,\alpha/e}} \otimes \Eb^0_{r_{y,\alpha/e}'}
\Bigg[
\exp \left( - \frac{3}{8} \int_0^{t_2} \frac{ds}{(2s-X_s+\beta')^2} \right) f_{t_2}(2s - X_s + \beta', X_s', s \leq t_2)
\Bigg\vert \Fc_{x,y}
\Bigg].
\]
We have thus shown the left and right hand sides of \eqref{eq:proof_prop_Cauchy_e} differ by at most some $o_{t_2 \to \infty}(1)$. Since they do not depend on $t_2$, we obtain the claim \eqref{eq:proof_prop_Cauchy_e} by letting $t_2 \to \infty$.
This concludes the fact that $(\dhat{m}_\eps(A), \eps >0)$ converges in $L^2$ towards $\sqrt{\frac{2}{\pi}}\dhat{\mu}(A)$.

\medskip

The fact that for all $\gamma \in (1,2)$, $(\dhat{m}_\eps^\gamma(A), \eps < \eps_\gamma)$ is a Cauchy sequence in $L^2$ follows along lines that are very similar to the proof of the fact that $(\dhat{\mu}_\eps(A), \eps >0)$ is a Cauchy sequence in $L^2$. For this reason we omit the details and we now turn to the proof of the convergence of $((2-\gamma)^{-1} \dhat{m}^\gamma(A), \gamma \in (1,2))$ towards $2 \dhat{\mu}(A)$. Here, we do not restrict ourselves to the sequence $(\gamma_n, n \geq 1)$ as stated in Proposition \ref{prop:Cauchy_L2} to ease notations. We hope the reader will forgive us for this lack of rigour. By Fatou's lemma,
\[
\limsup_{\gamma \to 2} \EXPECT{x_0}{ \abs{2 \dhat{\mu}(A) - \frac{1}{2-\gamma} \dhat{m}^\gamma(A) }^2 }
\leq \limsup_{\gamma \to 2} \limsup_{\eps \to 0} \EXPECT{x_0}{ \abs{2 \dhat{\mu}_\eps(A) - \frac{1}{2-\gamma} \dhat{m}_\eps^\gamma(A) }^2 }
\]
and we aim to show that the above right hand side term vanishes. As before and thanks to \eqref{eq:prop_L2_b} and \eqref{eq:prop_L2_subcritical}, we only need to show the following two pointwise convergences
\begin{align*}
\limsup_{\gamma \to 2} \limsup_{\eps \to 0} \frac{1}{2-\gamma}
\EXPECT{x_0}{\dhat{m}_\eps^\gamma(dx) \left( \frac{1}{2-\gamma} \dhat{m}_\eps^\gamma(dy) - 2 \dhat{\mu}_\eps(dy) \right)} = 0
\end{align*}
and
\begin{align*}
\limsup_{\gamma \to 2} \limsup_{\eps \to 0}
\EXPECT{x_0}{\dhat{\mu}_\eps(dx) \left( \frac{1}{2-\gamma} \dhat{m}_\eps^\gamma(dy) - 2 \dhat{\mu}_\eps(dy) \right)} = 0
\end{align*}
where $(x,y) \in (A \times A)_\eta$ is fixed. In both cases, this follows from the fact that
\begin{equation}
\label{eq:proof_prop_Cauchy_d}
\frac{1}{2-\gamma} \Eb^0_{r_{y,\alpha/e}} \otimes \Eb^0_{r_{y,\alpha/e}'} \Bigg[ \sqrt{|\log \eps|} \eps^{\gamma^2/2} e^{\gamma X_t} f_t(X_s,X'_s,s \leq t)
\Bigg\vert \Fc_{x,y} \Bigg]
\end{equation}
converges as $\eps \to 0$ and then $\gamma \to 2$ to the same limit as
\begin{equation}
\label{eq:proof_prop_Cauchy_f}
2 \Eb^0_{r_{y,\alpha/e}} \otimes \Eb^0_{r_{y,\alpha/e}'} \Bigg[ \sqrt{|\log \eps|} \eps^2 \left( - \sqrt{X_t^2 + (X_t')^2} + 2t + \beta' \right) e^{2X_t} f_t(X_s,X'_s,s \leq t)
\Bigg\vert \Fc_{x,y} \Bigg].
\end{equation}
Let us justify this claim.
By \eqref{eq:lem_Bessel0_BM}, \eqref{eq:proof_prop_Cauchy_d} is equal to
\begin{align*}
& \frac{1}{2-\gamma} 
\sqrt{r_{y,\alpha/e}} e^{\gamma r_{y,\alpha/e}} \frac{\sqrt{|\log \eps|}}{\sqrt{|\log (e\eps/\eta)|}} \left( \frac{\eta}{e} \right)^{\gamma^2/2} 
\Eb_{r_{y,\alpha/e}} \otimes \Eb^0_{r_{y,\alpha/e}'}
\Bigg[
\left( \frac{t}{X_t + \gamma t} \right)^{1/2} \\
& \times \exp \left( - \frac{3}{8} \int_0^t \frac{ds}{(X_s+\gamma s)^2} \right) f_t(X_s + \gamma s, X_s', s \leq t) \Bigg\vert \Fc_{x,y} \Bigg].
\end{align*}
As before, let $t_2 > t_1 - t_0$ be large. One can show in a similar manner as what we did above that
\begin{align*}
& \frac{1}{2-\gamma} \Eb_{r_{y,\alpha/e}} \otimes \Eb^0_{r_{y,\alpha/e}'}
\Bigg[
\left( \frac{t}{X_t + \gamma t} \right)^{1/2}
\exp \left( - \frac{3}{8} \int_0^t \frac{ds}{(X_s+\gamma s)^2} \right) f_t(X_s + \gamma s, X_s', s \leq t) \Bigg] \\
& = o_{t_2 \to \infty}(1) + \frac{1}{2-\gamma} \Eb_{r_{y,\alpha/e}} \otimes \Eb^0_{r_{y,\alpha/e}'}
\Bigg[
\left( \frac{t_2}{X_{t_2} + \gamma t_2} \right)^{1/2}
\exp \left( - \frac{3}{8} \int_0^{t_2} \frac{ds}{(X_s+\gamma s)^2} \right) \\
& \times f_{t_2}(X_s + \gamma s, X_s', s \leq t_2) \indic{\forall s \in [t_2,t], X_s < (2-\gamma)s + \beta'} \Bigg\vert \Fc_{x,y} \Bigg].
\end{align*}
Since (see \cite[Proposition 6.8.1]{Resnick1992} for instance)
\begin{align*}
& \lim_{\gamma \to 2} \lim_{t \to \infty} \frac{1}{2-\gamma} \Pb_{X_{t_2}} \left( \forall s \leq t, X_s < (2-\gamma)s + \beta' + (2-\gamma)t_2 \right) \\
& = \lim_{\gamma \to 2} \frac{1}{2-\gamma} \left( 1 - e^{-2(2-\gamma)(\beta' - X_{t_2})} \right)
= 2 (\beta' - X_{t_2}),
\end{align*}
this shows that the liminf and limsup of \eqref{eq:proof_prop_Cauchy_d} as $\eps \to 0$ and then $\gamma \to 2$ are equal to $o_{t_2 \to \infty}(1)$ plus
\begin{align*}
& 2
\sqrt{r_{y,\alpha/e}} e^{2 r_{y,\alpha/e}} \left( \frac{\eta}{e} \right)^2 
\Eb_{r_{y,\alpha/e}} \otimes \Eb^0_{r_{y,\alpha/e}'}
\Bigg[
\left( \frac{t_2}{X_{t_2} + \gamma t_2} \right)^{1/2} \\
& \times \exp \left( - \frac{3}{8} \int_0^{t_2} \frac{ds}{(X_s+2s)^2} \right) (\beta' -X_{t_2}) f_{t_2}(X_s + 2 s, X_s', s \leq t_2) \Bigg\vert \Fc_{x,y} \Bigg].
\end{align*}
By using \eqref{eq:lem_Bessel0_BM} in the other direction, we see that the above term converges as $t_2 \to \infty$ towards
\begin{align*}
& 2 (\eta/e)^2 \lim_{t_2 \to \infty} 
\Eb^0_{r_{y,\alpha/e}} \otimes \Eb^0_{r_{y,\alpha/e}'} \Bigg[ \sqrt{t_2} e^{-t_2} \left( - X_{t_2} + 2t_2 + \beta' \right) e^{2X_{t_2}} f_{t_2}(X_s,X'_s,s \leq t_2)
\Bigg\vert \Fc_{x,y} \Bigg] \\
& = 2 \lim_{\eps \to 0} \Eb^0_{r_{y,\alpha/e}} \otimes \Eb^0_{r_{y,\alpha/e}'} \Bigg[ \sqrt{|\log \eps|} \eps^2 \left( - \sqrt{X_t^2 + (X_t')^2} + 2t + \beta' \right) e^{2X_t} f_t(X_s,X'_s,s \leq t)
\Bigg\vert \Fc_{x,y} \Bigg]
\end{align*}
recalling that $t = \log(\frac{\eta}{e \eps})$ and since $X'$ will be trapped by zero. We have shown that \eqref{eq:proof_prop_Cauchy_d} converges as $\eps \to 0$ and then $\gamma \to 2$ to the same limit as \eqref{eq:proof_prop_Cauchy_f} as wanted. This concludes the proof of the fact that $(2-\gamma)^{-1} \dhat{m}_\eps^\gamma(A)$ converges in $L^2$ as $\eps \to 0$ and then $\gamma \to 2$ towards $2 \dhat{\mu}(A)$.
\end{proof}

\appendix

\section{Process of Bessel bridges: proof of Lemma \ref{lem:process_bessel_bridge}}\label{sec:appendix_bessel_bridge}

We prove Lemma \ref{lem:process_h} for completeness. It is a direct consequence of the following:

\begin{lemma}\label{lem:process_bessel_bridge}
For all $\delta \in \{e^{-n}, n \geq 0 \}$, let $x \in D \mapsto f_{x,\delta} \in [0,\infty)$ be continuous functions.
By enlarging the probability space we are working on if necessary, we can construct a random field $(h_{x,\delta}, x \in D, \delta \in (0,1])$ that is independent of $(B_t, t \leq \tau)$ and such that
\begin{itemize}
\item
for all $x \in D$, and $n \geq 0$, $(h_{x,e^{-t}}, t \in [n,n+1])$ has the law of a zero-dimensional Bessel bridge from $f_{x,e^{-n}}$ to $f_{x,e^{-n-1}}$;
\item
for all $\delta_0 \in (0,1]$ and $x, y \in D$, $(h_{x,\delta}, \delta \leq \delta_0)$ and $(h_{y,\delta}, \delta \leq \delta_0)$ are independent as soon as $|x-y| \geq 2 \delta_0$;
\item
For all $n \geq 0$, $(h_{x,e^{-t}}, x \in D, t \in [n,n+1])$ and $(h_{x,e^{-t}}, x \in D, t \notin [n,n+1])$ are independent;
\item
for all $n \geq 0$ and $z \in e^{-n-10} \Z^2 \cap D$, $(h_{x,\delta}, x \in D, \floor{e^{n+10} x} = e^{n+10}z, e^{-n-1} \leq \delta \leq e^{-n} )$ is continuous.
\end{itemize}
\end{lemma}

\begin{proof}[Proof of Lemma \ref{lem:process_bessel_bridge}]
We start by explaining how to construct a continuous process $(b^{u,v}_t, u,v \geq 0, 0\leq t \leq 1)$ such that for all $u,v \geq 0$, $(b^{u,v}_t, 0 \leq t \leq 1)$ has the law of a zero-dimensional Bessel bridge from $u$ to $v$.
Let $(b^{1 \to 0, d=0}_t, 0 \leq t \leq 1)$, $(b^{0 \to 1, d=0}_t, 0 \leq t \leq 1)$ and $(b^{0 \to 0, d=4n}_t, 0 \leq t \leq 1)$, $n \geq 1$, be independent Bessel bridges with starting and ending points and dimensions written in superscript. Since 0 is a trap for zero-dimensional Bessel process, $(b^{0 \to 1, d=0}_t, 0 \leq t \leq 1)$ is defined as the time reversal of a zero-dimensional Bessel bridge from 1 to 0. For $w \geq 0$, let $(\alpha_{w,n}, n \geq 1)$ be a sequence of random variables such that for all $n \geq 1$,
\[
\Prob{ \alpha_{w,n} = 1, \forall k \neq n, \alpha_{w,k} = 0 } = \frac{1}{n!} (w/2)^{2n-1} \Gamma(n) I_1(w)
\]
and
\[
\Prob{ \forall k \geq 1, \alpha_{w,k} = 0 } = 1 - \sum_{n \geq 1} \Prob{ \alpha_{w,n} = 1, \forall k \neq n, \alpha_{w,k} = 0 }.
\]
Here $I_1$ is a modified Bessel function of the first kind and $\Gamma$ is the Gamma function.
By using a single uniform random variable on $[0,1]$, it is easy to build all the variables $\alpha_{w,n}, w \geq 0, n \geq 1$ on the same probability space such that they are independent from the Bessel bridges above and such that for all $n \geq 1$, $w \mapsto \alpha_{w,n}$ is continuous. We now define for all $u, v \geq 0$, and $t \in [0,1]$,
\[
b^{u,v}_t = u b^{1 \to 0, d=0}_t + v b^{0 \to 1, d=0}_t + \sum_{n \geq 1} \alpha_{\sqrt{uv}, n} b^{0 \to 0, d=4n}_t.
\]
By construction, $(b^{u,v}_t, u,v \geq 0, 0\leq t \leq 1)$ is a continuous process. Moreover,  by \cite[Theorem (5.8)]{Pitman82}, for all $u, v \geq 0$, $b^{u,v}$ has the law of a zero-dimensional Bessel bridge from $u$ to $v$ over the time interval $[0,1]$ as desired.

We now explain how to construct the process $(h_{x,\delta}, x \in D, \delta \in (0,1])$. For $n \geq 0$ and $x \in D$, define $x_n := e^{-n-10} \floor{ e^{n+10} x } \in e^{-n-10} \Z^2$.
For all $n \geq 0$ and $z \in e^{-n-10} \Z^2 \cap D$, consider independent continuous processes $(h^{n,z}_{x,\delta}, x \in D, x_n = z, e^{-n-1} \leq \delta \leq e^{-n} )$ such that for all $x \in D$ with $x_n = z$, $(h^{n,z}_{x,e^{-t}}, n \leq t \leq n+1 )$ has the law of a zero-dimensional Bessel bridge from $f_{x,e^{-n}}$ to $f_{x,e^{-n-1}}$. This countable collection of independent continuous processes can be constructed thanks to the first step above. We now define for all $x \in D$ and $\delta \in (0,1]$, $h_{x,\delta} = h^{n,x_n}_{x,\delta}$ where $n \geq 0$ is such that $e^{-n-1} < \delta \leq e^{-n}$. By construction, the process $h$ satisfies the desired properties.
\end{proof}

\section{Semi-continuity of subcritical measures: proof of Proposition \ref{prop:process_subcritical}}
\label{sec:app_continuity}

In this section we explain how we obtain Proposition \ref{prop:process_subcritical}. We will only sketch the proof since it follows from \cite{jegoGMC} as well as from arguments having similar flavour as what we already did in this paper.

\begin{proof}
We will first truncate the measure to make it bounded in $L^2$. We will then show that the truncated version is continuous in $\gamma$ by Kolmogorov's continuity theorem and by $L^2$ computations. The statement on the non-truncated measures will then follow.

Let $0 < \gamma_- < \gamma_+ < 2$. We are going to study the regularity of $\gamma \in [ \gamma_-,\gamma_+] \mapsto m^\gamma$. Recall Notation \ref{not:k_x} and the definition of the process $(h_{x,\delta}, x \in D, \delta \in (0,1])$. Fix $\bar{\gamma} \in (\gamma_+,2)$ very close to $\gamma_+$. For $\beta >0$ large, define for all $\eps = e^{-k}$ and $x \in D$ at distance at least $\eps$ from $x_0$, the good event
\[
G_\eps(x) := \left\{ \forall s \in [k_x,k], h_{x,e^{-s}} \leq \bar{\gamma} s + \beta \right\}
\]
and the modified measures
\[
\bar{m}^\gamma_\eps(dx, \beta) = \mathbf{1}_{G_\eps(x)} m^\gamma_\eps(dx).
\]
Since $\bar{\gamma} > \gamma_+$, one can show that this modification does affect the measures in the $L^1$ sense:
\begin{equation}
\label{eq:app_continuity_b}
\lim_{\beta \to \infty} \sup_{\gamma \in [\gamma_-,\gamma_+]} \limsup_{\eps \to 0} \EXPECT{x_0}{ m^\gamma_\eps(D) - \bar{m}^\gamma_\eps(D, \beta) } = 0.
\end{equation}
Moreover, if $\bar{\gamma}$ is close enough to $\gamma_+$, the modified measures are bounded in $L^2$ (consequence of \cite[Proposition 4.2]{jegoGMC}) and we can show with a reasoning similar to what we did in Section \ref{subsec:Cauchy} (this does not follow completely from \cite{jegoGMC} since the good events that we define here are slightly different from the ones considered in \cite{jegoGMC}) that for all Borel set $A$ and all $\gamma \in [\gamma_-, \gamma_+]$, $(\bar{m}^\gamma_\eps(A, \beta), \eps >0)$ is a Cauchy sequence in $L^2$. We will denote $\bar{m}^\gamma(A, \beta)$ the limiting random variable.
We can further show that for all Borel set $A$ and for all $\gamma_1, \gamma_2 \in [\gamma_-,\gamma_+]$,
\begin{equation}
\label{eq:app_continuity}
\limsup_{\eps \to 0} \EXPECT{x_0}{ \left( \bar{m}^{\gamma_1}_\eps(A, \beta) - \bar{m}^{\gamma_2}_\eps(A, \beta) \right)^2 } \leq C (\gamma_1 - \gamma_2)^2
\end{equation}
for some $C>0$ possibly depending on $\beta, \gamma_-, \gamma_+, \bar{\gamma}$.
This follows on the one hand from a reasoning similar to what we have already done to transfer computations from local times to zero-dimensional Bessel process, and on the other hand from the following estimate which is a consequence of \eqref{eq:lem_Bessel0_BM}: for all $K>0$, there exists $C>0$ why may depend on $K, \beta, \gamma_-, \gamma_+, \bar{\gamma}$ such that
\[
\limsup_{t \to \infty} \sup_{\gamma_1, \gamma_2 \in [\gamma_-,\gamma_+]} \sup_{r \in [0,K]} \sqrt{t} \abs{ \Eb^0_r \left[ \left( e^{-\frac{\gamma_1^2}{2} t} e^{\gamma_1 X_t} - e^{-\frac{\gamma_2^2}{2} t} e^{\gamma_2 X_t} \right) \indic{\forall s \leq t, X_s \leq \bar{\gamma} s + \beta} \right] } \leq C (\gamma_1 - \gamma_2).
\]
Let $\mathcal{P} := \{ [a,b) \times [c,d): a,b,c,d \in \Q \}$. $\mathcal{P}$ is a countable pi-system generating the Borel sigma-algebra on $\R^2$.
From \eqref{eq:app_continuity} and Kolmogorov's continuity theorem, we deduce that we can build the variables $\bar{m}^\gamma(A,\beta)$ simultaneously for all $\gamma \in [\gamma_-, \gamma_+]$, $\beta \in \N$ and $A \in \mathcal{P}$ in such a way that for all $\beta \in \N$ and $A \in \mathcal{P}$, $\gamma \in [\gamma_-, \gamma_+] \mapsto \bar{m}^\gamma(A,\beta)$ is continuous. Let $\bar{m}^\gamma(A,\infty)$ be the nondecreasing limit of $(\bar{m}^\gamma(A,\beta), \beta \geq 1)$. A nondecreasing sequence of continuous functions being lower-semicontinous, we have shown that we can build on the same probability space the variables $\bar{m}^\gamma(A,\infty)$, $\gamma \in [\gamma_-, \gamma_+], A \in \mathcal{P}$ such that for all $A \in \mathcal{P}$, $\gamma \in [\gamma_-, \gamma_+] \mapsto \bar{m}^\gamma(A,\infty)$ is lower-semicontinuous. For all $\gamma \in [\gamma_-, \gamma_+]$, $\bar{m}^\gamma$ defines a Borel measure. By \eqref{eq:app_continuity_b}, for all $\gamma \in [\gamma_-, \gamma_+], A \in \mathcal{P}$, $m^\gamma(A) = \bar{m}^\gamma(A,\infty)$ $\prob_{x_0}$-a.s. Concluding the proof of Proposition \ref{prop:process_subcritical} is now routine.
\end{proof}

\paragraph*{Acknowledgement}

I am grateful to Nathana\"{e}l Berestycki for many inspiring discussions and to the referees for their careful readings of the paper.

\bibliographystyle{alpha}
\bibliography{../../bibliography}

\end{document}